\documentclass[11pt]{article}
\usepackage{authblk}
\usepackage[utf8]{inputenc}
\usepackage{amsmath}
\usepackage{amsfonts}
\usepackage{amssymb}
\usepackage{amsthm}
\usepackage{geometry}
\usepackage{enumitem}
\usepackage{comment}
\usepackage{xcolor}
\usepackage{appendix}
\usepackage{mathrsfs}
\usepackage{url}
\usepackage{longtable}

\usepackage{booktabs,tabularx}

\usepackage{tikz}

\usetikzlibrary{arrows.meta,positioning,fit,backgrounds,calc,decorations.pathreplacing}

\usepackage{tikz-cd}

\newif\ifJonesDetection
\JonesDetectionfalse        

\newcommand{\eps}{\varepsilon}
\newcommand{\cQ}{Q}

\newcommand{\cS}{\mathscr{S}}
\newcommand{\cU}{\mathscr{U}}

\newcommand{\cm}{\mathcal{M}}

\newcommand{\cros}{\mathrm{cross}}

\newcommand{\unl}{\TikCircle[1] \, \TikCircle[1]}

\newtheorem{lemma}{Lemma}[section]
\newtheorem{define}[lemma]{Definition} 
\newtheorem{theorem}[lemma]{Theorem}
\newtheorem{proposition}[lemma]{Proposition}
\newtheorem{corollary}[lemma]{Corollary}
\newtheorem{remark}[lemma]{Remark} 
\newtheorem{convention}[lemma]{Convention} 

\newcommand{\inlinefigflip}[4][1.5ex]{%
    \raisebox{#4}{%
        \ifnum#3=1%
            \scalebox{-1}[1]{\includegraphics[height=#1]{#2}}%
        \else%
            \includegraphics[height=#1]{#2}%
        \fi%
    }%
}

\newcommand\TikCircle[1][2.5]{\tikz[baseline=-#1]{\draw[thick](0,0)circle[radius=#1mm];}}

\newcommand{\Symb}[2]{\cS_{#1 \; ; \; #2}}

\title{On trivial Jones--Vassiliev polynomials} 

\author[1]{Avishy Carmi\thanks{avishycarmi@gmail.com }}
\author[1,2]{Eliahu Cohen\thanks{eliahu.cohen@biu.ac.il}}
\affil[1]{Faculty of Engineering and the Institute of Nanotechnology and Advanced Materials, Bar-Ilan University, Ramat Gan 5290002, Israel}
\affil[2]{Institute for Quantum Studies, Chapman University, Orange, California 92866, USA}
\date{}

\begin{document}

\maketitle

\begin{abstract}
We develop a finite-type framework for studying the Jones polynomial and its ability to distinguish the unknot. The main difficulty is to propagate the vanishing of its finite-type coefficient layers from the natural upper degree bound down to the first potentially informative orders. To overcome this, we introduce a local clasped-twist construction whose diagrammatic reductions remain uniformly controlled even when the twist is made arbitrarily long. This separates the role of the twist length from the degree bound and permits a descending vanishing argument. A single construction transfers high-order finite-type vanishing to a long residual twist, where local smoothing relations and an anchor reproducing the original knot force vanishing at all relevant lower orders. Two constructions placed in disjoint regions then generate a controlled two-crossing family. On this family, the resulting low-order relations impose a component count after successive smoothings that contradicts the topology of a suitably chosen pair of crossings. As a consequence, every nontrivial knot has a nonzero Jones finite-type coefficient at an order bounded above by three times its crossing number. In particular, a knot whose Jones polynomial is equal to that of the unknot must itself be the unknot.
\end{abstract}



\section{Introduction}
\addcontentsline{toc}{section}{Introduction}

The question whether the Jones polynomial detects the unknot has
occupied a central place in knot theory since the introduction of the
Jones polynomial.  Bigelow gave an influential discussion of the
problem and of several equivalent or closely related formulations
\cite{Bigelow-JKTR}; see also the recent survey of Manolescu
\cite{Manolescu-survey}.  Extensive computations have supplied strong
evidence for unknot detection.  In particular, Tuzun and Sikora
verified the conjecture for knots through \(24\) crossings
\cite{TuzunSikora-24}.  In the normalization used here, the conjecture
asserts that
\[
V_K(x)=1
\qquad\Longrightarrow\qquad
K\text{ is the unknot}.
\]

A second strand of the subject is the theory of finite-type, or
Vassiliev, invariants, initiated by Vassiliev
\cite{vassiliev1990} and developed through the work of Kontsevich,
Bar-Natan, Goussarov, and many others
\cite{kontsevich1993,bar-natan1995,
ChmutovDuzhinMostovoy2012,goussarov1994}.  Birman and Lin showed that suitable
expansions of quantum knot polynomials produce finite-type invariants
\cite{birman-lin1993}; see also Birman's survey
\cite{birman1993bams}.  The present paper exploits this connection in
the opposite direction: rather than studying the individual
finite-type coefficients of the Jones polynomial in isolation, we
retain the exact Jones skein relation and organize all of its
finite-type layers into a finite algebraic package.

At a conceptual level, the proof combines finite-type rigidity with a local geometric
construction.  We first rewrite the Jones polynomial as a finite sequence of
Jones--Vassiliev coefficient vectors and restrict these coefficients to cubes of crossing
changes.  On a long anti-parallel twist, the Jones skein relation and the local smoothing
geometry force a coefficient layer, once the next two layers vanish, to propagate across
almost the entire cube.  A single vanishing anchor then confines the layer to two
antipodal states, and finite-type degree eliminates these remaining values whenever the
twist is longer than the order under consideration.

The main difficulty is that lengthening an ordinary twist also raises its ambient Jones
degree bound, apparently preventing this descending process from being initialized.  The
clasping twist machine separates these competing effects: its signed states are compressed
above a threshold independent of the twist length, while opening the clasps transfers that
vanishing to the long residual twist.  Two such machines, planted at a suitably chosen pair
of crossings, produce a genuine two-crossing face on which the second Jones--Vassiliev
layer vanishes.  The lowest-order skein recursion then converts this algebraic vanishing
into a prediction for the number of components after one and two oriented smoothings.
That prediction contradicts the elementary topology of the chosen crossing pair, proving
the quantitative triviality barrier and hence unknot detection.

The main theorem also has a consequence for the classical
unknot-detection problem in finite-type theory.  By Goussarov's
characterization of \(m\)-equivalence, a knot which is \(m\)-trivial
for every \(m\) agrees with the unknot under every finite-type
invariant \cite{goussarov1994}. The Jones--Vassiliev coefficients
constructed in this paper form a particular sequence of finite-type
invariants, and the quantitative triviality barrier shows that at
least one of them is nonzero on every nontrivial knot.  Consequently,
a knot which is \(m\)-trivial in the sense of Goussarov for all \(m\)
must be the unknot. Thus, the results establish the unknot-detection
form of the Vassiliev--Goussarov conjecture.

In fact, the full collection of finite-type invariants is not needed
for this conclusion: the finite-type coefficients extracted from the
Jones polynomial already detect the unknot.  It follows that any
universal finite-type invariant over a field of characteristic zero,
in particular the normalized Kontsevich integral
\cite{kontsevich1993}, distinguishes every nontrivial knot from the
unknot.  This is an unknot-detection statement, it does not assert
that finite-type invariants distinguish every pair of non-isotopic
knots.

\subsection{A walkthrough}

The construction begins by converting the Jones polynomial into a finite expansion of finite-type coefficients while retaining the exact skein relation. Concretely, set
\[
p=x-x^{-1},
\qquad\text{equivalently}\qquad
x^2-px-1=0,
\]
and work in the free \(\mathbb Z[p]\)-module with basis \(\{1,x\}\). This gives rise to the Jones–Vassiliev polynomial (JVP),
\[
V_K(x,p)
 =
\sum_{q\geq 0}
\bigl(a_q(K)+b_q(K)x\bigr)p^q.
\]
The coefficient vector
\[
f_q(K):=
\begin{pmatrix}
a_q(K)\\ b_q(K)
\end{pmatrix}
\]
consists of finite-type invariants of order at most \(q\); this is
established in Lemma~\ref{JVP}.  We also use
\[
c_q:=a_q+b_q.
\]
The order-zero coefficient has the particularly useful interpretation
\[
c_0(L)=(-2)^{\ell(L)-1},
\]
where \(\ell(L)\) is the number of components of \(L\)
(Lemma~\ref{JVPfreec}).  Thus, unlike a general finite-type coefficient,
\(c_0\) records a concrete topological quantity.

The JVP is finite rather than formal.  If a diagram has \(n\) crossings
and \(\ell\) components, Corollary~\ref{cor:p-degree-cap} gives
\[
f_q=0
\qquad
\bigl(q>3n+\ell-1\bigr).
\]
For a knot represented by a crossing-minimal diagram this becomes
\[
f_q(K)=0
\qquad
\bigl(q>3c(K)\bigr).
\]
This finite upper bound is one of the two endpoints of the proof.  The
other endpoint is the low-order component-count information carried
by \(c_0\).  The purpose of the intervening machinery is to propagate
vanishing from the upper degree cap down to order \(2\), on a
specially constructed two-crossing variation face.

We say that a knot \(K\) is \(J_N\)-trivial if
\[
f_q(K)=0
\qquad
(1\leq q<N);
\]
see Definition~\ref{def:jm-trivial}.  The main quantitative statement of the paper is
the following.

\medskip
\noindent
\textbf{Main theorem (Theorem~\ref{thm:JN-triviality-barrier}).}
\emph{If \(K\) is a nontrivial \(J_N\)-trivial knot, then}
\[
N\leq 3c(K).
\]

\medskip
\noindent
Equivalently, every nontrivial knot possesses a nonzero
Jones--Vassiliev layer \(f_q\) with
\[
2\leq q\leq 3c(K),
\]
where the lower bound uses the universal vanishing \(f_1=0\) for
knots from Lemma~\ref{lem:f1-vanishing}.  Since \(V_K(x)=1\) is equivalent to the
vanishing of every positive JVP layer, the main theorem immediately
gives the Jones unknot conjecture; this is Theorem~\ref{thm:detection}.

\paragraph{Variation cubes and finite-type propagation.}
Fixing a diagram shadow and varying its crossing signs produces a
Boolean cube.  A finite-type invariant of order at most \(q\) restricts
to a pseudo-Boolean function of degree at most \(q\) on this cube, and
the Vassiliev skein operation becomes a discrete derivative.  The
Jones skein relation is stronger than the bare finite-type relation:
at a given level it couples \(f_q\) to the two preceding JVP layers and
to evaluations on oriented smoothings.  Its coefficient form is given
in Lemma~\ref{JVPskeins}.

The basic local geometric input is \emph{anti-parallel nugatory
invariance}.  In an anti-parallel pair of crossings, smoothing one
crossing renders the other nugatory.  Consequently, after two
successive JVP layers have vanished, the skein relation degenerates
to the graph \(3_1\)-law of Lemma~\ref{lem:31-law}.  On an anti-parallel pair it gives
\[
a_q(++)=a_q(+-)=a_q(-+),
\]
and
\[
b_q(+-)=b_q(-+)=b_q(--).
\]
If the anti-parallel relation graph is connected, these local
equalities propagate through the entire crossing cube.  Lemma~\ref{lem:graph-connectivity}
shows that \(a_q\) is then constant away from the all-negative vertex,
whereas \(b_q\) is constant away from the all-positive vertex.

A single non-extreme vertex at which \(f_q\) vanishes therefore serves
as an \emph{anchor}.  The support of \(a_q\) is reduced to at most the
all-negative vertex, and that of \(b_q\) to at most the all-positive
vertex.  On an \(L\)-dimensional cube, these remaining values are
\(L\)-fold discrete derivatives.  Since \(f_q\) has finite-type degree
at most \(q\), they vanish whenever
\[
L>q.
\]
This is the one-point-support mechanism summarized in
Corollary~\ref{cor:one-output} and specialized to long anti-parallel twists in
Lemma~\ref{lem:twist-supp}.

\paragraph{The initialization problem.}
A long twist supplies the large cube dimension needed to enforce
\(L>q\), but it also increases the ordinary crossing-number degree
cap.  A direct use of a twist therefore encounters a mismatch:
lengthening the twist strengthens the final derivative argument while
simultaneously moving the initial vanishing layer farther upward.
The raw ambient cap grows with \(L\), so it does not initialize the
descent at an \(L\)-independent level.

The principal construction of the paper is designed to separate these
two roles.  A \emph{clasping twist machine}, introduced in
Definition~\ref{def:clasping-twist-machine}, consists of an odd anti-parallel twist block
\[
Z=\{z_1,\ldots,z_L\}
\]
and four clasp coordinates
\[
Y=\{y_0,y'_0,y_{L+1},y'_{L+1}\}.
\]
The full signed machine cube is
\[
Q_{\mathcal M}=Q_Y\times Q_Z.
\]
The machine captures the geometry of a long twist whose ends
clasp both outgoing strands of a seed crossing that has already been smoothed. 
Its geometry is arranged so that almost all signed states collapse,
after local Reidemeister moves, to diagrams with a degree bound
independent of \(L\).  The remaining states are controlled
recursively by the graph \(3_1\)-law or are converted geometrically
into nonexceptional states.

Proposition~\ref{prop:clasping-twist-machines} verifies the complete geometric package.  Its
effective compression threshold, the finite-type level from which descent may proceed, is
\[
\widehat B
 =
B_{\mathrm{host}}+11,
\]
where \(B_{\mathrm{host}}\) is the pre-planting host cap.  The important
point is not the numerical value \(11\), but its independence of the
twist length \(L\).  

Lemma~\ref{lem:locked-state-escape} proves the one-step compression implication
\[
G_{q+2}(Q_{\mathcal M}),
\quad
G_{q+1}(Q_{\mathcal M})
\quad\Longrightarrow\quad
G_q(Q_{\mathcal M})
\qquad(q>\widehat B),
\]
where \(G_q(Q)\) denotes the assertion that \(f_q\) vanishes identically
on \(Q\).  Beginning at any finite ambient degree cap, which may
depend on \(L\), Theorem~\ref{thm:signed-cube-compression} descends through this recurrence and
obtains
\[
G_q(Q_{\mathcal M})
\qquad(q>\widehat B).
\]
The long, \(L\)-dependent ambient cap is therefore used only as a
remote terminal condition.  The machine compresses the actual
initialization level to the fixed threshold \(\widehat B\).

\paragraph{Opening the clasps and descending on the residual twist.}
On the distinguished slice
\[
y'_0=-,
\qquad
y'_{L+1}=+,
\]
apply the Jones skein relation successively at the two unprimed clasp
crossings.  The one-smoothed terms cancel, giving the exact identity
\[
V_{-+}+V_{+-}
 =
x^{-4}V_{++}+x^4V_{--}-p^2V_{00}.
\]
The double-smoothed state \(V_{00}\) is the ordinary residual twist
cube \(Q_\rho\).  Since every signed machine state vanishes above
\(\widehat B\), coefficient extraction from this identity yields
Lemma~\ref{lem:four-clasp-residual-initialization}:
\[
G_r(Q_\rho)
\qquad(r>\widehat B+2).
\]
This initializes two consecutive residual levels immediately above
\(\widehat B+2\).

One canonical residual state \(Z^\circ\) cancels by
Reidemeister--II moves to the original seed crossing and is therefore
isotopic to the host diagram.  It supplies the anchor needed for the
graph-delta argument.  If the host coefficients vanish through
\(\widehat B+2\) and
\[
L>\widehat B+2,
\]
Theorem~\ref{thm:four-clasp-descent} descends from the initialized residual levels to obtain
\[
G_q(Q_\rho)
\qquad(q\geq 2).
\]

\paragraph{Two machines and generated anchors.}
The final obstruction requires a genuine two-crossing face.  Planting
two long machines simultaneously would naively reintroduce the
length-dependent cap, since either machine would regard the other as
part of its ambient diagram.  Theorem~\ref{thm:joint-descent2} avoids this by a
facewise procedure.

First, fix the \(u\)-machine at the residual state reproducing the
original crossing \(u\), and apply the one-machine descent to the
\(v\)-machine.  This generates two vanishing anchor values, one over
each residual sign of the \(v\)-machine.  Next, restrict the
\(v\)-machine to its two-state residual sign face and use those two
values as anchors for descending on the \(u\)-cube.  The result is
\[
G_q(F_u\times F_v)
\qquad(q\geq 2),
\]
where \(F_u\) and \(F_v\) are genuine one-crossing variation faces.
In particular,
\[
G_2(F_u\times F_v).
\]
The locality assertions in Proposition~\ref{prop:clasping-twist-machines} ensure that the two
effective crossings and their oriented smoothings reproduce the
original seed crossings and their smoothings.

\paragraph{The final component-count obstruction.}
Section~\ref{sec:triviality-barriers} converts the order-two vanishing into a contradiction in
elementary planar topology.  In a crossing-minimal diagram of a
nontrivial knot, choose a non-nugatory crossing \(u\).  Lemma~\ref{lem:admissible}
provides an admissible witness \(v\): after smoothing \(u\), the
crossing \(v\) lies between the two resulting components.  Hence
smoothing \(v\) must merge them, by Lemma~\ref{lem:merge}.

Plant the two machines at \(u\) and \(v\).  Theorem~\ref{thm:joint-descent2} gives
\(c_2=0\) on the resulting effective two-crossing face.  Its signed
vertices are knots, so
\[
c_0=1,
\qquad
c_1=0.
\]
The symbolic recursion of Lemma~\ref{lem:BLsymskein} then forces
\[
c_0(\widetilde u=0)=-2
\]
and
\[
c_0(\widetilde u=0,\widetilde v=0)=4.
\]
By Lemma~\ref{JVPfreec}, these values say that the single smoothing has two
components and the double smoothing has three components.  The
smoothing-proxy Lemma~\ref{lem:smooth-proxy} identifies these effective smoothings with
the original diagrams
\[
D\setminus\{u\},
\qquad
D\setminus\{u,v\}.
\]
But admissibility says that the second smoothing merges the two
components of \(D\setminus\{u\}\), so
\(D\setminus\{u,v\}\) must have one component, not three.  This
contradiction proves Theorem~\ref{thm:JN-triviality-barrier}.

\subsection{Proof blueprint}
\addcontentsline{toc}{subsection}{Proof blueprint}

The logical route to the main theorem may be read as follows.

\begin{enumerate}
\item
\textbf{Package the Jones polynomial into finite-type layers.}
Lemma~\ref{JVP} gives the unique finite expansion
\[
V_K(x,p)=\sum_{q\geq0}(a_q+b_qx)p^q,
\]
with \(f_q=(a_q,b_q)^T\) of order at most \(q\).
Definition~\ref{def:jm-trivial} introduces \(J_N\)-triviality.

\item
\textbf{Establish the two numerical endpoints.}
Corollary~\ref{cor:p-degree-cap} gives the upper degree cap
\[
f_q(K)=0\qquad(q>3c(K)),
\]
while Lemma~\ref{JVPfreec} identifies
\[
c_0(L)=(-2)^{\ell(L)-1}.
\]
Lemma~\ref{lem:f1-vanishing} supplies \(c_1=0\) on knots.

\item
\textbf{Convert anti-parallel geometry into cube propagation.}
Lemma~\ref{lem:ap} gives the anti-parallel smoothing identities.
Under two consecutive higher-layer vanishings,
Lemma~\ref{lem:31-law} produces the graph \(3_1\)-law.
Lemma~\ref{lem:graph-connectivity} and Corollary~\ref{cor:one-output} then reduce an anchored layer to
two possible one-point supports.

\item
\textbf{Kill one-point supports by dimension.}
Lemma~\ref{lem:twist-supp} observes that the exceptional values are \(L\)-fold
derivatives.  Since \(f_q\) has order at most \(q\), they vanish
whenever \(L>q\).

\item
\textbf{Compress the length-dependent ambient cap.}
Definition~\ref{def:clasping-twist-machine} introduces the clasping twist machine.
Lemma~\ref{lem:locked-state-escape} proves the locked-state escape step, and
Theorem~\ref{thm:signed-cube-compression} descends from an arbitrary ambient cap to
\[
G_q(Q_{\mathcal M})\qquad(q>\widehat B),
\]
where
\[
\widehat B=B_{\mathrm{host}}+11
\]
is independent of \(L\).
Proposition~\ref{prop:clasping-twist-machines} proves the required local geometric statements.

\item
\textbf{Initialize the residual twist.}
The paired two-clasp skein in Lemma~\ref{lem:four-clasp-residual-initialization} converts signed-cube
compression into
\[
G_r(Q_\rho)\qquad(r>\widehat B+2).
\]

\item
\textbf{Descend from the residual initialization.}
The residual state reproducing the seed crossing supplies an anchor.
Theorem~\ref{thm:four-clasp-descent} combines this anchor, the graph \(3_1\)-law, and
\(L>\widehat B+2\) to obtain
\[
G_q(Q_\rho)\qquad(q\geq2).
\]

\item
\textbf{Generate a two-crossing zero face.}
Theorem~\ref{thm:joint-descent2} first uses one machine to generate two anchors for the
other and then descends facewise, yielding
\[
G_2(F_u\times F_v).
\]

\item
\textbf{Choose crossings with incompatible smoothing behavior.}
Lemmas~\ref{lem:shadow-connectivity}--\ref{lem:merge} produce an admissible pair \((u,v)\): smoothing
\(u\) splits the knot into two components, while smoothing \(v\)
afterward must merge them.  Lemmas~\ref{lem:smooth-slide}--\ref{lem:smooth-proxy} identify the effective
machine smoothings with these original smoothings.

\item
\textbf{Apply the symbolic recursion.}
On \(F_u\times F_v\), the vanishings
\[
c_2=0,\qquad c_1=0,\qquad c_0=1
\]
and Lemma~\ref{lem:BLsymskein} force the double smoothing to have three components.
Admissibility forces it to have one.  This contradiction proves the
triviality barrier of Theorem~\ref{thm:JN-triviality-barrier}.

\item
\textbf{Deduce unknot detection.}
If \(V_K(x)=1\), uniqueness of the JVP expansion gives
\(f_q(K)=0\) for every \(q\geq1\), so \(K\) is \(J_N\)-trivial for
every \(N\).  Taking
\[
N=3c(K)+1
\]
contradicts Theorem~\ref{thm:JN-triviality-barrier} unless \(K\) is the unknot.  This is
Theorem~\ref{thm:detection}.
\end{enumerate}

\subsection{Notation and conventions}

Throughout the text, links and diagrams are oriented, and the Jones polynomial is normalized by
\[
V_U(x)=1
\]
for the unknot \(U\).  

\paragraph{\emph{Shadow}.} A shadow \(S\) is the plane \(4\)-valent graph obtained from a
diagram by forgetting its over--under information.  We write
\[
\operatorname{cross}(S)
\]
for its crossing set, \(\mathcal D(S)\) for the diagrams carried by \(S\), and \(c(K)\) for the
crossing number of a knot \(K\). The notation \(D_1\cong D_2\) means ambient isotopy, relative to the boundary whenever the diagrams are compared inside a specified local disk.

\paragraph{\emph{Cube}.}
The variation cube of a shadow is
\[
Q(S)=\{\pm1\}^{\operatorname{cross}(S)}.
\]
A vertex is denoted by \(\varepsilon\), and
\[
\varepsilon(i=s),\qquad s\in\{+,-,0\},
\]
means that the \(i\)-th crossing is assigned sign \(s\).  The value \(0\) always denotes the
\emph{oriented smoothing}; it is an evaluation on a residual shadow rather than a signed vertex
of \(Q(S)\).  

\paragraph{\emph{Oriented smoothing and residual cubes}.}
For \(J\subseteq\operatorname{cross}(S)\),
\[
D\setminus J
\]
denotes the diagram obtained by oriented smoothing at every crossing in \(J\), and
\(Q_{S\setminus J}\) denotes the corresponding residual variation cube.  By contrast,
\[
|S|\setminus\{u\}
\]
denotes topological deletion of the point \(u\) from the underlying plane shadow.

\paragraph{\emph{JVP finite-type datum}.}
Set
\[
p=x-x^{-1},
\qquad
x^2-px-1=0.
\]
The Jones--Vassiliev expansion is written
\[
V_L(x,p)
 =
\sum_{q\ge0}\bigl(a_q(L)+b_q(L)x\bigr)p^q.
\]
We use
\[
f_q:=
\begin{pmatrix}
a_q\\ b_q
\end{pmatrix},
\qquad
c_q:=a_q+b_q,
\]
and impose the convention
\[
a_q=b_q=c_q=0\qquad(q<0).
\]
For a family \(Q\) of diagrams,
\[
G_q(Q)
\]
means that \(f_q\) vanishes identically on \(Q\).

\paragraph{\emph{Discrete derivatives}.}
Unless stated otherwise, discrete derivatives are unnormalized:
\[
\partial_i f
 =
f(i=+)-f(i=-),
\qquad
\partial_T:=\prod_{i\in T}\partial_i.
\]
Operators supported at distinct crossings commute.  

\paragraph{\emph{Exceptional/antipodal vertices}.}
The all-positive and all-negative vertices of
a signed cube are denoted by
\[
\alpha=(+,\ldots,+),
\qquad
\omega=(-,\ldots,-).
\]

\paragraph{\emph{Clasping twist machines}.}
The machine captures the geometry of a long twist whose ends
clasp both outgoing strands of a seed crossing of a host knot. The seed crossing is smoothed
before the machine is planted. For a clasping twist machine, \(Z=\{z_1,\ldots,z_L\}\) denotes the odd residual twist block, \(Y\) denotes the four clasp coordinates, and
\[
Q_{\mathcal M}=Q_Y\times Q_Z
\]
is the full signed machine cube.  The symbol \(\rho\) denotes the double-smoothed clasp state and
\(Q_\rho\) its residual twist cube.  The states \(Z^+\) and \(Z^-\) reduce to a single positive,
respectively negative, effective crossing; \(Z^\circ\) is the one reproducing the original seed
crossing, and
\[
F_Z=\{Z^+,Z^-\}
\]
is the residual sign face. Finally, \(B_{\mathrm{host}}\) denotes the Jones--Vassiliev degree cap of the pre-planting host shadow.

\paragraph{\emph{Machine compression threshold}.}
The effective compression threshold of the clasping machine is
\[
\widehat B
 =
B_{\mathrm{host}}+\kappa_{\mathrm{mach}}.
\]
In our construction, the machine overhead constant $\kappa_{\mathrm{mach}}=11$.
This threshold marks the initial finite-type layer that vanishes from which the descent proceeds.

\newpage

\section{Preliminaries}
\label{sec:preliminaries}

\subsection{Finite-type invariants}

Fix the category $\mathcal K$ of oriented links in $S^3$, considered
up to ambient isotopy. For $r\geq0$, let $\mathcal K^{(r)}$ denote
the set of immersed oriented singular links with exactly $r$
transverse double points and no other singularities, and set
\[
\mathcal K^{\mathrm{sing}}
 :=
\bigcup_{r\geq0}\mathcal K^{(r)}.
\]
Let $A$ be an abelian group, typically $\mathbb Z$ or $\mathbb Z[p]$.

\begin{define}[Vassiliev extension and order]
A function
\[
v:\mathcal K\longrightarrow A
\]
has a unique extension
\[
\widetilde v:\mathcal K^{\mathrm{sing}}\longrightarrow A
\]
determined recursively by the Vassiliev skein rule
\[
\widetilde v(L_\times)
 =
\widetilde v(L_+)-\widetilde v(L_-)
\]
at every double point.
The invariant $v$ is of finite type of order at most $d$ if
\[
\widetilde v(L)=0
\qquad
\text{for every }L\in\mathcal K^{(d+1)}.
\]
The order, or type, of $v$ is the least such $d$.
\end{define}

\begin{define}[Symbol and weight system]
Let $v$ be of order at most $d$. Its degree-$d$ symbol
\[
\sigma_d(v)
\]
is the restriction of $\widetilde v$ to
$\mathcal K^{(d)}$, regarded as a function on degree-$d$ chord
diagrams. The symbol factors through the $1T$ and $4T$ relations and
therefore defines a linear functional on the degree-$d$ diagram space
$\mathcal A_d$.
\end{define}

\subsection{Pseudo-Boolean functions and finite-type invariants}

\paragraph{Shadows and cubes.}

Fix a shadow $S$ with outside crossing set $I=\cros(S)$, $|I|=n$. The associated
\emph{variation cube} is
\begin{equation}\label{eq:cube}
  \cQ(S) \;=\; \{\varepsilon=(\varepsilon_i)_{i\in I} \mid \varepsilon_i\in\{\pm 1\}\}
  \;\cong\; \{\pm 1\}^n.
\end{equation}
For each $\varepsilon\in \cQ(S)$ let $D(\varepsilon)\in\mathcal D(S)$ denote the diagram carried by $S$ with over/under choices encoded by $\varepsilon$ (see Conventions).

\paragraph{Pseudo‑Boolean functions on a fixed shadow.}
Let $A$ be an abelian group (typically $\mathbb Z$ or $\mathbb Z[p]$).
A \emph{pseudo‑Boolean function on the cube of $S$} is any map
\[
  f_S:\cQ(S)\longrightarrow A,\qquad
  \varepsilon\longmapsto f_S(\varepsilon).
\]
Given an ambient‑isotopy invariant $v:\mathcal K\to A$, we obtain its restriction to the cube of $S$ by evaluation at the vertices:
\[
  f^{\,S}_v(\varepsilon)\;:=\;v\bigl(D(\varepsilon)\bigr),\qquad \varepsilon\in \cQ(S).
\]
Thus, every knot/link invariant supplies, for each choice of shadow, a canonical pseudo‑Boolean function on $\cQ(S)$. Whenever we write $v(D)$ we mean $v([D])$ for the isotopy class represented by the diagram.

\paragraph{Faces and $\cU$‑subcubes.}
For $\cU\subseteq I$ and a base state $\varepsilon \in \cQ(S)$ the \emph{$\cU$‑subcube through $\varepsilon$} is
\[
  \cQ_\cU(\eps)\;:=\;\{\eta\in \cQ(S):\ \eta_i=\eps_i \text{ for all } i\notin \cU\}.
\]
We freely regard alternating sums over such faces as \emph{finite differences} (discrete derivatives). 



\paragraph{Residual cubes and smoothing.}
For later use we smooth a specified set $J\subseteq I$ (oriented “$0$” smoothing from Conventions) and denote by $\cQ_{S\setminus J}$ the variation cube of the resulting shadow; see Definition~\ref{def:residualQ}. Order‑0 constancy on residual cubes is recorded in Corollary~\ref{cor:constancy}.

We extend any cube function $f$ to include evaluations at smoothed crossings by defining $f(\varepsilon_i = 0)$ (alternatively $f(\eps_i^0)$) to mean $v(D(\varepsilon_i = 0))$, i.e.\ evaluation on the smoothed diagram (a vertex of a residual cube).\\[1ex]

The Boolean cube-complex formalism enjoys a number of advantages. The Vassiliev skein relation is realized by a boundary operator, $\partial_i$, contracting/projecting along the $i$-th dimension. The type or order of the Vassiliev invariant becomes the \emph{polynomial degree} of the pseudo-Boolean function. Beyond the observation that this may explain why polynomial invariants are native objects for encoding combinatorial finite-type data, we have tools from spectral analysis at our disposal.

\paragraph{Discrete (finite) derivatives.}
Let $f:\{\pm1\}^{n}\to\mathbb{C}$ be a pseudo-Boolean function.
For coordinate $i$ we define the (unnormalized) derivative
\begin{equation}
        \label{eq:unnormdiff}
        (\partial_{i}f)(\varepsilon) \equiv f(\varepsilon_i^+)-f(\varepsilon_i^-).
\end{equation}
where $\varepsilon_i^\pm$ is $\varepsilon$ with the $i$-th entry set to $\pm 1$. For a subset $T\subseteq[n]$,
\begin{equation}
\partial_{T} = \prod_{i\in T}\partial_{i}, \qquad (\partial_{T}f)(\varepsilon) = \left(\prod_{i \in T} \partial_i\right) f(\eps) .
\end{equation}
Equivalently, it is the alternating sum of $f$ over the $T$-subcube through $\varepsilon$.
%
%

\begin{lemma}[Cube rigidity]
\label{lem:cube-rigidity}
If $f : \{\pm 1\}^k \to A$ satisfies $\partial_i f \equiv 0$ for all coordinates $i$, then $f$ is constant.
\end{lemma}
\begin{proof}
Connect any two vertices by a path of edge flips.
\end{proof}

\begin{lemma}[Local commutation]
Let $i\neq j$ be distinct coordinates of a variation cube, and let
$f$ be any cube function. Then
\[
\partial_i\partial_jf
 =
\partial_j\partial_if,
\]
and
\[
\partial_j\bigl(f(\varepsilon_i=0)\bigr)
 =
(\partial_jf)(\varepsilon_i=0).
\]
More generally, local modifications supported at distinct crossing
coordinates commute whenever both compositions are defined.
\end{lemma}

\begin{proof}
Each operator changes or evaluates only its indicated coordinate.
Since $i\neq j$, the two local operations are independent, and hence
their order is immaterial.
\end{proof}


\paragraph{Mirror involution $\tau$.}
Define the global sign flip (mirror on the cube) by
\[
    \tau(\eta) := -\eta \qquad \text{(i.e.\ } (\tau(\eta))_k = -\eta_k \text{ for all } k \in I\text{)}.
\]
This is an involution $\tau^2 = \mathrm{id}$. The mirror involution satisfies the following properties:
\begin{enumerate}
    \item For any $\eta$ and any $i$, the mirror involution
\begin{equation}
\label{eq:K1}
    \tau(\eta_i^+) = (\tau\eta)_i^-,
    \qquad
    \tau(\eta_i^-) = (\tau\eta)_i^+.
\end{equation}

\item For any $g : \cQ \to A$ and any $i$,
\begin{equation}
    \label{eq:K2}
    \partial_i(g \circ \tau) = -(\partial_i g) \circ \tau.
\end{equation}
Proof:
\[
    \partial_i(g \circ \tau)(\eta)
    = g(\tau(\eta_i^+)) - g(\tau(\eta_i^-))
    = g((\tau\eta)_i^-) - g((\tau\eta)_i^+)
    = -\partial_i g(\tau\eta).
\]
\end{enumerate}
%
Formally, we have extended evaluation maps
\[
    f(\eta(i = 0)),
\]
coming from evaluating the invariant $f$ on the diagram where crossing $i$ is smoothed (See Conventions). Crucially, mirror does not change smoothing evaluation
\begin{equation}
    \label{eq:K3}
    \tau(\eta(i = 0)) = (\tau\eta)(i = 0).
\end{equation}

\paragraph{Polynomial degree equals Vassiliev order.}
A pseudo-Boolean invariant $f$ is said to be \emph{finite type of (Vassiliev) order $d$} if
\begin{equation}
\partial_{T}f\equiv0 \quad\text{for every }T\subseteq[n]\text{ with }|T|>d.
\end{equation}
Intuitively, all $(d+1)$-dimensional faces evaluate to alternating sums of zero. The derivative representation immediately implies
\begin{equation}
\text{Type}(f)=\deg(f).
\end{equation}


The derivative operator not only detects type but produces new finite-type functions of \emph{singular} knots: Fix a face defined by coordinates $T$ and take the alternating sum of $f$ along that face. The result is $\partial_{T}f$ restricted to the complementary coordinates. If $f$ is type $d$, then $\partial_{T}f$ is type $d-|T|$. 

Every pseudo-Boolean function admits a Fourier-Walsh spectral representation. This aspect of the theory is briefly reviewed in the appendix.

\section{Finite-type expansions via ring isomorphism: The JVP}

\label{sec:JVP}
We propose a \emph{finite} packaging of the Jones finite-type layers that is stable under basic operations and easy to compute from tabulated Jones polynomials. The finite-type expansion, referred to as the Jones--Vassiliev polynomial (JVP), is given below. The basic idea is to define the indeterminate $p := x-x^{-1}$ in the standard Jones skein
\begin{equation*}
 x^{-2}V^{(+)}-x^{2}V^{(-)}=pV^{(0)},
\end{equation*}
and then work in the quotient rank-2 module with a basis $\{1,x\}$.

\begin{lemma}[Jones–Vassiliev expansion (JVP)]
\label{JVP}
Let $R=\mathbb{Z}[p]$ and consider the $R$-algebra $R[x]/(x^2-px-1)$. For any oriented link $K$ with Jones polynomial $V_K(x)\in\mathbb{Z}[x,x^{-1}]$, the following hold.
\begin{enumerate}
\item Under the homomorphism $p\mapsto x-x^{-1}$, the class of $V_K(x)$ in $R[x]/(x^2-px-1)$ admits a unique finite expansion
   $$
   V_K(x,p)=\sum_{q \geq 0} \big(a_q(K)+b_q(K)x\big)p^q,\qquad a_q(K),b_q(K)\in\mathbb{Z}.
   $$
   (Equivalently, $R[x]/(x^2-px-1)$ is a free $R$-module with basis $\{1,x\}$.)

\item For each $q\ge 0$, the coefficient functionals $a_q,b_q : \mathcal{K} \to \mathbb{Z}$ are Vassiliev invariants of order $\le q$; write $c_q:=a_q+b_q$ likewise of order $\le q$. 

\end{enumerate}
\end{lemma}


The uniqueness of the JVP expansion is proven in Appendix \ref{Proof_JVP}.

\begin{define}[$J_m$-triviality; JVP coefficients]
\label{def:jm-trivial}
A knot $K$ is $J_m$-trivial if its $q$-type JVP coefficient $a_q(K)=b_q(K) = 0$ for $1 \leq q \leq m-1$.
\end{define}

\begin{define}[JVP-trivial / $p$-constant polynomial]
    \label{def:p-constancy}
    A JVP is said to be $p$-constant if its degree in $p$ is $0$. Equivalently, $a_q=b_q=0$ for all $q \geq 1$,
    \[
    V_0(x,p) = \alpha_0 + \beta_0 x.
    \]
\end{define}

It turns out that the JVP free coefficients (type-$0$ Vassiliev invariants) encode the number of link components.

\begin{lemma}
    \label{JVPfreec}
    Let $K$ be a link with $\ell_K$ components. Then, $c_0(K) \equiv a_0(K) + b_0(K) = (-2)^{\ell_K - 1}$, and
    \begin{equation}
    a_0(K) = \left \{ 
    \begin{array}{ll}
        (-2)^{\ell_K - 1}, & \ell_K \equiv 1 \mod 2 \\
        0, & \ell_K \equiv 0 \mod 2
    \end{array} \right. \; \;, \quad
    b_0(K) = \left \{
    \begin{array}{ll}
        0, & \ell_K \equiv 1 \mod 2 \\
        (-2)^{\ell_K - 1}, & \ell_K \equiv 0 \mod 2
    \end{array} \right.
    \end{equation}
\end{lemma}

\begin{remark}[Unknot normalization]
For knots, Lemma~\ref{JVPfreec} forces $a_0 = 1$, $b_0 = 0$. Thus, $p$-constant means $V_K(x,p) = 1$ in our normalization.
\end{remark}

\begin{lemma}
\label{lem:f1-vanishing}
For every knot $K$,
\[
a_1(K)=0, \quad b_1(K)=0 .
\]
\end{lemma}
 
\begin{proof}
By Lemma~\ref{JVP}, both $a_1$ and $b_1$ are finite-type invariants of order $\le 1$. On knots, the degree-one chord-diagram space vanishes by the $1T$ relation, so every order-${\le}\,1$ knot invariant is constant. Since the unknot has JVP equal to $1$, both $a_1$ and $b_1$ vanish on the unknot. Hence both vanish on every knot.
\end{proof}

\subsection{Degree bound}

An upper bound on the $p$-degree of $V_L(x,p)$ is linear in the number of crossings and link components.

\begin{corollary}[Crossing-number cap for the $p$-adic JVP]
\label{cor:p-degree-cap}
Let $D$ be an oriented diagram of a link $L$ with $n$ crossings and $\ell$ components. Then the JVP coefficients
\[
    a_q(L) = b_q(L) = 0 \qquad \text{for all } q > 3n + \ell - 1.
\]
Equivalently,
\[
    \deg_p V_L(x,p) \le 3n + \ell - 1.
\]
In particular, for knots ($\ell = 1$), $\deg_p V_K(x,p) \le 3n$.
\end{corollary}

\begin{proof}
Write an arbitrary element of the JVP ring as
\[
F=P(p)+Q(p)x,
\qquad
P(p),Q(p)\in\mathbb Z[p].
\]
Using $x^2=px+1$ we thus have
\[
xF
 =
Q+(P+pQ)x.
\]
Since \(x^{-1}=x-p\), one also has
\[
x^{-1}F
 =
(Q-pP)+Px.
\]
Thus multiplication by either \(x\) or \(x^{-1}\) raises the
\(p\)-degree by at most one. It follows inductively that
\[
\deg_p(x^m)\le |m|
\qquad(m\in\mathbb Z).
\]
Consequently, any Laurent polynomial in \(x\) supported in the
exponent interval \([-M,M]\) has JVP \(p\)-degree at most \(M\).

We now estimate the Laurent support of the Jones polynomial. To avoid
confusion with the coefficient polynomials \(P(p)\) and \(Q(p)\), let
\(\mathsf A\) denote the Kauffman bracket variable. A bracket state
\(\sigma\) contributes~\cite{Kauffman1987}
\[
\mathsf A^{a(\sigma)-b(\sigma)}
\bigl(-\mathsf A^2-\mathsf A^{-2}\bigr)^{|\sigma|-1},
\]
where
\[
a(\sigma)+b(\sigma)=n
\]
and \(|\sigma|\) is the number of state circles. Since
\[
|a(\sigma)-b(\sigma)|\le n
\]
and each smoothing changes the number of components by at most one,
\[
|\sigma|\le n+\ell.
\]
Hence every monomial appearing in a bracket-state contribution has
\(\mathsf A\)-exponent of absolute value at most
\[
n+2(|\sigma|-1)
 \le
3n+2\ell-2.
\]
The writhe normalization multiplies the bracket by
\[
(-\mathsf A^3)^{-w(D)}.
\]
Since
\[
|w(D)|\le n,
\]
the absolute value of every \(\mathsf A\)-exponent in the normalized
Jones polynomial is at most
\[
3n+2\ell-2+3n
 =
6n+2\ell-2.
\]
In the present convention, the Jones variable is
\[
x=\mathsf A^{-2}.
\]
Thus, an \(x\)-exponent \(m\) corresponds to the
\(\mathsf A\)-exponent \(-2m\). Therefore the Jones polynomial is
supported in \(x\)-exponents satisfying
\[
|m|
 \le
\frac{6n+2\ell-2}{2}
 =
3n+\ell-1.
\]
By the first part of the proof,
\[
\deg_p V_L(x,p)\le 3n+\ell-1.
\]
Equivalently,
\[
a_q(L)=b_q(L)=0
\qquad
(q>3n+\ell-1).
\]
For a knot, \(\ell=1\), and hence
\[
\deg_p V_K(x,p)\le 3n.
\]
\end{proof}

\begin{remark}[Asymptotic sharpness of the crossing-number cap]
The coefficient \(3\) in Corollary~\ref{cor:p-degree-cap} is asymptotically sharp,
although the additive constant is not claimed to be optimal.
Indeed, the relations
\[
x^2=px+1,
\qquad
x^{-1}=x-p
\]
give, by induction,
\[
\deg_p(x^j)=
\begin{cases}
j-1, & j\geq 1,\\
-j,  & j\leq 0.
\end{cases}
\]
Let \(T_m=T(2,2m+1)\) denote the handedness of the
\((2,2m+1)\)-torus knot whose Jones polynomial is
\[
V_{T_m}(x)
 =
x^{2m}+x^{2m+4}-x^{2m+6}
+x^{2m+8}-\cdots+x^{6m}-x^{6m+2}.
\]
The term of largest exponent occurs uniquely, and hence
\[
\deg_p V_{T_m}=6m+1.
\]
Since
\[
c(T_m)=2m+1,
\]
this may be written as
\[
\deg_p V_{T_m}=3c(T_m)-2.
\]
For the mirror knot,
\[
V_{\overline{T_m}}(x)=V_{T_m}(x^{-1}),
\]
and the unique extreme monomial is \(x^{-(6m+2)}\). Therefore
\[
\deg_p V_{\overline{T_m}}
 =
6m+2
 =
3c(T_m)-1.
\]
Consequently, no universal estimate of the form
\[
\deg_p V_K\leq \gamma\,c(K)+O(1)
\]
can hold with \(\gamma<3\).
For \(m=1\), this gives the trefoil and its mirror:
\[
V_{T_1}(x)=x^2+x^6-x^8,
\qquad
V_{\overline{T_1}}(x)=x^{-2}+x^{-6}-x^{-8},
\]
with
\[
\deg_p V_{T_1}=7,
\qquad
\deg_p V_{\overline{T_1}}=8,
\]
while Corollary~\ref{cor:p-degree-cap} gives the common upper bound \(9\).
Indeed, direct evaluation gives the JVP expansion
\[
V_{T_1}(x,p)= -p^7x -p^6 -5p^5x -4p^4 -6p^3x -3p^2 + 1
\]
for the trefoil, and
\[
V_{\overline{T_1}}(x,p) = -p^8 + p^7x -6p^6 + 5p^5x -10p^4 +6p^3x -3p^2 +1
\]
for its mirror.
This also illustrates that JVP \(p\)-degree need not be preserved
under mirroring, even though the same crossing-number cap applies to
a knot and its mirror.
\end{remark}

\medskip


\begin{lemma}[Equivalence of mirror and primal caps]
\label{lem:mirror-cap}
Let \(\mathcal H\) be a family of diagrams carried by a fixed oriented shadow
\(S\), with \(n\) crossings and \(\ell\) components. Set
\[
        B=3n+\ell-1.
\]
Then every diagram in \(\mathcal H\), and every diagram in the mirror family
\(\tau\mathcal H\), has JVP \(p\)-degree at most \(B\).
In particular, if \(S\) is the shadow of a crossing-minimal knot diagram
\(D\) of \(K\), then
\[
        B=3c(K)
\]
is a uniform JVP degree cap for every crossing-flip of \(D\) and for every
mirror of such a crossing-flip.
\end{lemma}
\begin{proof}
The mirror of a diagram is carried by the same shadow and has the same number
of crossings. The component count is also unchanged on a fixed shadow. Hence
Corollary~\ref{cor:p-degree-cap} gives the same bound for the mirror family. In the knot case
\(\ell=1\) and \(n=c(K)\), so the bound is \(B=3c(K)\).
\end{proof}

\begin{remark}[$\deg_p V_L$ is a link invariant]
\label{rem:cap-invariance}
    Because $V_L$ is a link invariant its $p$-adic degree is similarly a link invariant.
    For a knot $K$ with crossing number $c(K)$, Corollary~\ref{cor:p-degree-cap} implies
    \[
    \deg_p V_K(x,p) = \mathcal{O}(c(K)).
    \]
\end{remark}

\subsection{Skein relations}

\begin{lemma}[Finite-type skein relations]
    \label{JVPskeins}
    For a knot/link dependent function, $K \mapsto f(K)$, let $f(\eps_i^-)$, $f(\eps_i^0)$, and $f(\eps_i^+)$ be the function evaluated at a knot obtained from $K$ by allowing the $i$-th crossing to vary as over, under, and oriented smoothing, respectively. The Jones--Vassiliev skein relation induces the following levels of recursive dependencies:
    \begin{equation}
    \label{eq:skein}
        \begin{array}{ll}
             \text{Polynomial level:} &  \partial_i V = p\left[x V(\eps_i^-) + V(\eps_i^0) + (x-p) V(\eps_i^+) \right],\\[1.5ex]
            \text{Finite-type level:} & 
            \partial_i \begin{pmatrix}a_q \\[1ex] b_q \end{pmatrix} = \begin{pmatrix} b_{q-1}(\eps_i^-) + a_{q-1}(\eps_i^0) + b_{q-1}(\eps_i^+) - a_{q-2}(\eps_i^+) \\[1ex] a_{q-1}(\eps_i^-) + b_{q-1}(\eps_i^0) + a_{q-1}(\eps_i^+) + b_{q-2}(\eps_i^-) \end{pmatrix}, \\[4ex]
        \end{array}
    \end{equation}
    where the \emph{unnormalized} partial derivative, $\partial_i V \equiv V(\eps_i^+)-V(\eps_i^-)$.
\end{lemma}

 Let $f_r = [a_r,b_r]^T$. For every $r \geq 0$ and every index set $S \ni x$, the JVP skein above may be written at level $r$
\begin{equation}
\label{eq:theskein}
\partial_x
\binom{\partial_{S\setminus\{x\}}a_r}{\partial_{S\setminus\{x\}}b_r}
=
\partial_{S\setminus\{x\}}
\binom{
b_{r-1}(x{=}{-})+a_{r-1}(x{=}0)+b_{r-1}(x{=}{+})-a_{r-2}(x{=}+)
}{
a_{r-1}(x{=}{-})+b_{r-1}(x{=}0)+a_{r-1}(x{=}{+})+b_{r-2}(x{=}-)
}.
    \end{equation}
This is the form of the skein used throughout.

\subsection{Mirror transformation}

The JVP of a mirror diagram may be obtained by swapping 
$$x \mapsto x^{-1} := x-p, \qquad p \mapsto -p.$$
Equivalently, one may swap $(+) \leftrightarrow (-)$ through the global sign flipping $\tau$.
The finite-type coefficients of the mirror JVP are related to the primal one as follows.

\begin{lemma}[Mirror transformation]
    \label{lem:JVP-mirror}
    Let $\eta$ be a cube state, and let $\bar{D} := D(\tau(\eta))$ be a mirror diagram of $D := D(\eta)$. Write its JVP
    \[
    V_{\bar{D}}(x,p) = \sum_{q \geq 0} (a_q(-\eta) + b_q(-\eta) x) p^q.
    \]
    Then
    \[
    a_q(-\eta) = (-1)^q(a_q(\eta) + b_{q-1}(\eta)), \quad b_q(-\eta) = (-1)^q b_q(\eta),
    \]
    where $f_q(\eta) := (a_q(\eta), b_q(\eta))^T$ are the coefficients of $V_D(x,p)$.
\end{lemma}

\subsection{Symbolic calculus}


\begin{convention}[Smoothing notation \(D\setminus J\)]
\label{conv:smoothing}
Let \(S\) be a fixed shadow with outside crossing set \(\cros(S)\), and let
\(D\in\mathcal D(S)\) be a diagram carried by \(S\).
For a subset \(J\subseteq \cros(S)\), we write $D\setminus J$ for the diagram obtained from \(D\) by performing the \emph{oriented smoothing} at every
crossing in \(J\) (our “\(0\)” smoothing). Thus:
\begin{itemize}
  \item \(\,D\setminus\emptyset = D\), and for a singleton \(\{i\}\) we have \(D\setminus\{i\}=D(\eps_i=0)\).
  \item Smoothing at distinct crossings is local and order–independent, so
        \(D\setminus J\) is well defined (independent of any ordering of \(J\)).
  \item If \(J\subseteq \cros(S)\) and \(S\setminus J\) denotes the shadow with those
        vertices smoothed, then the \emph{residual cube} \(\cQ_{S\setminus J}\) is the
        variation cube on the remaining crossings; functions of the form
        \(D\mapsto f(D\setminus J)\) are naturally viewed on \(\cQ_{S\setminus J}\)
        (see Definition~\ref{def:residualQ} and Corollary~\ref{cor:constancy}).
  \item We occasionally write \(J_r\subseteq J\) to mean an arbitrarily chosen \(r\)-subset
        of \(J\); by order–independence, statements about \(D\setminus J_r\) do not
        depend on the choice of \(J_r\) (used, e.g., when discussing incremental smoothings
        and the last–row fingerprint).
  \item In single‑crossing skeins we retain \(D(+), D(-), D(0)\); the set‑notation
        agrees with this: \(D\setminus\{i\}=D(\eps_i=0)\).
\end{itemize}
\end{convention}

\begin{define}[Residual cubes]
    \label{def:residualQ}
    For a shadow $S$ and $J \subset \cros(S)$, the residual cube $\cQ_{S \setminus J}$ is the variation cube of over/under choices on the remaining crossings after smoothing all $J$.
\end{define}

\begin{corollary}[Order-0 constancy on residual cubes] 
\label{cor:constancy}
Fix any shadow $S$ of $K$. Because $c_0$ is a type-0 (order-0) Vassiliev invariant, for any fixed subset $J \subset \cros(S)$ of crossings the map $D \mapsto c_0(D \setminus J)$ is constant on the residual cube of over/under choices. Here $\cros(S)$ denotes the vertex set of the shadow (the common set of crossings for all $D \in \mathcal{D}(S)$). 
\end{corollary}
\begin{proof} 
Since $c_0$ is type-0, $\partial_T c_0 \equiv 0$ for all nonempty $T$, i.e., it is constant on every face of every residual cube $\cQ_{S \setminus J}$.
\end{proof}


\begin{remark}[Finite-type Jones coefficients $c_q$]
    Define 
    $$c_q \equiv a_q + b_q,$$ 
    where $a_q$ and $b_q$ are the JVP coefficients. The symbolic calculus in this section works with either the Birman-Lin or the JVP's $c_q$. In contrast with the BL expansion, the JVP has only finitely many non-vanishing coefficients (explicitly, $a_q \equiv 0, \, b_q \equiv 0$ for $q > 3n+\ell-1$, as shown by Corollary~\ref{cor:p-degree-cap}).
    The significance of this detail becomes clear in the ensuing.
\end{remark}



\begin{define}[Set-indexed symbols]
\label{def:symbols}
For any $K\subset I$ and $J\subset I\setminus K$, define
\[
\Symb{K}{J} := \partial_J c_{|J|}^K
\]
By $c^K_{|J|}$ we mean the coefficient function $c_{|J|}$ evaluated on the diagram $D \setminus K$. To obtain $\Symb{K}{J}$ first perform oriented smoothings at all crossings in $K$, then take the unnormalized alternating sum over the $J$-face (order $|J|$).
%
\end{define}

\begin{lemma}[JVP symbol-skein recursion]
\label{lem:BLsymskein}
Fix a crossing $i$. For every $n\ge 1$,
\begin{equation}
\label{eq:BLsymskein}
\Symb{\varnothing}{J}
\;=\; 2\, \Symb{\varnothing}{J\setminus\{i\}}
\;+\; \Symb{\{i\}}{J\setminus\{i\}}
\qquad (i\in J,\ |J|=n).
\end{equation}
and more generally, 
\begin{equation}
\label{eq:BLsymskein1}
\Symb{K}{J}=2\,\Symb{K}{J\setminus\{i\}} + \Symb{K\cup\{i\}}{J\setminus\{i\}}.
\end{equation}
\end{lemma}

\subsubsection{Order-0 constancy: $c_0$ is a shadow/cube invariant}

The free coefficient $c_0$ is a 0-order finite-type invariant, which in turn makes it an invariant of variation cubes or shadows rather than of any single vertex (knot/link). In addition, we have seen it uniquely encodes the link component count. These two facts render it a powerful tool for analyzing the behavior of shadows under oriented smoothing.

\begin{corollary}[Shadow invariance of component counts]
\label{cor:components}
Let $S$ be a shadow and $D(S)$ the set of all diagrams carried by $S$. Then:
\begin{enumerate}
\item (\textbf{No smoothing}) The number of link components $\ell(D)$ is independent of $D \in D(S)$.
\item (\textbf{Fixed smoothings}) More generally, for any $J \subseteq \cros(S)$ (crossing set of $S$), the component count of $D \setminus J$ (link obtained from $D$ by smoothing exactly the crossings in $J$) is independent of $D \in D(S)$.
\end{enumerate}
\end{corollary}

\begin{proof}
By order-0 constancy (Corollary~\ref{cor:constancy}), $D \mapsto c_0(D \setminus J)$ is constant on the residual cube determined by $J$. By Lemma~\ref{JVPfreec}, $c_0(D \setminus J) = (-2)^{\ell(D \setminus J) - 1}$, so $\ell(D \setminus J)$ is constant as $D$ varies. Taking $J = \varnothing$ yields (1).
\end{proof}


\section{Anti-parallel nugatory graphs}

Consider an R2 pair of crossings $(i,j)$. There are two distinctive ways to arrange the directions of the two interlacing strands. The anti-parallel configuration shown below exhibits a distinctive smoothing behavior as this ends up with a nugatory crossing. This behavior is independent of crossing signs and hence shared by any long twist whose interlacing strands are anti-parallel. The rule
\[
\text{\emph{Smoothing one crossing makes the other nugatory}}
\]
is referred to here as anti-parallel nugatory invariance. For an anti-parallel pair, smoothing one of the crossings gives:
\[
    \includegraphics[width=0.2\textwidth]{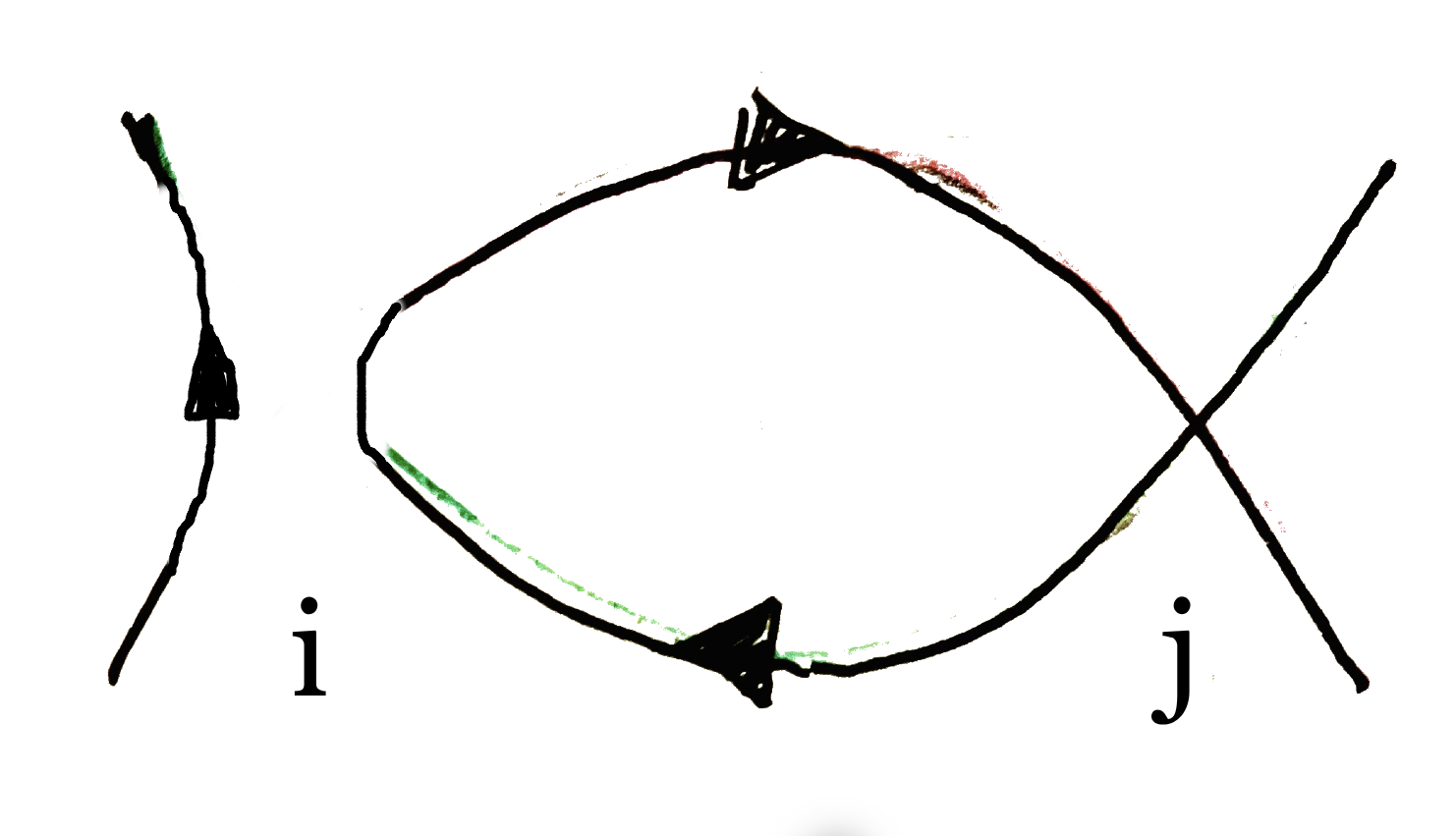}
    \quad \raisebox{0.8cm}{$\stackrel{i=0}{\Longleftarrow}$} \quad
    \includegraphics[width=0.2\textwidth]{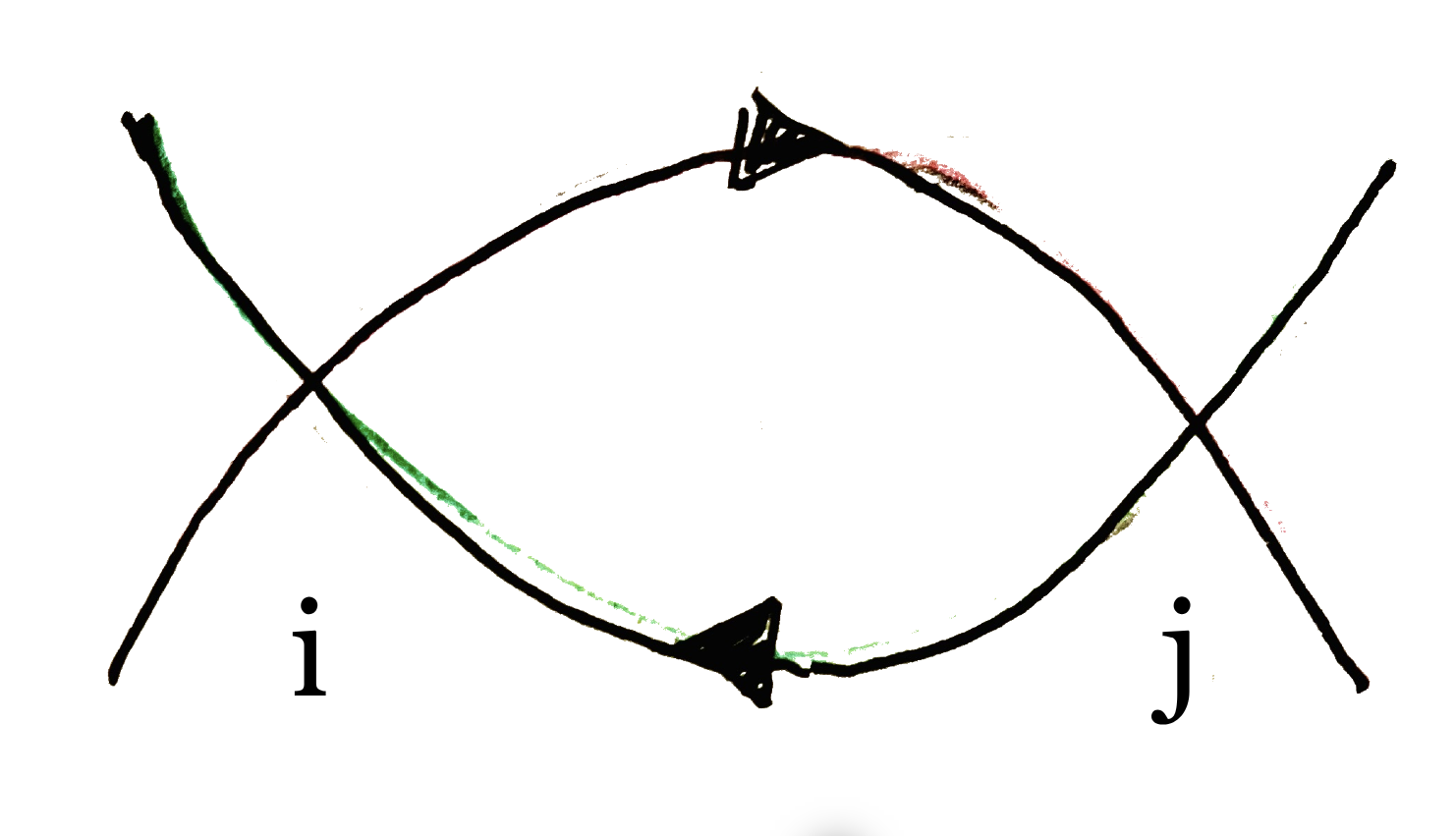} \quad \raisebox{0.8cm}{$\stackrel{j=0}{\Longrightarrow}$} \quad
    \includegraphics[width=0.2\textwidth]{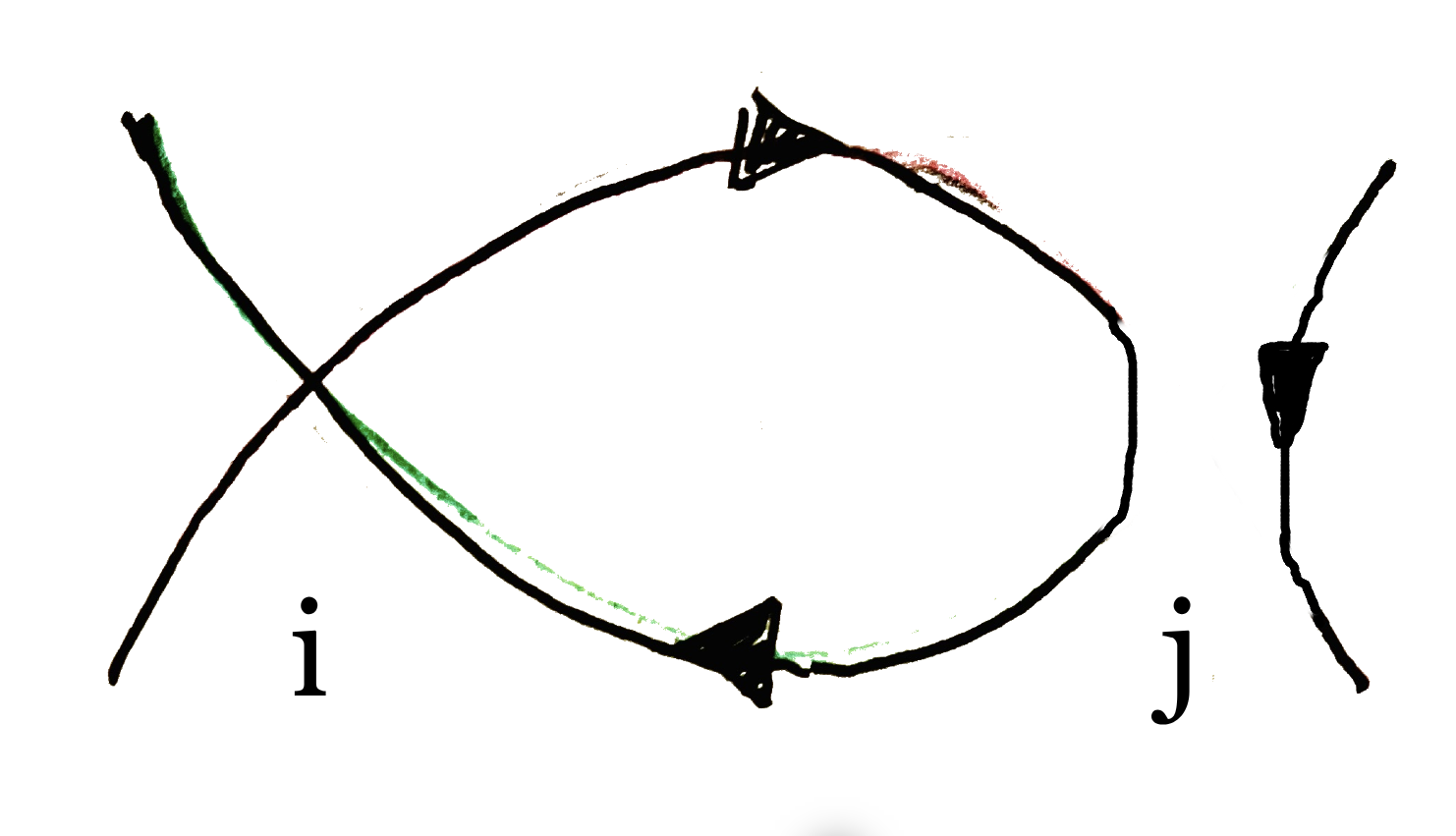}
\]
hence, by Reidemeister-I equivalence
\[
D(i = 0, \ \varepsilon_j = \pm) \cong D(\eps_i = \mp, \ j = 0).
\]
More generally, in a long twist anti-parallel nugatory invariance is respected pairwise.
In terms of shadow diagrams
\[
\includegraphics[width=0.4\textwidth]{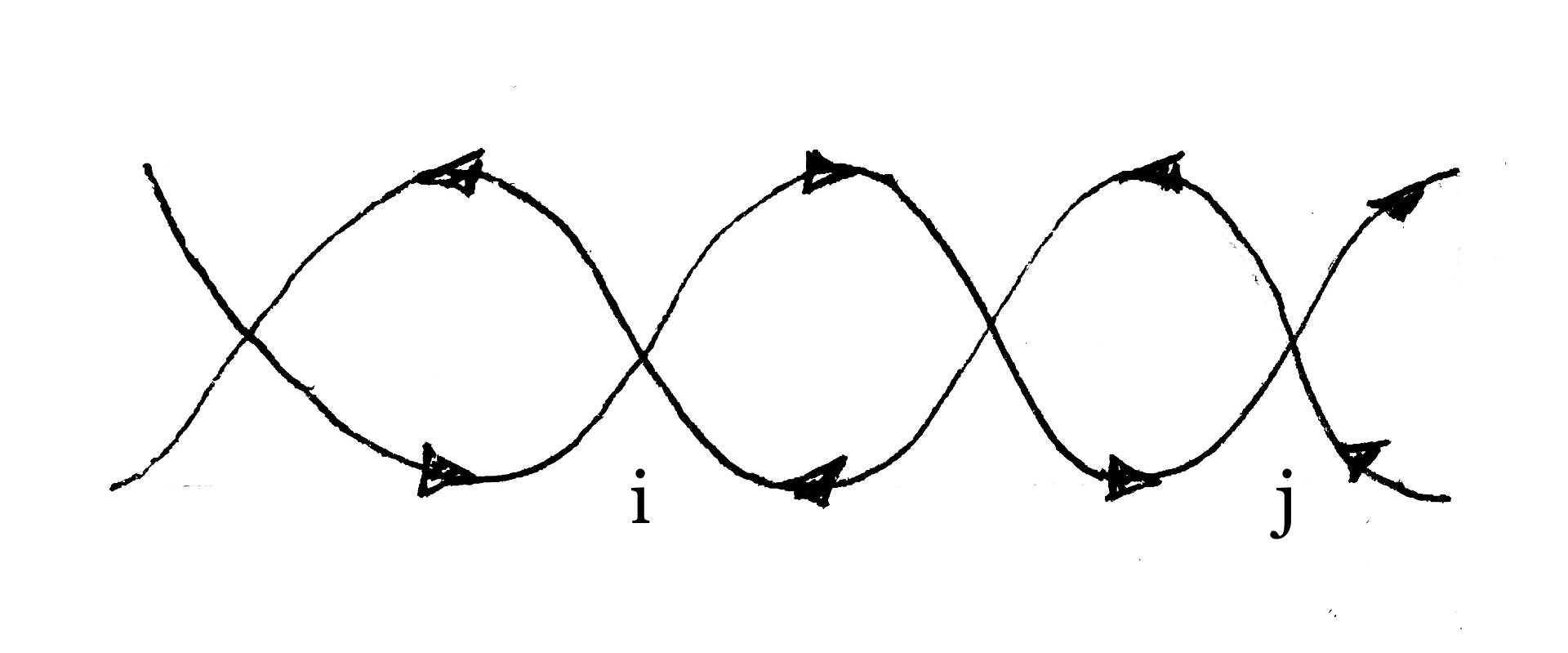}
\quad \raisebox{1cm}{$\stackrel{j=0}{\Longrightarrow}$} \quad
\includegraphics[width=0.4\textwidth]{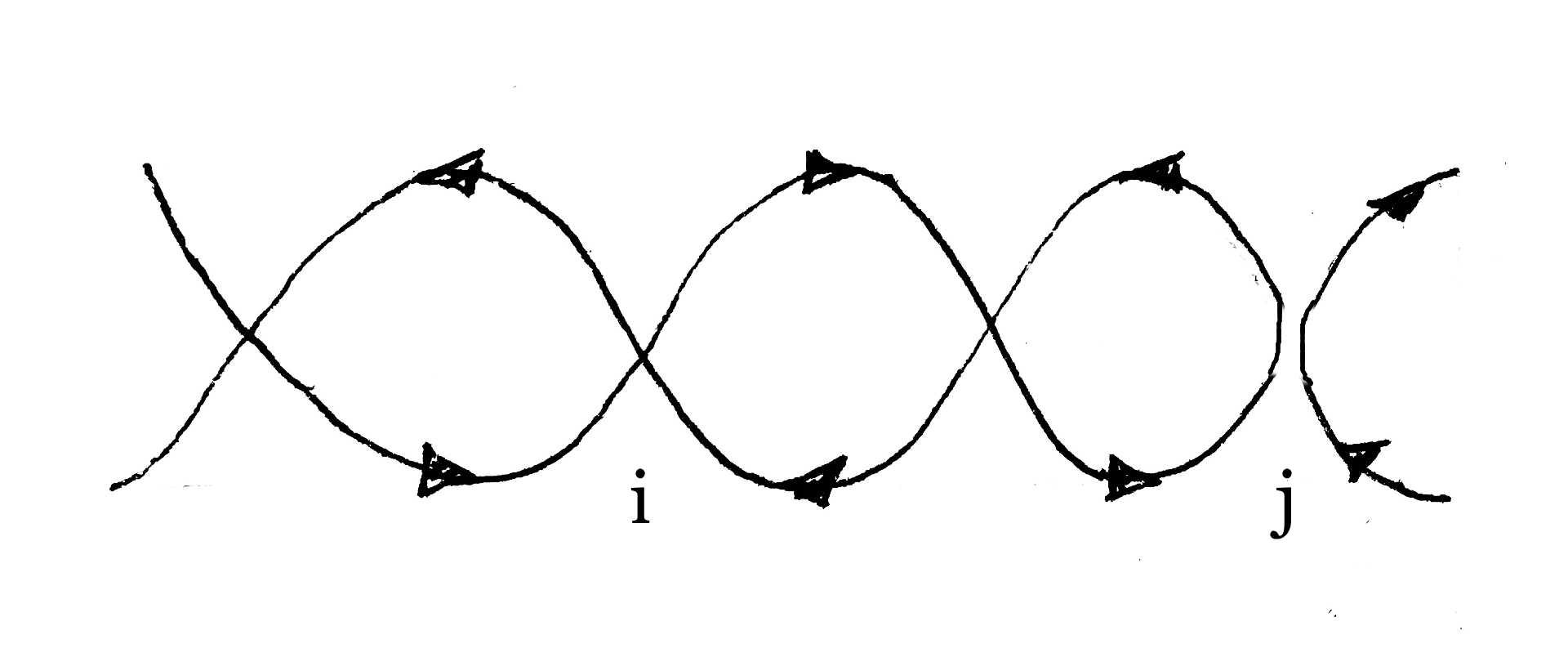}
\]




\begin{lemma}[Anti-parallel nugatory invariance]
\label{lem:ap}
Let $(i,j)$ be an anti-parallel pair in a diagram shadow, and let $v$
be any ambient-isotopy invariant.  Then for every choice~$\eta$ of the other
crossing states,
\[
  v\bigl(D(\eta; \ i{=}0, \ \eps_j{=}+)\bigr)
    = v\bigl(D(\eta;\ i{=}0,\ \eps_j{=}-)\bigr),
\]
hence
\[
  \partial_j(v\circ D)\bigl(\eta(i{=}0)\bigr) = 0.
\]
The same holds with $i$ and $j$ interchanged.  Since mirror preserves
smoothing, $\tau(\eta(i{=}0))=(\tau\eta)(i{=}0)$, the same identity holds on
the mirror cube. Let $f_q=(a_q,b_q)$ be the JVP coefficients. Then, in particular, for every $q\ge 0$,
\[
\partial_i f_q(j=0) = 0, \qquad \partial_j f_q(i=0)=0.
\]

\end{lemma}
\begin{proof}
Work inside the anti-parallel twist disk and fix the exterior of the diagram. After smoothing $i$, the remaining crossing $j$ lies in a monogon. A Reidemeister~I move supported inside the disk removes that monogon, so the diagrams
\[
D(\eta; i=0,\eps_j=+) \quad \text{and} \quad D(\eta; i=0,\eps_j=-)
\]
are isotopic relative to the boundary of the disk. Hence, $v$ takes the same value on both, which proves
\[
\partial_j(v\circ D)(\eta(i=0))=0.
\]
The proof with $i$ and $j$ interchanged is identical. Since mirror commutes with oriented smoothing, the same identity holds on $\tau(\cQ)$.
\end{proof}

The significance of the anti-parallel invariance rule becomes apparent when working at particular levels of the finite-type ladder. Let $f_q$ be the JVP order-$q$ coefficients. Whenever 
\[
f_{q+2} \equiv 0, \qquad  f_{q+1} \equiv 0,
\]
anti-parallel nugatory invariance enforces the JVP skein to exhibit a distinctive rigid behavior where
order-$q$ coefficients propagate through adjacent vertices in the cube. When crossings are pairwise anti-parallel, this connectivity pattern becomes as strong as to exclude exactly two exceptional vertices.


\subsection{Graph-delta machines}


\begin{define}[Finite-type datum $G_q(Q)$]
Let \(f_q=(a_q,b_q)\) denote the JVP coefficient vector. For a variation cube \(Q\), write
\[
G_q(Q)
\]
for the assertion
\[
f_q\equiv0\quad\text{on }Q.
\]
\end{define}


\begin{define}[Anti-parallel relation graph]
Let
\[
V=\{z_1,\ldots,z_M\}
\]
be a set of crossing coordinates in a variation cube \(Q_V\cong\{\pm1\}^M\). An
anti-parallel relation graph on \(V\) is a graph
\[
\Gamma=(V,E)
\]
such that for every edge \(\{z_i,z_j\}\in E\), the pair \((z_i,z_j)\) satisfies the
anti-parallel smoothing identities
\[
\partial_{z_i} f_r(z_j=0)=0,
\qquad
\partial_{z_j} f_r(z_i=0)=0
\]
for all relevant levels \(r\).
\end{define}

\begin{define}[Graph-delta machine]
A graph-delta machine is a set of crossing coordinates \(V\) equipped with a connected
anti-parallel relation graph
\[
\Gamma=(V,E).
\]
The all-positive and all-negative vertices of \(Q_V\) are denoted
\[
\alpha=(+,\ldots,+),
\qquad
\omega=(-,\ldots,-).
\]
\end{define}

\begin{lemma}[Graph \(3_1\)-law]
\label{lem:31-law}   
Let \(\Gamma=(V,E)\) be an anti-parallel relation graph. Assume
\[
G_{q+2}(Q_V),
\qquad
G_{q+1}(Q_V).
\]
Then for every edge \(\{z_i,z_j\}\in E\),
\[
\partial_{z_i}a_q(z_j=+)=0,
\qquad
\partial_{z_i}b_q(z_j=-)=0,
\]
and similarly with \(i,j\) interchanged. Equivalently, for every edge $\{z_i,z_j\} \in E$
%
\[
a_q(++) =
\]
\[
a_q(+-) \quad = \quad  a_q(-+)
\]
\vspace{-2pt}
\[
b_q(+-) \quad = \quad  b_q(-+)
\]
\[
= b_q(--)
\]
\end{lemma}

\begin{proof}
We use the JVP skein relation \eqref{eq:theskein}
\begin{equation}
\label{eq:master-skein}
\partial_x
\binom{\partial_{S\setminus\{x\}}a_r}{\partial_{S\setminus\{x\}}b_r}
=
\partial_{S\setminus\{x\}}
\binom{
b_{r-1}(x{=}{-})+a_{r-1}(x{=}0)+b_{r-1}(x{=}{+})-a_{r-2}(x{=}+)
}{
a_{r-1}(x{=}{-})+b_{r-1}(x{=}0)+a_{r-1}(x{=}{+})+b_{r-2}(x{=}-)
}.
\end{equation}
for $x\in S$. Apply the \(a\)-component of the JVP skein relation at level \(r=q+2\) in the coordinate \(z_j\),
then take \(\partial_{z_i}\):
\[
\partial_{z_i}\partial_{z_j}a_{q+2}
=
\partial_{z_i}b_{q+1}(z_j=-)
+
\partial_{z_i}a_{q+1}(z_j=0)
+
\partial_{z_i}b_{q+1}(z_j=+)
-
\partial_{z_i}a_q(z_j=+).
\]
The left-hand side vanishes by \(G_{q+2}\). The two endpoint terms at level $q+1$ vanish by \(G_{q+1}\). The
smoothing term vanishes by the anti-parallel identity
\[
\partial_{z_i}a_{q+1}(z_j=0)=0.
\]
Thus
\[
\partial_{z_i}a_q(z_j=+)=0.
\]
The \(b\)-component is identical:
\[
\partial_{z_i}\partial_{z_j}b_{q+2}
=
\partial_{z_i}a_{q+1}(z_j=-)
+
\partial_{z_i}b_{q+1}(z_j=0)
+
\partial_{z_i}a_{q+1}(z_j=+)
+
\partial_{z_i}b_q(z_j=-).
\]
Again all terms except the last vanish, giving
\[
\partial_{z_i}b_q(z_j=-)=0.
\]
\end{proof}

\begin{lemma}[Connected graph support propagation]
\label{lem:graph-connectivity}
Let \(\Gamma=(V,E)\) be connected. Assume the relations
\[
\partial_{z_i}a(z_j=+)=0
\]
hold for every oriented edge \(z_i\leftrightarrow z_j\) of \(\Gamma\). Then \(a\) is constant on the non-exceptional branch
\[
Q_V\setminus\{\omega\}.
\]
Similarly, if
\[
\partial_{z_i}b(z_j=-)=0
\]
holds for every oriented edge, then \(b\) is constant on
\[
Q_V\setminus\{\alpha\}.
\]
\end{lemma}

\begin{proof}
We prove the \(a\)-statement. Let \(\epsilon\neq\omega\). Then at least one coordinate is \(+\).
Let
\[
S=\{z\in V:\epsilon_z=+\}
\]
be the current plus set. If \(S\neq V\), connectedness of \(\Gamma\) gives an edge from \(S\) to
\(V\setminus S\). Along such an edge, keep the plus endpoint fixed as the witness \(z_j=+\), and
flip the negative endpoint \(z_i\) using
\[
\partial_{z_i}a(z_j=+)=0.
\]
This increases the plus set by one without changing the value of \(a\). Iterating, \(\epsilon\)
is connected to \(\alpha\). Hence \(a\) is constant on \(Q_V\setminus\{\omega\}\).
The \(b\)-statement is the sign-reversed argument, using a minus coordinate as the witness and
connecting every vertex other than \(\alpha\) to \(\omega\).
\end{proof}

\begin{remark}[Anchoring condition]
Lemma~\ref{lem:graph-connectivity} dictates the propagation of coefficients across the connected non-exceptional branch of $\Gamma$. If the coefficients of one of the vertices in this branch vanish, then the support shrinks accordingly. We call such a vertex an anchor.
\end{remark}

\begin{corollary}[One-output graph-delta support]
\label{cor:one-output}
Let \(Q_V\) be a connected graph-delta machine of dimension \(M=|V|\). Assume
\[
G_{q+2}(Q_V),
\qquad
G_{q+1}(Q_V).
\]
Assume also the anchor vanishing
\[
f_q(v^\circ)=0, \qquad v^\circ \in Q_V \setminus\{\alpha,\omega\}.
\]
Then
\[
\operatorname{Supp}(a_q)\subseteq\{\omega\}, \qquad \operatorname{Supp}(b_q)\subseteq\{\alpha\}.
\]
\end{corollary}

\begin{proof}
By the graph \(3_1\)-law and connected graph support propagation, Lemma~\ref{lem:31-law} and Lemma~\ref{lem:graph-connectivity}, \(a_q\) is constant on
\[
Q_V\setminus\{\omega\}
\]
and $b_q$ is constant on
\[
Q_V\setminus\{\alpha\}
\]
Since \(v^\circ \in Q_V\setminus\{\alpha,\omega\}\) and \(a_q(v^\circ)=0\), it follows that
\[
a_q=0
\]
on \(Q_V\setminus\{\omega\}\). Similarly, since \(b_q(v^\circ)=0\), it follows that
\[
b_q=0
\]
on \(Q_V\setminus\{\alpha\}\).
\end{proof}



\section{Twist-delta machines}

We specialize the preceding formalism to a class of anti-parallel relation graphs, long twists respecting pairwise anti-parallel conditions. 

\begin{define}[Twist-delta cube]
\label{def:twist-delta}
Let $Q_Z=\{\pm1\}^{Z}$ be a variation cube whose coordinates are crossings in one long anti-parallel twist region. We assume that for every pair of distinct active coordinates
\[
z_i,z_j\in Z,
\]
the anti-parallel smoothing identities hold:
\[
\partial_{z_i}f_r(z_j=0)=0,
\qquad
\partial_{z_j}f_r(z_i=0)=0
\]
for every relevant level \(r\).
We write
\[
\omega_Q=(-,\ldots,-),
\qquad
\alpha_Q=(+,\ldots,+)
\]
for the exceptional all-negative and all-positive vertices of \(Q_Z\). In what follows we always assume an odd twist length, $|Z| \equiv 1 \mod 2$.
\end{define}

The twist-delta cube underlies the crossing variations of a $|Z|$-long anti-parallel twist.
Consider a host knot $K$. Pick one of $K$'s crossings and delineate it by a small disk. Within the disk, plant a twist after cutting out the base crossing as shown below. 
\vspace{-0.2cm}
\[
\includegraphics[width=0.3\textwidth]{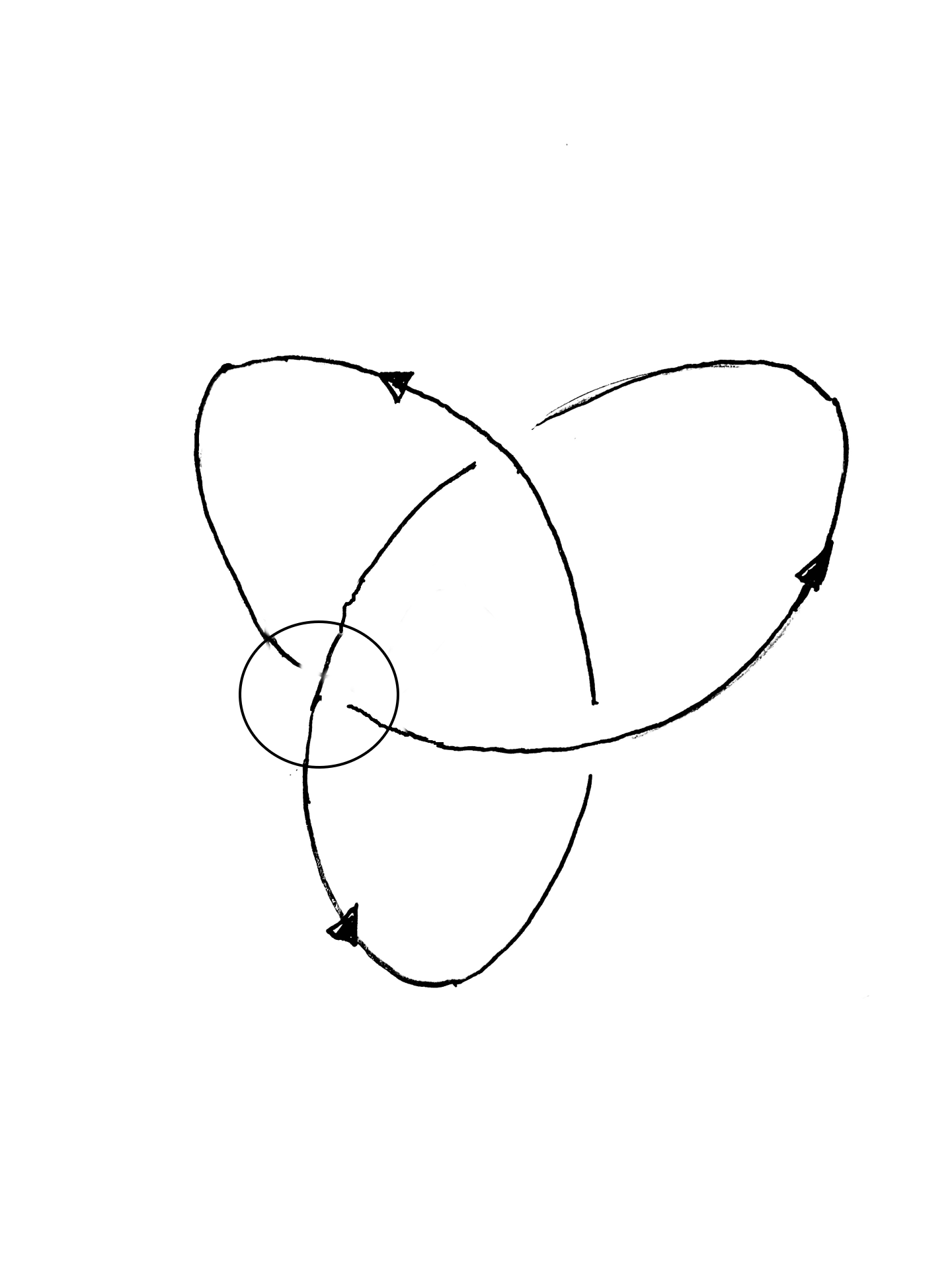}
\; \raisebox{3cm}{$\Longrightarrow$} \;
\raisebox{1cm}{
\includegraphics[width=0.3\textwidth]{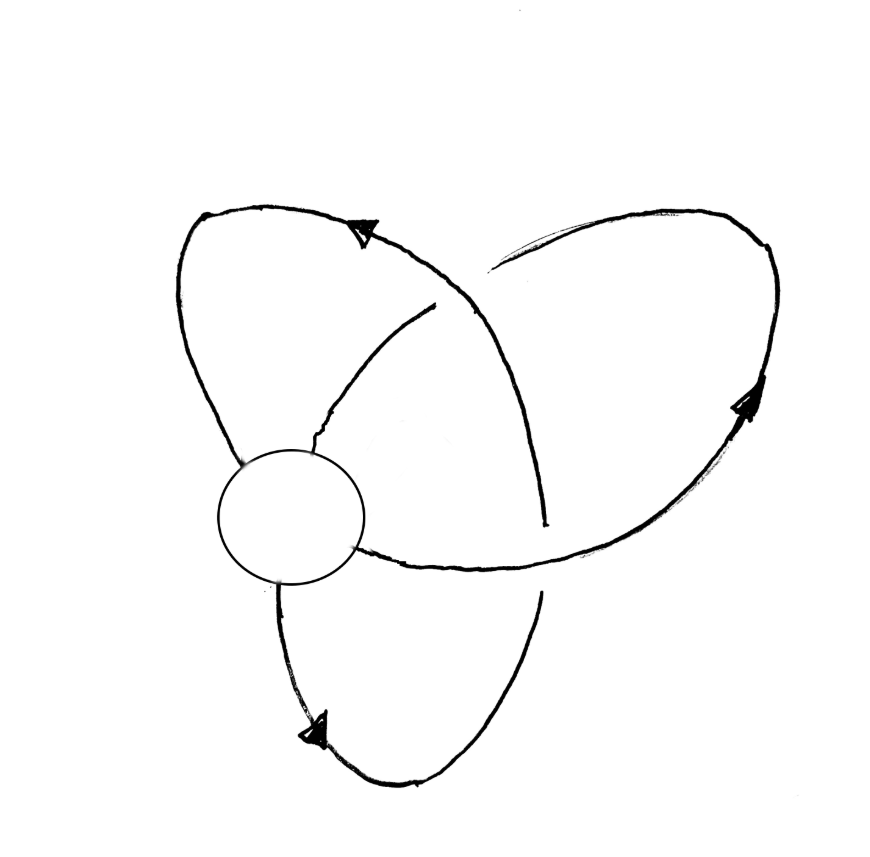}
}
\; \raisebox{3cm}{$\Longrightarrow$} \;
\includegraphics[width=0.27\textwidth]{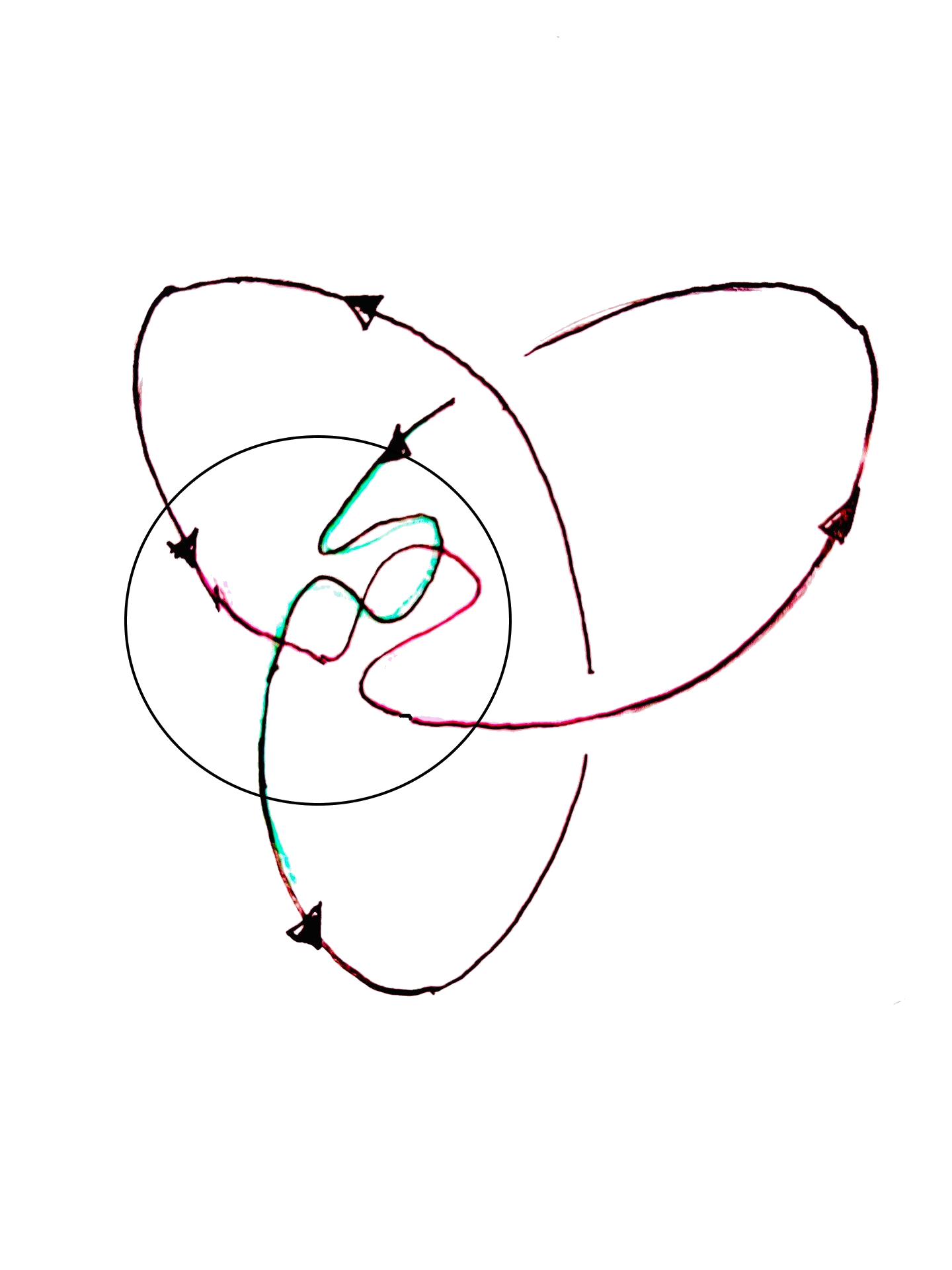}
\]
Once planted, the twist becomes pairwise anti-parallel in the sense of Definition~\ref{def:twist-delta}. 
The next result shows that, provided with the right anchoring, the twist induces one-point supports at finite-type levels below any two layer vanishing.

\begin{lemma}[Twist \(\delta\)-support with anchor]
\label{lem:twist-supp}
Let \(Q_Z=\{\pm1\}^Z\) be a twist-delta cube of odd length
\[
        L=|Z|>1.
\]
Assume
\[
        G_{q+2}(Q_Z), \qquad G_{q+1}(Q_Z).
\]
Assume moreover that there exists a non-exceptional anchor state
\[
        Z^\circ\in Q_Z\setminus\{\alpha,\omega\}
\]
such that
\[
        f_q(Z^\circ)=0.
\]
Then
\[
        a_q(Z)=0 \quad (Z\neq\omega), \qquad
        b_q(Z)=0 \quad (Z\neq\alpha).
\]
Moreover, the exceptional antipodes satisfy
\[
        a_q(\omega)=(-1)^L\partial_Z a_q, \qquad
        b_q(\alpha)=\partial_Z b_q.
\]
Consequently, if \(L>q\), then
\[
        a_q\equiv 0, \qquad b_q\equiv 0
\]
on the whole twist cube.
In particular, when the twist is planted at a crossing of a host diagram and
one residual sign state \(Z^\circ\) is isotopic to the host diagram, the anchor
condition follows from \(f_q(\text{host})=0\).
\end{lemma}
\begin{proof}
By Definition \ref{def:twist-delta}, the twist coordinates form a connected anti-parallel
relation graph. Since \(G_{q+2}\) and \(G_{q+1}\) hold, Corollary~\ref{cor:one-output} applies.
The anchor \(Z^\circ\) lies in both non-exceptional branches, so
\[
        a_q(Z)=0 \quad (Z\neq\omega), \qquad
        b_q(Z)=0 \quad (Z\neq\alpha).
\]
Only the two antipodal vertices may remain. Since \(a_q\) is supported, if at
all, only at \(\omega\),
\[
        \partial_Z a_q
        =
        \sum_{\varepsilon\in Q_Z}
        (-1)^{\#\{\varepsilon_i=-\}}a_q(\varepsilon)
        =
        (-1)^L a_q(\omega).
\]
Thus
\[
        a_q(\omega)=(-1)^L\partial_Z a_q.
\]
Similarly, since \(b_q\) is supported, if at all, only at \(\alpha\),
\[
        \partial_Z b_q=b_q(\alpha).
\]
Finally, if \(L>q\), then \(\partial_Z a_q=\partial_Z b_q=0\) by finite-type
degree. Hence both exceptional values vanish as well. This proves the claim.
\end{proof}

\begin{corollary}[Delta support below the effective degree cap]
    Let $K$ be the host knot. Plant a $L$-long twist at one of $K$'s crossings. Assume $L \geq 3$ and set $B = 3(c(K)-1)+3L$.
    Then
    \[
    a_B(Z) = 0 \quad (Z \neq \omega), \qquad b_B(Z)=0 \quad (Z \neq \alpha).
    \]
\end{corollary}
\begin{proof}
    By Corollary~\ref{cor:p-degree-cap} the effective degree cap after planting the twist is at most $B=3(c(K)-1)+3L$. Thus
    \[
    f_{B+1} \equiv 0, \quad f_{B+2} \equiv 0.
    \]
    Therefore, $G_{B+1}$ and $G_{B+2}$ hold on the twist cube. In addition, since $B > 3c(K)$ and $3c(K)$ is the base degree cap of $K$
    \[
    f_B(K) = 0,
    \]
    thereby providing the anchoring condition.
    The corollary now follows by applying Lemma~\ref{lem:twist-supp} with $q=B$.
\end{proof}

\subsection{Clasping twist machines}

\newif\ifoldClaspingMachine
\oldClaspingMachinefalse

\begin{define}[Clasping twist machine]
\label{def:clasping-twist-machine}
A clasping twist machine is a pair
\[
        (Z,Y),
\]
where
\[
        Z=\{z_1,\ldots,z_L\}, \qquad Q_Z=\{\pm1\}^Z, \qquad L \equiv 1 \mod 2, \qquad L \geq 3,
\]
is the residual twist block, and
\[
Y=\{y_0,y_0',y_{L+1},y_{L+1}'\},
\qquad
Q_Y=\{\pm1\}^{Y},
\]
is the clasps block. Here, $(y_0,y_0')$ are the two crossings of the left clasp and
$(y_{L+1},y_{L+1}')$ are the two crossings of the right clasp.
The full signed machine cube is
\[
Q_{\mathcal M}:=Q_Y\times Q_Z.
\]
All coordinates in \(Z\) are assumed to be pairwise
anti-parallel. The left clasp pairs $(y_0,y_0')$ are anti-parallel, and similarly
the right clasp pairs $(y_{L+1},y_{L+1}')$.
The machine captures the geometry of a long twist whose ends
clasp both outgoing strands of a base crossing. The base crossing is smoothed
before the machine is planted. 
\end{define}



\subsubsection{Host cap and effective machine threshold}

Let $\mathcal H$ be an exterior host family carried by a fixed
oriented shadow. Let
\[
B_{\mathrm{host}}(\mathcal H)
\]
denote the JVP $p$-degree cap supplied by Corollary~\ref{cor:p-degree-cap} for that
host shadow. Here $B_{\mathrm{host}}(\mathcal H)$ refers to the
pre-planting host shadow, with the seed crossing still counted.
The clasping construction has an effective finite-type overhead
\[
\kappa_{\mathrm{mach}}
\]
which is independent of the host knot, the residual length $L$, and
the exterior state. We define the effective machine threshold by
\[
\widehat B(\mathcal H)
 :=
B_{\mathrm{host}}(\mathcal H)+\kappa_{\mathrm{mach}}.
\]
For the clasping twist machine constructed below, $\kappa_{\mathrm{mach}}=11$.
Thus
\[
\widehat B(\mathcal H)
 =
B_{\mathrm{host}}(\mathcal H)+11.
\]
The number $\widehat B$ is an effective compression threshold. This threshold marks the initial finite-type layer that vanishes from which the descent proceeds.
When $\mathcal H$ is the crossing-flip family of a crossing-minimal
diagram of a knot $K$, write
\[
B_0:=3c(K),
\qquad
\widehat B:=B_0+\kappa_{\mathrm{mach}}.
\]
For the present machine,
\[
\widehat B=3c(K)+11.
\]

\subsubsection{Machine states and their faces}
\label{sec:machine-states}
The machine behavior in terms of its degree cap is classified by the twist state $Z$ and the clasp states. We write a clasp state in the form
\[
[s_0s_0';s_1s_1']
 =
(y_0=s_0,\,
  y_0'=s_0',\,
  y_{L+1}=s_1,\,
  y_{L+1}'=s_1').
\]
Define the four locked clasp states
\[
\cm_{+-}:=[++;--],
\qquad
\cm_{-+}:=[--;++],
\qquad
\cm_{++}:=[++;++],
\qquad
\cm_{--}:=[--;--],
\]
and their corresponding $Z$-faces
\[
Q_{\cm_\varepsilon}
 :=
\{\cm_\varepsilon\}\times Q_Z,
\qquad
\varepsilon\in\{++,--,+-,-+\}.
\]
The controlled signed locus is the face union
\[
U_{\mathrm{ctrl}}
 :=
\left(\left(
Q_Y\setminus\{\cm_{++},\cm_{--},\cm_{+-},\cm_{-+}\}
\right)\times Q_Z\right) \; \cup \; \left(Q_Y \times (Q_Z \setminus \{\alpha_Z,\omega_Z\})\right).
\]


\subsubsection{Machine hypotheses}
\label{sec:clasping-hypo}
The following hypotheses record the behavior of a clasping twist machine. The first three are inherent to the machine and characterize its behavior at finite-type layers above its compression threshold; 
these hypotheses are later used to establish the initial finite-type vanishing above this threshold. The fourth hypothesis recovers the ordinary twist, one of whose states is an anchor and by that furnishes finite-type descending below the compression threshold. The actual geometry of the machine is provided in Proposition~\ref{prop:clasping-twist-machines}.

For a fixed exterior state $E$, write
\[
Q_{\mathcal M}^{E}
 :=
\{E\}\times Q_Y\times Q_Z.
\]
For an exterior host family $\mathcal H$, define
\[
Q_{\mathcal M}(\mathcal H)
 :=
\mathcal H\times Q_Y\times Q_Z
 =
\bigcup_{E\in\mathcal H}Q_{\mathcal M}^E,
\]
and define
\[
Q_\rho(\mathcal H)
 :=
\mathcal H\times Q_\rho.
\]
Thus,
\[
G_q\bigl(Q_{\mathcal M}(\mathcal H)\bigr)
\]
means that
\[
f_q(E;Y;Z)=0
\]
for every $E\in\mathcal H$, $Y\in Q_Y$, and $Z\in Q_Z$.
Similarly,
\[
G_q\bigl(Q_\rho(\mathcal H)\bigr)
\]
means uniform vanishing on the residual cube for every exterior state
in $\mathcal H$. All machine hypotheses below are understood uniformly over the stated
exterior host family.

\medskip
\noindent
\textbf{(C1) Classification and controlled collapse.}
For odd $L$, the complement of the controlled locus consists of the
six algebraically escapable states
\[
\mathcal M_{+-}\times\{\alpha_Z,\omega_Z\},
\qquad
\mathcal M_{-+}\times\{\alpha_Z,\omega_Z\},
\qquad
\mathcal M_{++}\times\{\omega_Z\},
\qquad
\mathcal M_{--}\times\{\alpha_Z\},
\tag{C1a}
\]
and the two exceptional states
\[
\mathcal M_{++}\times\{\alpha_Z\},
\qquad
\mathcal M_{--}\times\{\omega_Z\}.
\tag{C1b}
\]
For every exterior state $E$, every
\[
\eta\in
\bigl(
Q_Y\setminus
\{\mathcal M_{++},\mathcal M_{--},
  \mathcal M_{+-},\mathcal M_{-+}\}
\bigr)\times Q_Z,
\]
and every $r>\widehat B$,
\[
f_r(E;\eta)=0.
\tag{C1c}
\]
Moreover, let $r>\widehat B$ and assume
\[
G_{r+1}(Q_{\mathcal M}^{E}),
\qquad
G_{r+2}(Q_{\mathcal M}^{E}).
\]
Then
\[
f_r(E;Y;Z)=0
\]
for every $Y\in Q_Y$ and every
\[
Z\in Q_Z\setminus\{\alpha_Z,\omega_Z\}.
\tag{C1d}
\]

\medskip
\noindent
\textbf{(C2) Exceptional collapse.}
Let $r>\widehat B$ and assume
\[
G_{r+1}(Q_{\mathcal M}^{E}),
\qquad
G_{r+2}(Q_{\mathcal M}^{E}).
\]
Then
\[
f_r(E;\mathcal M_{++};\alpha_Z)=0,
\qquad
f_r(E;\mathcal M_{--};\omega_Z)=0.
\tag{C2}
\]

\medskip
\noindent
\textbf{(C3) Anti-parallel clasp pairs.}
Each of the pairs
\[
\{y_0,y_0'\},
\qquad
\{y_{L+1},y_{L+1}'\}
\]
is an anti-parallel relation pair in the sense of Section~3.
Equivalently, for every invariant level $r$ and with all other
coordinates fixed,
\[
\partial_{y_0}f_r(y_0'=0)
 =
\partial_{y_0'}f_r(y_0=0)
 =
0,
\]
and
\[
\partial_{y_{L+1}}f_r(y_{L+1}'=0)
 =
\partial_{y_{L+1}'}f_r(y_{L+1}=0)
 =
0.                                    \tag{C3}
\]

\medskip
\noindent
\textbf{(C4) Original residual slice.}
On the original primed-sign slice
\[
y_0'=-,
\qquad
y_{L+1}'=+,
\]
double smoothing of the two unprimed clasp coordinates gives the
ordinary residual twist cube:
\[
\rho := (y_0=0,\,
  y_0'=-,\,
  y_{L+1}=0,\,
  y_{L+1}'=+), \qquad
D(\rho;\,Z)
\simeq
D(Z)                              \tag{C4}
\]
relative to the exterior of the planting disk.
The residual face 
\[
Q_\rho=\{\rho\} \times Q_Z
\]
is the ordinary anti-parallel twist cube planted at a base crossing. Thus, under $G_{q+2},G_{q+1}$, the usual $3_1$-law of Lemma~\ref{lem:31-law} and graph connectivity of Lemma~\ref{lem:graph-connectivity} apply to $Q_\rho$:
\[
a_q\text{ is constant on }Z\neq\omega, \qquad b_q\text{ is constant on }Z\neq\alpha.
\]
No assertion is made about a diagram obtained by smoothing exactly one
of the four clasp coordinates.

\subsection{Compressing finite-type layers above the threshold}

The controlled signed faces contribute a fixed degree cap overhead. The only machine states that may increase the degree cap uncontrollably are those recorded in (C1a) and (C1b). Among them, the states
\[
\cm_{++} \times \{\alpha_Z\}, \qquad \cm_{--} \times \{\omega_Z\},
\]
are truly exceptional. These states are resolved through geometric arguments in Proposition~\ref{prop:clasping-twist-machines}(G2).

\medskip

\begin{lemma}[Locked-state escape]
\label{lem:locked-state-escape}
Fix an exterior state $E\in\mathcal H$. Assume
\textnormal{(C1)--(C3)}. Let 
\[
q>\widehat B(\mathcal H),
\]
and suppose
\[
G_{q+2}(Q_{\mathcal M}^E),
\qquad
G_{q+1}(Q_{\mathcal M}^E).
\]
Then
\[
G_q(Q_{\mathcal M}^E).
\]
\end{lemma}


\begin{proof}
Fix $E\in\mathcal H$ and suppress it from the notation. All uses of
\textnormal{(C1d)} and \textnormal{(C2)} below are therefore made on
the fixed cube $Q_{\mathcal M}^E$.

Every state in $U_{\mathrm{ctrl}}$ vanishes at level $q$ by
\textnormal{(C1c)} and \textnormal{(C1d)}, since $q>\widehat B$, and $G_{q+1}$ and $G_{q+2}$ hold. The exceptional states vanish similarly by \textnormal{(C2)}. It remains to treat the six locked states in (C1a).

\paragraph{Step 1: mixed clasp states.}
Consider first
\[
\cm_{+-}\times\{Z\}, \qquad \cm_{-+}\times \{Z\}.
\]
Fix $Z$ and suppress all exterior coordinates.
Begin with
\[
\cm_{+-} = [++;--].
\]
Apply the graph $3_1$-law Lemma~\ref{lem:31-law} to the left anti-parallel pair
$(y_0,y_0')$, with the right clasp pair and $Z$ fixed.  Since the left
pair is in state $++$,
\begin{equation}
\label{eq:lck41}
\partial_{y_0} a_q(y_0'=+)=0 \quad \Longrightarrow \quad 
a_q([++;--];Z)
 =
a_q([-+;--];Z).
\end{equation}
The state $[-+;--]$ is controlled because its left clasp pair has
opposite signs. By \textnormal{(C1)}
\[
a_q([-+;--];Z)=0,
\]
and therefore
\begin{equation}
\label{eq:lck42}
a_q(\cm_{+-};Z)=0.
\end{equation}
Apply the $b$-part of the graph $3_1$-law to the right anti-parallel
pair $(y_{L+1},y_{L+1}')$. Since that pair is in state $--$,
\begin{equation}
\label{eq:lck43}
\partial_{y_{L+1}} b_q(y_{L+1}'=-)=0 \quad \Longrightarrow \quad 
b_q([++;--];Z)
 =
b_q([++;+-];Z).
\end{equation}
The state $[++;+-]$ is controlled, so by \textnormal{(C1)}
\[
b_q([++;+-];Z)=0.
\]
Thus
\begin{equation}
\label{eq:lck44}
b_q(\cm_{+-};Z)=0.
\end{equation}
Now consider
\[
\cm_{-+}=[--;++]. 
\]
The right pair is in state $++$, so the $a$-part of the graph
$3_1$-law gives
\begin{equation}
\label{eq:lck45}
\partial_{y_{L+1}} a_q(y_{L+1}'=+)=0 \quad \Longrightarrow \quad
a_q([--;++];Z)
 =
a_q([--;-+];Z)=0,
\end{equation}
because $[--;-+]$ is controlled.
The left pair is in state $--$, so the $b$-part gives
\begin{equation}
\partial_{y_0} b_q(y_0'=-)=0 \quad \Longrightarrow \quad
b_q([--;++];Z)
 =
b_q([+-;++];Z)=0,
\end{equation}
because $[+-;++]$ is controlled.
Thus, both components vanish on both locked faces 
\[
a_q(\cm_{+-};Z) = 0, \qquad b_q(\cm_{+-};Z)=0, \qquad a_q(\cm_{-+};Z) = 0, \qquad b_q(\cm_{-+};Z)=0.
\]

\paragraph{Step 2: same sign clasp states.}
Consider now the remaining uncontrolled states in (C1a)
\[
\cm_{++} \times \{\omega_Z\}, \qquad \cm_{--}\times \{\alpha_Z\}.
\]
Take
\[
\cm_{++} \times \{\omega_Z\} = ([++;++];\omega_Z).
\]
Apply the graph $3_1$-law Lemma~\ref{lem:31-law} to the left anti-parallel pair
$(y_0,y_0')$, with the right clasp pair and $Z=\omega_Z$ fixed.  Since the left
pair is in state $++$,
\begin{equation}
\label{eq:lck41a}
\partial_{y_0} a_q(y_0'=+)=0 \quad \Longrightarrow \quad 
a_q([++;++];Z)
 =
a_q([-+;++];Z).
\end{equation}
The state $[-+;++]$ is controlled because its left clasp pair has
opposite signs. By \textnormal{(C1)}
\[
a_q([-+;++];Z)=0,
\]
and therefore
\begin{equation}
\label{eq:lck42a}
a_q(\cm_{++};Z)=0.
\end{equation}
To derive the $b$ component apply the $3_1$-law Lemma~\ref{lem:31-law} to the twist $Z$ with
the clasp states fixed. Let
\[
\omega_Z^{(i,+)} := \omega_Z(z_i=+).
\]
Since the twist is in state $Z=\omega_Z$, for any $i \neq j$
\begin{equation}
    \partial_{z_i} b_q(z_j=-) = 0 \quad \Longrightarrow \quad b_q([++;++]; Z=\omega_Z) = b_q([++;++]; Z=\omega_Z^{(i,+)}).
\end{equation}
The state $([++;++];\omega_Z^{(i,+)})$ is controlled because the twist consists of
opposite signs
\[
([++;++];\omega_Z^{(i,+)}) \in Q_Y \times (Q_Z \setminus \{\alpha_Z,\omega_Z\}) \subset U_{\text{ctrl}}.
\]
By \textnormal{(C1)}
\[
b_q([++;++];Z=\omega_Z^{(i,+)})=0,
\]
and therefore
\begin{equation}
\label{eq:lck42b}
b_q(\cm_{++};Z=\omega_Z)=0.
\end{equation}
Consider now
\[
\cm_{--} \times \{\alpha_Z\} = ([--;--];\alpha_Z).
\]
Apply the graph $3_1$-law Lemma~\ref{lem:31-law} to the left anti-parallel pair
$(y_0,y_0')$, with the right clasp pair and $Z=\alpha_Z$ fixed.  Since the left
pair is in state $--$,
\begin{equation}
\label{eq:lck41ab}
\partial_{y_0} b_q(y_0'=-)=0 \quad \Longrightarrow \quad 
b_q([--;--];Z)
 =
b_q([+-;--];Z).
\end{equation}
The state $[+-;--]$ is controlled because its left clasp pair has
opposite signs. By \textnormal{(C1)}
\[
b_q([+-;--];Z)=0,
\]
and therefore
\begin{equation}
\label{eq:lck42ab}
b_q(\cm_{--};Z)=0.
\end{equation}
To derive the $a$ component apply the $3_1$-law Lemma~\ref{lem:31-law} to the twist $Z$ with
the clasp states fixed. Let
\[
\alpha_Z^{(i,-)} := \alpha_Z(z_i=-).
\]
Since the twist is in state $Z=\alpha_Z$, for any $i \neq j$
\begin{equation}
    \partial_{z_i} a_q(z_j=+) = 0 \quad \Longrightarrow \quad a_q([--;--]; Z=\alpha_Z) = a_q([--;--]; Z=\alpha_Z^{(i,-)}).
\end{equation}
The state $([--;--];\alpha_Z^{(i,-)})$ is controlled because the twist consists of
opposite signs
\[
([--;--];\alpha_Z^{(i,-)}) \in Q_Y \times (Q_Z \setminus \{\alpha_Z,\omega_Z\}) \subset U_{\text{ctrl}}.
\]
By \textnormal{(C1)}
\[
a_q([--;--];Z=\alpha_Z^{(i,-)})=0,
\]
and therefore
\begin{equation}
\label{eq:lck42ba}
a_q(\cm_{--};Z=\alpha_Z)=0.
\end{equation}
From \eqref{eq:lck42a}, \eqref{eq:lck42b}, \eqref{eq:lck42ab} and \eqref{eq:lck42ba}
\[
a_q(\cm_{++};Z=\omega_Z)=b_q(\cm_{++};Z=\omega_Z)=0, \quad
a_q(\cm_{--};Z=\alpha_Z)=b_q(\cm_{--};Z=\alpha_Z)=0.
\]
Together with the controlled-locus vanishing (C1), 
\[
f_q(E;\eta;Z) = 0, \qquad \forall \eta \in U_{\text{ctrl}},
\]
and the exceptional vanishing (C2)
\[
f_q(E;\cm_{++};Z=\alpha_Z)=0, \qquad f_q(E;\cm_{--};Z=\omega_Z)=0,
\]
this proves
\[
G_q(Q_{\mathcal M}^E).
\]
\end{proof}

\begin{theorem}[Signed-cube compression]
\label{thm:signed-cube-compression}
Let $\mathcal H$ be an exterior host family carried by a fixed
oriented shadow. Assume that \textnormal{(C1)--(C3)} hold uniformly
over $\mathcal H$. Then
\[
G_q\bigl(Q_{\mathcal M}(\mathcal H)\bigr)
\qquad
\bigl(q>\widehat B(\mathcal H)\bigr).
\]
The conclusion is independent of the residual twist length $L$.
\end{theorem}

\begin{proof}
For fixed $L$, the family $Q_{\mathcal M}(\mathcal H)$ is carried by
a fixed finite shadow. Hence Corollary~\ref{cor:p-degree-cap} supplies a finite uniform
ambient cap $A_{\mathcal M}$ such that
\[
G_r\bigl(Q_{\mathcal M}(\mathcal H)\bigr)
\qquad
(r>A_{\mathcal M}).
\]
The number $A_{\mathcal M}$ may depend on $L$. Enlarge it, if
necessary, so that
\[
A_{\mathcal M}\geq \widehat B(\mathcal H)+1.
\]
Descend on
\[
q=A_{\mathcal M},A_{\mathcal M}-1,\ldots,
\widehat B(\mathcal H)+1.
\]
Assume inductively that
\[
G_r\bigl(Q_{\mathcal M}(\mathcal H)\bigr)
\qquad
(r>q).
\]
Then, for every $E\in\mathcal H$,
\[
G_{q+2}(Q_{\mathcal M}^E),
\qquad
G_{q+1}(Q_{\mathcal M}^E).
\]
Since $q>\widehat B(\mathcal H)$, Lemma~\ref{lem:locked-state-escape} gives
\[
G_q(Q_{\mathcal M}^E).
\]
As $E$ was arbitrary,
\[
G_q\bigl(Q_{\mathcal M}(\mathcal H)\bigr).
\]
The induction reaches $q=\widehat B(\mathcal H)+1$, proving the result.
\end{proof}



\begin{remark}[A uniform ambient cap]
The proof of Theorem~\ref{thm:signed-cube-compression} requires only the existence of a finite
ambient cap, but one may give the following estimate.
Suppose that the pre-planting host shadow has $n$ crossings and
$\ell$ components, so that
\[
B_{\mathrm{host}}=3n+\ell-1.
\]
Smoothing the seed crossing leaves
\[
n_0=n-1
\]
crossings and produces at most one additional component:
\[
\ell_0\leq\ell+1.
\]
The clasping twist machine adds $L+4$ crossings and may contribute at
most one further closed machine component. Hence the full signed
machine shadow satisfies
\[
n_{\mathcal M}
 =
n-1+(L+4)
 =
n+L+3,
\]
and
\[
\ell_{\mathcal M}
 \leq
\ell_0+1
 \leq
\ell+2.
\]
Corollary~\ref{cor:p-degree-cap} therefore gives
\[
\begin{aligned}
A_{\mathcal M}
 &\leq
3n_{\mathcal M}+\ell_{\mathcal M}-1\\
 &\leq
3(n+L+3)+(\ell+2)-1\\
 &=
(3n+\ell-1)+3L+11\\
 &=
B_{\mathrm{host}}+3L+11.
\end{aligned}
\]
This estimate is not used in the descending argument.
\end{remark}


\subsection{Machine initialization over the residual twist cube}

\begin{lemma}[Residual initialization]
\label{lem:four-clasp-residual-initialization}
Let $\mathcal H$ be an exterior host family. Assume that
\textnormal{(C1)--(C4)} hold uniformly over $\mathcal H$. Then
\[
G_r\bigl(Q_\rho(\mathcal H)\bigr)
\qquad
\bigl(r>\widehat B(\mathcal H)+2\bigr).
\]
\end{lemma}


\begin{proof}
Fix $E\in\mathcal H$ and $Z\in Q_Z$. Fix the primed clasp
coordinates at their original signs
\[
y_0'=-,
\qquad
y_{L+1}'=+.
\]
For $s,t\in\{-,0,+\}$ write
\[
D_{st}(Z) := D(y_0=s,\,
  y_0'=-,\,
  y_{L+1}=t,\,
  y_{L+1}'=+;\,Z)
, \qquad 
V_{st}(Z)
 :=
V(D_{st}(Z)).
\]

\paragraph{Step 1: prove a Jones skein paired identity.}
Let us first show that the Jones skein for $y_0$ and $y_{L+1}$ gives the exact paired identity
\begin{equation}
\label{eq:paired-skein}
V_{-+}(Z)+V_{+-}(Z)
 =
x^{-4}V_{++}(Z)
+x^4V_{--}(Z)
-p^2V_{00}(Z).
\end{equation}
We use the Jones skein in the form
\[
x^{-2}(+)-x^2(-)=p(0), \qquad p=x-x^{-1}.
\]
 
\paragraph{Step 1.1: compute $V_{-+}$}
Apply the skein in the first coordinate $y_0$, with $y_{L+1}=+$ fixed:
\[
x^{-2}V_{++}-x^2V_{-+}=pV_{0+}.
\]
Solving for $V_{-+}$,
\[
x^2V_{-+}=x^{-2}V_{++}-pV_{0+},
\]
so
\begin{equation}
\label{eq:l421}
V_{-+}=x^{-4}V_{++}-p\,x^{-2}V_{0+}.
\end{equation}
Now apply the skein in the second coordinate $y_{L+1}$, with $y_0=0$ fixed:
\[
x^{-2}V_{0+}-x^2V_{0-}=pV_{00}.
\]
Thus
\begin{equation}
\label{eq:l422}
x^{-2}V_{0+}=pV_{00}+x^2V_{0-}.
\end{equation}
Substitute~\eqref{eq:l422} into~\eqref{eq:l421}:
\[
V_{-+} = x^{-4}V_{++} - p\bigl(pV_{00}+x^2V_{0-}\bigr).
\]
Therefore
\begin{equation}
\label{eq:l423}
V_{-+} = x^{-4}V_{++} - p^2V_{00} - p\,x^2V_{0-}.
\end{equation}
 
\paragraph{Step 1.2: compute $V_{+-}$}
Apply the skein in the first coordinate $y_0$, with $y_{L+1}=-$ fixed:
\[
x^{-2}V_{+-}-x^2V_{--}=pV_{0-}.
\]
So
\[
x^{-2}V_{+-}=x^2V_{--}+pV_{0-},
\]
and hence
\begin{equation}
\label{eq:l424}
V_{+-} = x^4V_{--}+p\,x^2V_{0-}.
\end{equation}
\paragraph{Step 1.3: add \eqref{eq:l423} and \eqref{eq:l424}}
The unknown one-smoothed term cancels:
\[
-p\,x^2V_{0-}+p\,x^2V_{0-}=0.
\]
Thus
\[
V_{-+}+V_{+-} = x^{-4}V_{++}+x^4V_{--}-p^2V_{00}.
\]
That proves the exact paired skein identity.

\paragraph{Step 2: JVP coefficient extraction.}
Note that the exact identity implies
\begin{equation}
\label{eq:l425}
f_q(V_{+-})+f_q(V_{-+}) = \bigl(x^{-4}V_{++}\bigr)_q + \bigl(x^4V_{--}\bigr)_q - f_{q-2}(V_{00}),
\end{equation}
since multiplying by $p^2$ shifts the JVP coefficients by $2$. 
Now
\[
        x^4=(1+p^2)+(2p+p^3)x,
\]
and, using \(x^2=px+1\), multiplication by \(x^4\) sends a general coefficient vector
\(\mathscr{A}(p)+\mathscr{B}(p)x\) to a coefficient vector whose \(p^q\)-term may involve source
levels $q,\ q-1,\ q-2,\ q-3,\ q-4$.
Indeed, the extra \(q-4\) contribution appears from reducing
\[
        (2p+p^3)x\cdot \mathscr{B}(p)x
\]
via \(x^2=px+1\). Similarly,
\[
        x^{-4}=(1+3p^2+p^4)-(2p+p^3)x
\]
also uses only source levels $q,\ q-1,\ q-2,\ q-3,\ q-4$.
Assume
\[
q>\widehat B+4.
\]
Then, all source levels appearing in
\[
        (x^{-4}V_{++})_q
        \qquad\text{and}\qquad
        (x^4V_{--})_q
\]
are strictly greater than \(\widehat B\). 


\paragraph{Step 3: conclude $G_r(Q_\rho)$.}
By Theorem~\ref{thm:signed-cube-compression}, every signed machine
state vanishes at every level greater than $\widehat B$. Hence
\[
f_q(V_{+-})=f_q(V_{-+})=0,
\]
and by Step 2 also
\[
(x^{-4}V_{++})_q = (x^4V_{--})_q = 0.
\]
Substituting these into \eqref{eq:l425} yields
\[
f_{q-2}(V_{00})=0.
\]
By \textnormal{(C4)},
\[
D_{00}(Z) \cong D(\rho; Z).
\]
Thus
\[
f_{q-2}(D(\rho; Z))=0.
\]
Writing $r=q-2$, the condition $q>\widehat B+4$ becomes
\[
r>\widehat B+2.
\]
Since $E$ and $Z$ were arbitrary,
\[
G_r\bigl(Q_\rho(\mathcal H)\bigr)
\qquad
\bigl(r>\widehat B(\mathcal H)+2\bigr).
\]
\end{proof}

\subsection{Finite-type descent along an anchor}

Finite-type vanishing below the compression threshold is facilitated by an anchor. We therefore add Hypothesis (H1) below.

\begin{define}[Residual sign face]
\label{def:residual-sign-face}
Let
\[
Z=\{z_1,\ldots,z_L\},
\qquad
L\equiv1\pmod2,
\]
be the residual twist block.
Fix once and for all a Reidemeister--II pairing pattern on
$L-1$ of the residual crossings, leaving one crossing
\[
z_\ast\in Z
\]
unpaired. For each paired pair, fix one of its two opposite-sign
assignments, compatible with the displayed Reidemeister--II
cancellations.
Define $Z^+$ and $Z^-$ by using these same fixed opposite-sign
assignments on every paired pair and setting
\[
z_\ast=+
\quad\text{in }Z^+,
\qquad
z_\ast=-
\quad\text{in }Z^-.
\]
Thus $Z^+$ and $Z^-$ differ only at the unpaired effective crossing.

After cancelling all fixed Reidemeister--II pairs, $Z^+$ reduces to a
single positive crossing and $Z^-$ reduces to a single negative
crossing. The residual sign face is
\[
F_Z:=\{Z^+,Z^-\}.
\]
In particular,
\[
\operatorname{wr}(Z^+)=+1,
\qquad
\operatorname{wr}(Z^-)=-1.
\]
Both states are non-extreme when $L>1$:
\[
Z^+,Z^-\notin\{\alpha_Z,\omega_Z\}.
\]
\end{define}

\paragraph{(H1) Residual anchor.}
Let \(s\) be the seed crossing of the host diagram $D$.
Let $Z^+,Z^-$ be the residual sign states determined by the fixed
pairing convention of Definition~\ref{def:residual-sign-face}.
Since \(L\) is odd, the
residual twist block has two chosen residual sign states
\[
        \{\rho\}\times \{Z^+\}, \; \{\rho\}\times \{Z^-\} \in Q_\rho
\]
which cancel by Reidemeister-II moves to a single positive, respectively
negative, crossing. Let \(Z^\circ\) be the one whose surviving crossing has the
same sign as the original seed crossing \(s\). 
Then
\[
        D(\rho; Z^\circ)\cong D
\]
relative to the exterior of the planting disk. In particular, if \(f_r(D)=0\),
then
\[
        f_r(\rho; Z^\circ)=0.
\]
For \(L>1\), the state \(Z^\circ\) is non-exceptional:
\[
        Z^\circ\notin\{\alpha,\omega\}.
\]

\begin{theorem}[Finite-type descent]
\label{thm:four-clasp-descent}
Let $K$ be the host knot of a clasping twist machine. Assume \textnormal{(C1)--(C4)} and (H1).
Suppose
\[
L>\widehat B+2,
\qquad
L\equiv1\pmod2,
\]
and
\[
f_r(K)=0
\qquad
(1\le r\le \widehat B+2).
\]
Then
\[
G_q(Q_\rho)
\qquad
(q\ge2).
\]
\end{theorem}

\begin{proof}
Lemma~\ref{lem:four-clasp-residual-initialization} gives
\[
G_r(Q_\rho)
\qquad
(r>\widehat B+2).
\]
In particular,
\[
G_{\widehat B+4}(Q_\rho),
\qquad
G_{\widehat B+3}(Q_\rho).
\]
Descend on
\[
q=\widehat B+2,\widehat B+1,\ldots,2.
\]
Assume
\[
G_{q+2}(Q_\rho),
\qquad
G_{q+1}(Q_\rho).
\]
The graph $3_1$-law and connected-support propagation Lemma~\ref{lem:graph-connectivity} on the residual
twist cube imply that
\[
a_q(\rho; Z)
\quad\text{is constant for }Z\ne\omega,
\]
and
\[
b_q(\rho; Z)
\quad\text{is constant for }Z\ne\alpha.
\]
Let $Z^\circ$ be the residual anchor state. Since
\[
D(\rho; Z^\circ)\simeq K
\]
and $q\le \widehat B+2$, the assumed anchor vanishing gives
\[
f_q(\rho; Z^\circ)=0.
\]
The anchor $Z^\circ$ lies in both non-exceptional branches, and therefore by Lemma~\ref{lem:graph-connectivity}
\begin{equation}
\label{eq:non-exc1}
a_q(\rho; Z)=0 \qquad (Z\neq\omega),
\end{equation}
and
\begin{equation}
\label{eq:non-exc2}
b_q(\rho; Z)=0 \qquad (Z\neq\alpha).
\end{equation}
Thus $a_q$ can only be supported at $Z=\omega$, and $b_q$ can only be supported at $Z=\alpha$.
As these are one-point supports we have
\[
a_q(\rho;\omega) = (-1)^L \partial_Z a_q(\rho; Z), \qquad b_q(\rho; \alpha) = \partial_Z b_q(\rho; Z).
\]
Since $L = |Z|$ and $q\le \widehat B+2<L$ both antipodal endpoints vanish by degree/type considerations
\[
a_q(\rho; \omega)=0, \qquad b_q(\rho; \alpha)=0.
\]
Because the non-exceptional branches have already vanished, \eqref{eq:non-exc1} and \eqref{eq:non-exc2}, the endpoint vanishing gives \(a_q\equiv0\) and \(b_q\equiv0\) on \(Q_\rho\).
Therefore $G_q(Q_\rho)$. Descending to $q=2$ proves the theorem.
\end{proof}

\subsection{Two-clasping finite-type descent}


When two machines are jointly planted, each of them sees the other as an integral part of the ambient diagram. It thus seems as though their compression threshold grows linearly with the length of their counterparts, i.e. by Corollary~\ref{cor:p-degree-cap}, $\widehat B_1 = \widehat B+3L_2$ and $\widehat B_2 = \widehat B+3L_1$, with $\widehat B$ the usual machine threshold. In what follows, we show that these bounds may be mitigated by restricting one of the machines to a face. Essentially, we use one of the machines to clone the vanishing anchor. The pair of anchors are then used to seed the second machine. This approach allows for initializing both machines using Lemma~\ref{lem:four-clasp-residual-initialization} with a threshold $\widehat B$.
Further finite-type descent along the vanishing anchors may then proceed below this cap.


\begin{theorem}[Facewise two-clasping descent via generated anchors]
\label{thm:joint-descent2}
Let \(D\) be a crossing-minimal diagram of a knot \(K\), and set
\[
        B_0:=3c(K)=3\,\#\operatorname{cross}(D).
\]
Let
\[
        \widehat B:=B_0+\kappa_{\mathrm{mach}}.
\]
Let \(u\neq v\) be two crossings of \(D\). Plant two clasping twist machines in
disjoint disks at \(u\) and \(v\), with odd residual lengths
\[
        L_u>\widehat B+2,\qquad L_v>\widehat B+2.
\]
Let $Q_{\rho_u}$ and $Q_{\rho_v}$ be their residual cubes, and let
\[
        F_u=\{Z_u^+,Z_u^-\}\subset Q_{\rho_u},\qquad
        F_v=\{Z_v^+,Z_v^-\}\subset Q_{\rho_v}
\]
be their residual sign faces. Assume (C1)--(C4) and (H1) for both machines.
Assume
\[
        f_r(K)=0 \qquad 1\le r\le \widehat B+2.
\]
Then
\[
        G_q(F_u\times F_v) \qquad (q\ge 2).
\]
In particular,
\[
        G_2(F_u\times F_v).
\]
\end{theorem}

\begin{proof}
Let \(Z_u^\circ\in F_u\) be the residual sign state of the first machine whose
effective crossing has the original sign of \(u\). By the residual anchor hypothesis
(H1)
\[
        D(\rho_u; Z_u^\circ)\cong D.
\]


\medskip

\paragraph{Step 1: use the $v$-machine to generate two anchors.}
Restrict the $u$-machine to the fixed state \(Z_u=Z_u^\circ\).
On this face, the double-planted family is just the base knot \(K\)
with the second clasping twist machine planted at \(v\). 
Since \(Z_u=Z_u^\circ\) reproduces the original crossing \(u\), the
host family for the $v$-machine is isotopic to the original diagram
\(K\). Hence the ordinary host cap for the \(v\)-machine is
\[
        B_0=3c(K).
\]
After adding the fixed local overhead of the clasping machine, the relevant compression 
threshold is
\[
        \widehat B=B_0+\kappa_{\mathrm{mach}}.
\]
Since
\[
L_v>\widehat B+2
\]
and
\[
f_r(K)=0
\qquad
1\le r\le \widehat B+2,
\]
Theorem~\ref{thm:four-clasp-descent} applied to the $v$-machine gives
\[
G_q(Z_u^\circ,Q_{\rho_v})
\qquad
(q\ge2).
\]
In particular,
\begin{equation}
\label{eq:t4101}
f_q(Z_u^\circ,Z_v^+)=0,
\qquad
f_q(Z_u^\circ,Z_v^-)=0
\qquad
(q\ge2).
\end{equation}
Thus, the $v$-machine supplies two vanishing anchor values for the $u$-machine, one over each of the two residual sign states of the $v$-machine.

\medskip

\paragraph{Step 2: descend the $u$-machine over the two-state host
\(F_v\).}

Now restrict the $v$-machine to the two-state residual sign face
\[
F_v=\{Z_v^+,Z_v^-\}.
\]
%
After canceling the fixed Reidemeister-II pairs in the $v$-residual twist,
each vertex \(Z_v^\epsilon\in F_v\) is identified with the original shadow of
\(D\), with the crossing \(v\) assigned sign. Hence, each such host
diagram is a crossing-flip of the crossing-minimal diagram \(D\), and therefore
has exactly \(c(K)\) crossings. Thus the two-state host family \(F_v\) has ordinary host cap
\[
        B_0=3c(K).
\]
Adding the fixed local overhead gives the uniform compression threshold
\[
        \widehat B=B_0+\kappa_{\mathrm{mach}}
\]
for the \(u\)-machine over \(F_v\).
The $u$-machine satisfies Section~\ref{sec:clasping-hypo} hypotheses
over this two-state host family. 
Lemma~\ref{lem:four-clasp-residual-initialization} gives
\[
G_r(Q_{\rho_u}\times F_v)
\qquad
(r>\widehat B+2).
\]
Hence
\[
G_{\widehat B+4}(Q_{\rho_u}\times F_v),
\qquad
G_{\widehat B+3}(Q_{\rho_u}\times F_v)
\]
are available.
We now descend on
\[
q=\widehat B+2,\ \widehat B+1,\ldots,2.
\]
Assume inductively that
\[
G_{q+2}(Q_{\rho_u}\times F_v),
\qquad
G_{q+1}(Q_{\rho_u}\times F_v)
\]
hold.
Fix \(Z_v^\varepsilon\in F_v\). On the \(Z_u\)-cube, Lemma~\ref{lem:31-law} and Lemma~\ref{lem:graph-connectivity} give
\[
a_q(Z_u,Z_v^\varepsilon)
\text{ constant on }Z_u\neq\omega_u,
\]
and
\[
b_q(Z_u,Z_v^\varepsilon)
\text{ constant on }Z_u\neq\alpha_u.
\]
The state \(Z_u^\circ\) is non-extreme, since \(L_u\) is odd and
\(L_u>1\). By \eqref{eq:t4101},
\[
f_q(Z_u^\circ,Z_v^\varepsilon)=0.
\]
Therefore, both non-exceptional branch constants vanish
\begin{equation}
\label{eq:non-exc3}
    a_q(Z_u,Z_v^\eps) = 0 \quad (Z_u \neq \omega_u), \qquad b_q(Z_u,Z_v^\eps) = 0 \quad (Z_u \neq \alpha_u).
\end{equation}
Thus, \(a_q\) may only be supported at \(Z_u=\omega_u\), and \(b_q\) may only be supported at \(Z_u=\alpha_u\).  As these are one-point supports we have
\[
a_q(\omega_u,Z_v^\eps) = (-1)^{L_u} \partial_{Z_u} a_q(Z_u,Z_v^\eps), \qquad b_q(\alpha_u,Z_v^\eps) = \partial_{Z_u} b_q(Z_u,Z_v^\eps).
\]
Since $L_u = |Z_u|$ and $q\le \widehat B+2<L_u$ both antipodal endpoints vanish by degree/type considerations
\[
a_q(\omega_u,Z_v^\eps)=0, \qquad b_q(\alpha_u,Z_v^\eps)=0.
\]
Because the non-exceptional branches have already vanished \eqref{eq:non-exc3}, the endpoint
vanishing gives \(a_q\equiv0\) and \(b_q\equiv0\) on \(Q_{\rho_u} \times \{Z_v^\eps\}\).
Equivalently,
\[
f_q(Z_u,Z_v^\varepsilon)=0
\]
for every \(Z_u\).
Since \(Z_v^\varepsilon\in F_v\) was arbitrary, we get
\[
G_q(Q_{\rho_u}\times F_v).
\]
Descending to \(q=2\), we obtain
\[
G_q(Q_{\rho_u}\times F_v)
\qquad(q\ge2).
\]

\medskip

\paragraph{Step 3: restrict to the residual sign face of the $u$-machine.}
Finally restrict the $u$-residual cube to its two-state sign face
\[
F_u=\{Z_u^+,Z_u^-\}.
\]
Since
\[
F_u\times F_v\subset Q_{\rho_u}\times F_v,
\]
we obtain
\[
G_q(F_u\times F_v)
\qquad(q\ge2).
\]
In particular,
\[
G_2(F_u\times F_v).
\]
\end{proof}

\subsection{Geometric characterization}

The following geometric package supplies the algebraic requirements of the preceding
subsections:

\medskip
\noindent
\(\bf (G0)\) supplies pairwise anti-parallel twist relations on the full signed cube.

\noindent
\(\bf (G1)\) supplies the direct reductions of the non-exceptional states (C1).

\noindent
\(\bf (G2)\) converts the two exceptional antipodes into nonexceptional twist states, hence proving (C2).

\noindent
\(\bf (G3)\) supplies the clasp $3_1$-law relations (C3).

\noindent
\(\bf (G4)\) recovers the residual twist (C4).

\noindent
\(\bf (G5)-(G7)\) provide the anchor (H1), effective crossing and smoothing proxy.

\noindent
\(\bf (G8)\) supplies locality for the two-machine product.

\begin{proposition}[Geometry of clasping twist machines]
\label{prop:clasping-twist-machines}
Let a clasping twist machine be planted at a seed crossing \(s\) inside a disk
\(\Delta_s\), with odd residual length \(L>1\). The following local facts hold
relative to the exterior of \(\Delta_s\). Denote the clasp states
\[
[s_0,s_0';s_1,s_1'] = \{y_0=s_0,\ y_0'=s_0',\ y_{L+1}=s_1,\ y_{L+1}'=s_1'\}.
\]

\paragraph{\(\bf (G0)\) Anti-parallel twist cube.}
For every exterior state, every signed clasp state $Y\in Q_Y$,
and every pair $z_i\ne z_j$ in $Z$,
\[
\partial_{z_i}f_r(z_j=0)=0,
\qquad
\partial_{z_j}f_r(z_i=0)=0.
\]
Thus, the $Z$-coordinates remain pairwise anti-parallel throughout
the full signed machine cube.

\paragraph{\(\bf (G1)\) Controlled states and compression threshold.}
In the endpoint labeling convention of Definition~\ref{def:clasping-twist-machine} and Section~\ref{sec:machine-states}, the signed controlled faces in $U_{\text{ctrl}}$
\[
[+-;++]\times\{Z\},\ [-+;++]\times\{Z\},\ [+-;--]\times\{Z\},\ [-+;--]\times\{Z\},
\]
\
\[
[++;+-]\times\{Z\},\ [++;-+]\times\{Z\},\ [--;+-]\times\{Z\},\ [--;-+]\times\{Z\},
\]
\
\[
[+-;+-]\times\{Z\},\ [+-;-+]\times\{Z\},\ [-+;+-]\times\{Z\},\ [-+;-+]\times\{Z\},
\]
\
reduce for all $Z \in Q_Z$, after the local cancellations prescribed by the machine, to diagrams
whose JVP \(p\)-degree is bounded by
\[
        \widehat B
        =
        B_{\mathrm{host}}+\kappa_{\mathrm{mach}}.
\]
For the present construction,
\[
        \kappa_{\mathrm{mach}}=11.
\]
Thus, (C1c) holds with cap \(\widehat B\). Assume moreover
\[
G_{r+1}, \quad G_{r+2}, \quad r > \widehat B.
\]
Then (C1d) holds
\[
f_r(Y;Z) = 0, \qquad \forall Y \in Q_Y, \ \forall Z \not \in \{\alpha_Z,\omega_Z\}.
\]

\paragraph{\(\bf (G2)\) Exceptional collapse.} Assuming $G_{r+1},G_{r+2}$ for $r > \widehat B$,
the exceptional states
\[
f_r([++;++];Z=\alpha_Z) = 0, \qquad f_r([--;--]; Z=\omega_Z) = 0.
\]
Hence, (C2) holds.

\paragraph{\(\bf (G3)\) Anti-parallel clasps.} Each of the pairs
\[
\{y_0,y_0'\}, \qquad \{y_{L+1}, y_{L+1}'\},
\]
obeys an anti-parallel nugatory relation. Hence (C3) holds.

\paragraph{\(\bf (G4)\) Recovering the ordinary twist.} The residual face
\[
        Q_\rho=\{y_0=0,y_0'=-,y_{L+1}=0,y_{L+1}'=+\} \times Q_Z
\]
is an ordinary anti-parallel twist cube. Hence (C4) holds.

\paragraph{\(\bf (G5)\)} The residual sign states
\[
        Z^+,Z^-\in Q_\rho
\]
cancel by Reidemeister-II moves to a single positive, respectively negative,
effective crossing. If \(Z^\circ\) denotes the state whose surviving crossing
has the sign of the original seed crossing \(s\), then
\[
        D(\rho; Z^\circ)\cong D
\]
relative to the exterior of the planting disk. Hence (H1) holds.

\paragraph{\(\bf (G6)\)} The residual sign face
\[
        F_Z=\{Z^+,Z^-\}
\]
is naturally identified with a genuine one-crossing variation cube
\[
        \tilde s\in\{+,-\}.
\]
That is, after cancelling the fixed Reidemeister-II pairs in \(Z^+\) and
\(Z^-\), the two remaining local diagrams differ only by the sign of one
effective crossing \(\tilde s\).

\paragraph{\(\bf (G7)\)} Write \(D_{\mathrm{ext}}\) for the full diagram obtained by adjoining
the fixed exterior of the planting disk to the local tangle.
The oriented smoothing of the effective crossing is a smoothing proxy
for the original seed crossing:
\[
        D_{\mathrm{ext}}(\tilde s=0)\cong D(s=0)
\]
relative to the exterior of the planting disk.

\paragraph{\(\bf (G8)\)} If two machines are planted in disjoint disks, all reductions,
smoothings, and isotopies for one machine are supported in its own disk and
therefore commute with the corresponding operations for the other machine.
The global mirror involution preserves this product decomposition and commutes
with these local identifications.
Consequently, for two disjoint residual sign faces $F_{Z_u},F_{Z_v}$, the product
\[
        F_{Z_u}\times F_{Z_v}
\]
is a genuine two-crossing variation cube in the effective crossing coordinates
\[
        \tilde u,\tilde v\in\{+,-\}.
\]
Moreover,
\[
        D_{\mathrm{ext}}(\tilde u=0,\tilde v=\sigma)
        \cong
        D(u=0,v=\sigma),
        \qquad \sigma\in\{+,-\},
\]
and
\[
        D_{\mathrm{ext}}(\tilde u=0,\tilde v=0)
        \cong
        D(u=0,v=0).
\]

\end{proposition}

\newif\ifGeometry
\Geometrytrue

\ifGeometry

\begin{proof}
The proof is by construction.
All Reidemeister moves and isotopies in this proof are supported
inside the planting disk $\Delta_s$ and are taken relative to
$\partial\Delta_s$. Consequently, every reduction is uniform in the
exterior state, and the same sequence of local moves applies over any
exterior host family carried by a fixed shadow.

Consider a host knot $K$. Pick one of its crossings and delineate it by a small disk. Within the disk, smooth the base crossing and attach the clasping twist machine to the strands on both ends as shown below.
\[
\includegraphics[width=0.17\textwidth]{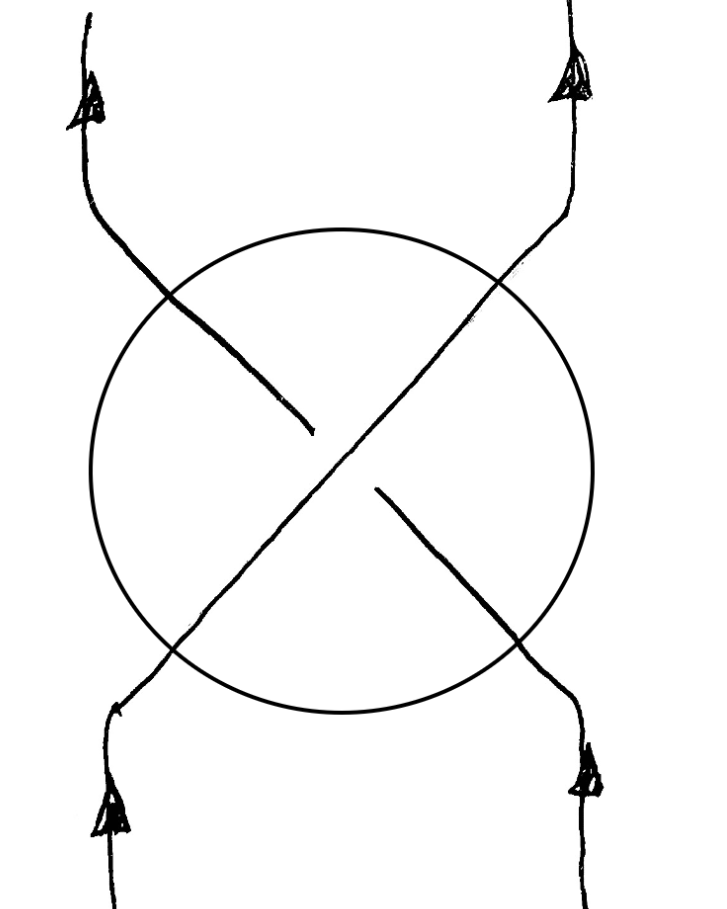}
\; \raisebox{1.7cm}{$\Longrightarrow$} \;
\includegraphics[width=0.17\textwidth]{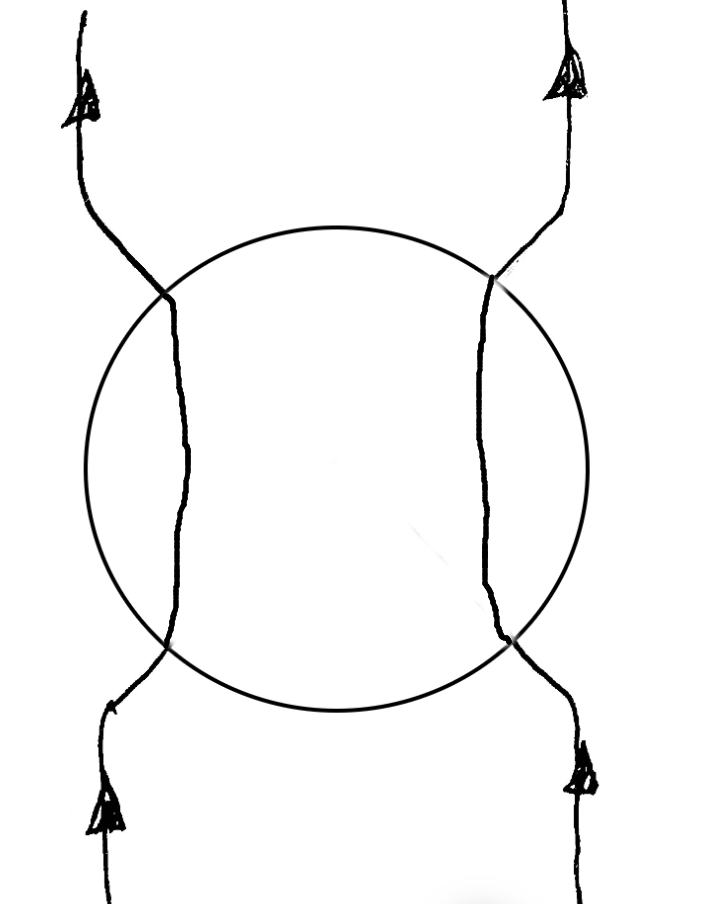}
\; \raisebox{1.7cm}{$\Longrightarrow$} \;
\raisebox{0.1cm}{
\includegraphics[width=0.3\textwidth]{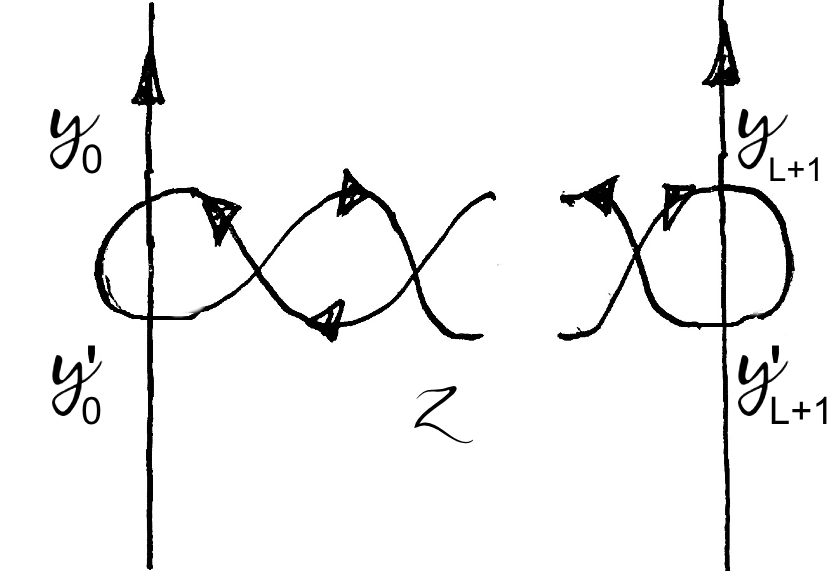}
}
\]
A crossing whose sign is indeterminate is drawn as a pair of intersecting arcs. The disconnected arcs in the drawing on the right represent part of a long anti-parallel twist with $L=|Z|$ crossings. 

\paragraph{$\bf (G0)$ Anti-parallel twist cube.} The twist is planted so that its arcs are anti-parallel. Oriented smoothing of any of its crossings renders the remaining nugatory. As an example
\[
\includegraphics[width=0.3\textwidth]{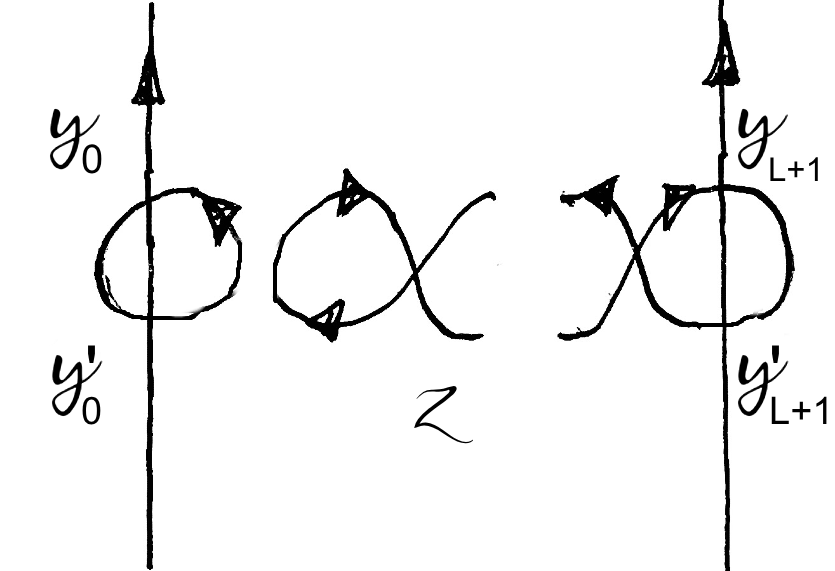}
\]
Therefore, the anti-parallel nugatory condition holds for any pair of distinct twist crossings.

\paragraph{$\bf (G1)$ Controlled states and $\kappa_{\text{mach}}$.} We verify the degree caps of the signed faces in $U_{\text{ctrl}}$.
Let the host shadow have $n$ crossings and $\ell$ components, so that
\[
B_{\mathrm{host}}=3n+\ell-1.
\]

\paragraph{Step 1: Directly reducible clasp states.}
We begin with the first set of twelve configurations. These consist of at least one clasp in a mixed cancelable state. 
There are exactly three types of reductions in this set of states. Once a clasp is put in cancelable state, it detaches from the seed strand by R2. The remaining twist crossings are undone by Reidemeister-I moves until a loop remains. This loop may either be disjoint or attached to the companion seed strand depending on the state of the complement clasp. The reductions are uniform in the residual state
\(Z\in Q_Z\). For every assignment of signs to the residual crossings, the
controlled clasp closure turns the \(Z\)-block into a chain of local nugatory
kinks after the displayed Reidemeister-II move or moves. These kinks are then
removed by Reidemeister-I moves. Hence, no residual crossing contributes to the
reduced degree cap on the controlled faces.

We illustrate the three possible reductions. 
Set $y_0 = -y_0'$ and assume $y_{L+1}=y_{L+1}'$. Use a single $R2$ and a sequence of $R1$ kinks to reduce the planted machine diagram to
\[
\raisebox{1.5cm}{$(y_0=-y_0',\ y_{L+1}=y_{L+1}') \quad \cong$} \;
\includegraphics[width=0.3\textwidth]{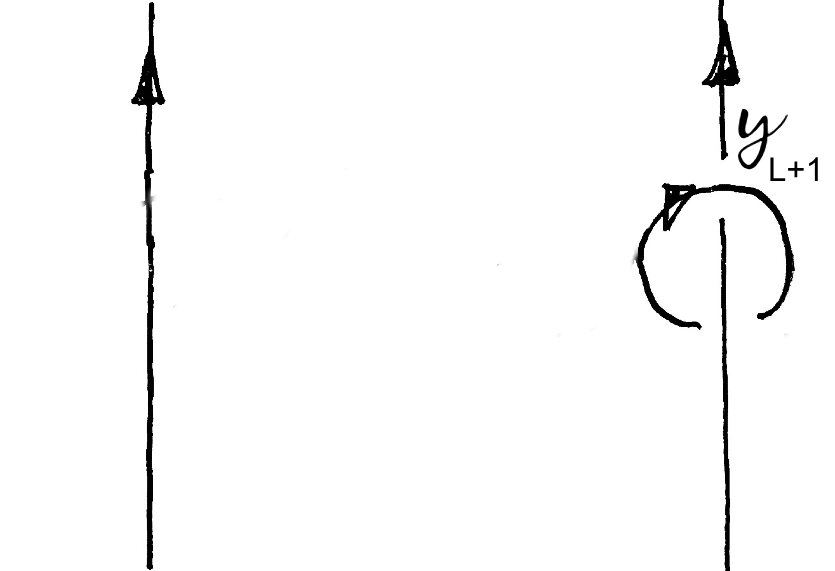}
\]
The exterior part of the diagram is left intact. 

Set $y_0 = y_0'$ and assume $y_{L+1}=-y_{L+1}'$. Use a single $R2$ and a sequence of $R1$ kinks to reduce the planted machine diagram to
\[
\raisebox{1.5cm}{$(y_0=y_0',\ y_{L+1}=-y_{L+1}') \quad \cong$} \;
\includegraphics[width=0.3\textwidth]{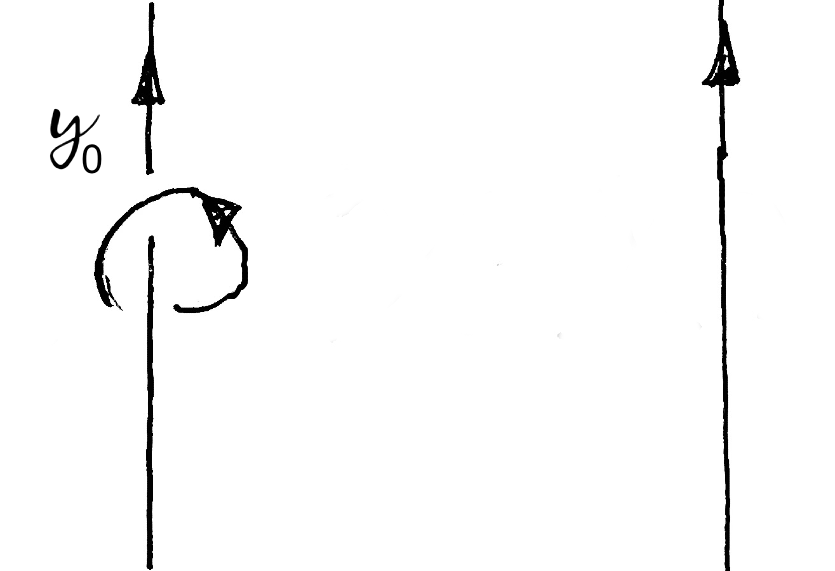}
\]
The exterior part of the diagram is left intact. For either of the first two reduction types, the seed crossing has been smoothed, at most two clasp crossings remain, and at most one
additional machine loop is present. Consequently, the reduced diagram
has
\[
n_{\mathrm I}\le n-1+2=n+1,
\qquad
\ell_{\mathrm I}\le \ell+2.
\]
Here the component estimate allows one component from smoothing the
seed crossing and one further component from a possible disjoint
machine loop. Corollary~\ref{cor:p-degree-cap} therefore gives
\[
B_{\mathrm I}
 \le
3(n+1)+(\ell+2)-1 = (3n+\ell-1)+5
 =
B_{\mathrm{host}}+5.
\]
%
Set $y_0 = -y_0', \ y_{L+1}=-y_{L+1}'$. Use a pair of $R2$ moves and a sequence of $R1$ kinks to reduce the planted machine diagram to
\[
\raisebox{1.5cm}{$(y_0=-y_0',\ y_{L+1}=-y_{L+1}') \quad \cong$} \;
\includegraphics[width=0.3\textwidth]{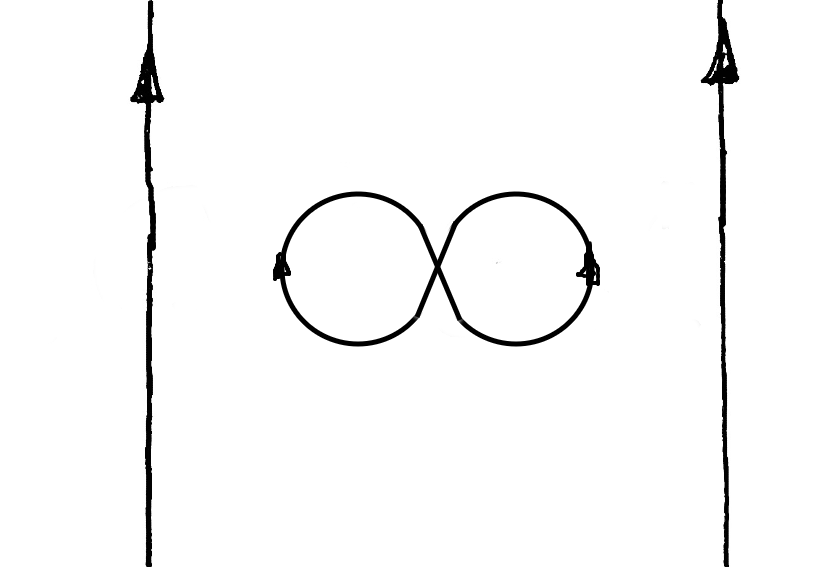}
\raisebox{1.5cm}{$\cong$} \;
\includegraphics[width=0.3\textwidth]{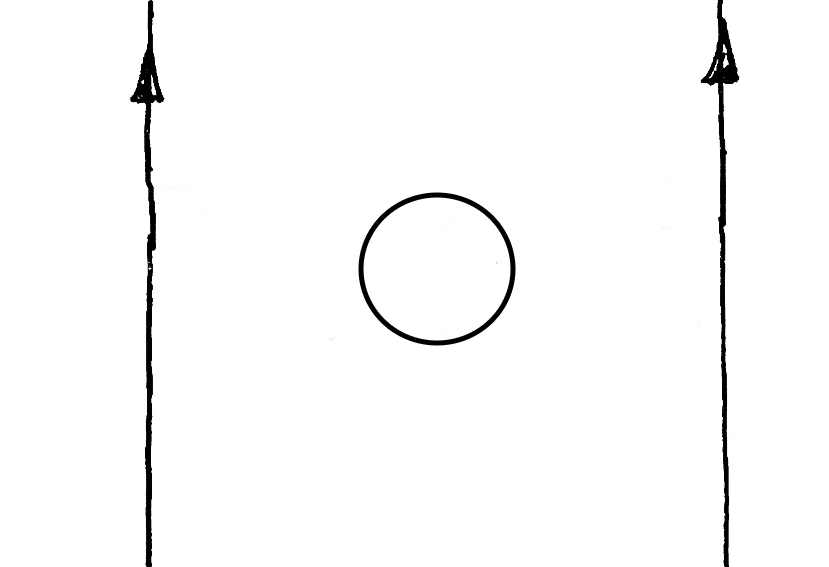}
\]
The exterior part of the diagram is left intact. 
For the third direct reduction type, no clasp crossing remains and
there is at most one additional machine loop. Hence
\[
n_{\mathrm{II}}\le n-1,
\qquad
\ell_{\mathrm{II}}\le \ell+2,
\]
and therefore
\[
B_{\mathrm{II}}
 \le
3(n-1)+(\ell+2)-1
 =
B_{\mathrm{host}}-1.
\]
It follows that every directly reducible signed state has JVP degree
at most
\[
B_{\mathrm{dir}}
 \le
B_{\mathrm{host}}+5.
\]
In particular, assuming $\kappa_{\text{mach}}=11$ gives
\[
r>\widehat B=B_{\mathrm{host}}+11
\quad\Longrightarrow\quad
f_r=0
\]
on every directly reducible signed state. Thus \textnormal{(C1c)}
holds.

\paragraph{Step 2: non-exceptional residual states.}
The remaining set of controlled states are characterized by non-exceptional twist configurations $Z \not \in \{\alpha_Z,\omega_Z\}$.
Assume
\[
G_{r+1}(Q_{\mathcal M}^{E}),
\qquad
G_{r+2}(Q_{\mathcal M}^{E}),
\qquad
r>\widehat B.
\]
The $3_1$-law Lemma~\ref{lem:31-law} applies. Thus, for any $i \neq j$
\[
\partial_{z_i} a_r(z_j=+)=0, \qquad \partial_{z_i} b_r(z_j=-)=0.
\]
The twist is pairwise anti-parallel and its exceptional branch is strongly connected (Lemma~\ref{lem:graph-connectivity}). Therefore
\[
a_r \ \text{is constant on } Q_Z \setminus \{\omega_Z\}
\]
and
\[
b_r \ \text{is constant on } Q_Z \setminus \{\alpha_Z\}.
\]
We show that these constants vanish. Write the JVP $a$-component skein \eqref{eq:theskein} at level $r+2$ for some crossing $z_i \in Z$
\[
\partial_{z_i} a_{r+2} = a_{r+1}(z_i=0) + b_{r+1}(z_i=+)+b_{r+1}(z_i=-)-a_r(z_i=+).
\]
The left hand side vanishes by $G_{r+2}$ and the $b$ terms on the right vanish by $G_{r+1}$. Therefore,
\begin{equation}
\label{eq:smoothing-a}
a_r(z_i=+)=a_{r+1}(z_i=0).
\end{equation}
The $b$ component skein at level $r+2$ 
\[
\partial_{z_j} b_{r+2} = b_{r+1}(z_j=0) + a_{r+1}(z_j=+)+a_{r+1}(z_j=-)+b_r(z_j=-).
\]
similarly gives
\begin{equation}
\label{eq:smoothing-b}
b_r(z_j=-)=-b_{r+1}(z_j=0).
\end{equation}
The smoothing terms on the right in \eqref{eq:smoothing-a} and \eqref{eq:smoothing-b} represent a diagram in which the twist becomes unlocked. In particular, since the twist is in non-extreme configuration it consists of both negative and positive sign crossings. Let $z_i$ and $z_j$ be some positive and negative sign crossings, respectively. Then $f_r(z_i=+,z_j=-)$ inherits the values of the order-$(r+1)$ coefficients of a diagram where one of these crossings is smoothed. To see what these coefficients are, we consider the shadow of the smoothed diagram.

After smoothing one residual twist crossing, the remaining residual
crossings are removed as Reidemeister-I kinks. The resulting auxiliary
shadow is obtained from the seed-smoothed host by retaining at most
four clasp crossings and at most two additional machine loops.
The entire collapse ends in a configuration whose shadow is
\[
\raisebox{1.5cm}{$(z_j=0) \cong (z_i=0) \quad \cong$} \;
\includegraphics[width=0.3\textwidth]{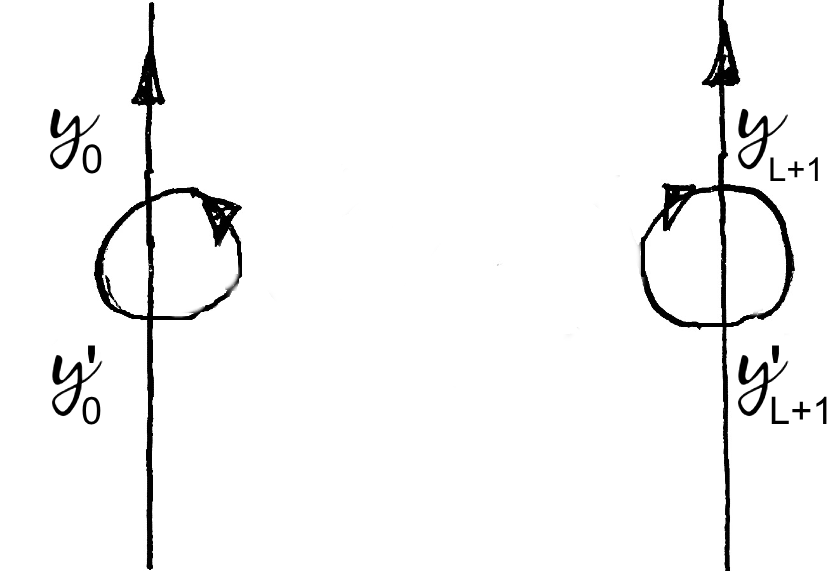}
\]
Therefore
\[
n_{\mathrm{aux}}
 \le
n-1+4=n+3,
\qquad
\ell_{\mathrm{aux}}
 \le
\ell+3.
\]
The component estimate includes at most one component created by
smoothing the seed and at most two additional machine loops. Hence
\[
B_{\mathrm{aux}}
 \le
3(n+3)+(\ell+3)-1
 =
B_{\mathrm{host}}+12.
\]

Set
\[
\widehat B:=B_{\mathrm{host}}+11.
\]
If $r>\widehat B$, then
\[
r+1>B_{\mathrm{host}}+12
       \ge B_{\mathrm{aux}}.
\]
Corollary~\ref{cor:p-degree-cap} consequently gives
\[
f_{r+1}(z_i=0)=f_{r+1}(z_j=0)=0.
\]
%
Equations \eqref{eq:smoothing-a} and \eqref{eq:smoothing-b} then imply
\[
a_r(z_i=+, z_j=-) = a_{r+1}(z_i=0,z_j=-)=0,
\]
and
\[
b_r(z_i=+, z_j=-) = -b_{r+1}(z_i=+,z_j=0)=0.
\]
Since the graph $3_1$-law makes $a_r$ constant on
$Q_Z\setminus\{\omega_Z\}$ and $b_r$ constant on
$Q_Z\setminus\{\alpha_Z\}$, it follows that
\[
f_r(Y;Z)=0
\qquad
\bigl(Z\notin\{\alpha_Z,\omega_Z\}\bigr).
\]
Thus, \textnormal{(C1d)} holds.
The effective compression overhead is therefore
\[
\kappa_{\mathrm{mach}}=11,
\]
and the corresponding threshold is
\[
\widehat B=B_{\mathrm{host}}+11.
\]

\paragraph{$\bf (G2)$ Exceptional collapse.}
For $s\in\{+,-\}$, let
\[
\Xi_s
 :=
([s,s;s,s];Z=\{s\}^{L})
\]
denote the corresponding exceptional state. Thus
\[
\Xi_+=([++;++];\alpha_Z),
\qquad
\Xi_-=([--;--];\omega_Z).
\]
Let $z_{\mathrm L}$ and $z_{\mathrm R}$ be the two distinct end
crossings of the residual twist adjacent to the left and right
clasps. Since $L\geq3$, these crossings are distinct. At the left
end, the displayed Reidemeister--III move slides
$z_{\mathrm L}$ across the seed strand, after which a
Reidemeister--I move removes it. The sign-reversed version of the same
relative-boundary move applies when $s=-$. Repeating the identical
sequence at the right end removes $z_{\mathrm R}$. Hence
\[
D(\Xi_s)
 \simeq
D\left([s,s;s,s];Z=\{s\}^{L-2}\right)
\]
relative to $\partial\Delta_s$.
\[
\includegraphics[width=0.3\textwidth]{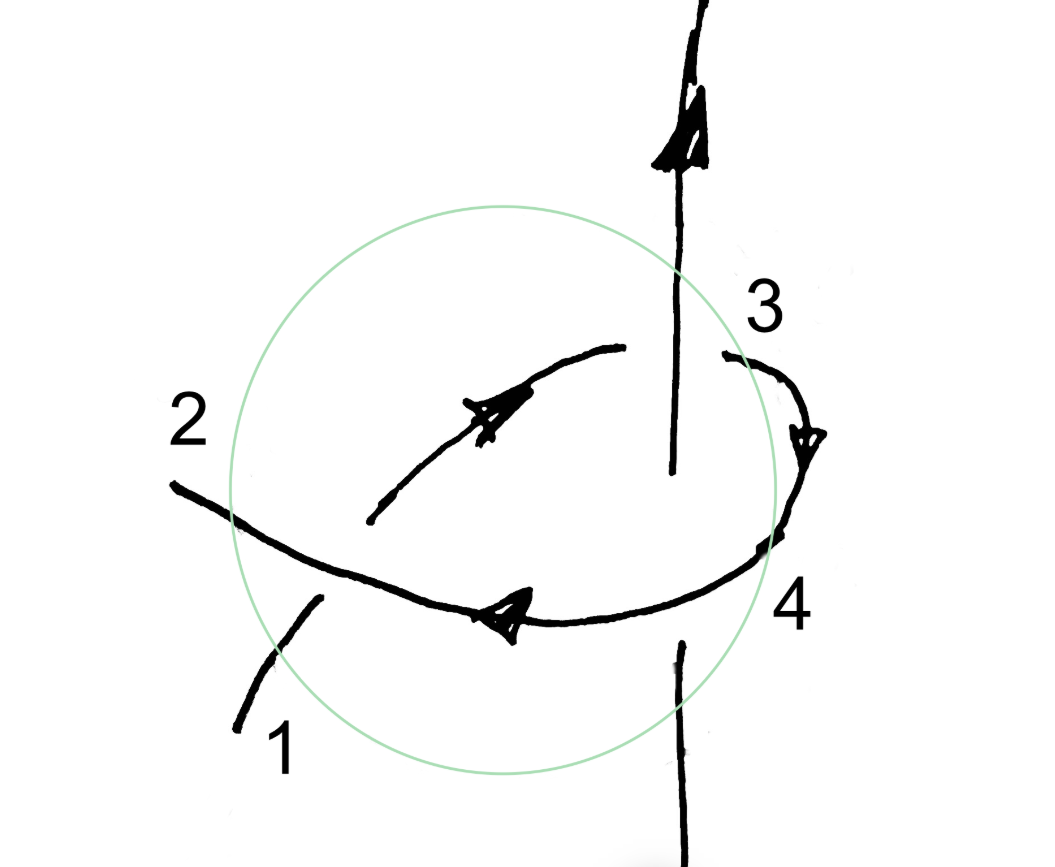}
\raisebox{1.7cm}{$\stackrel{R3}{\Longleftrightarrow}$} \quad
\raisebox{0.5cm}{
\includegraphics[width=0.23\textwidth]{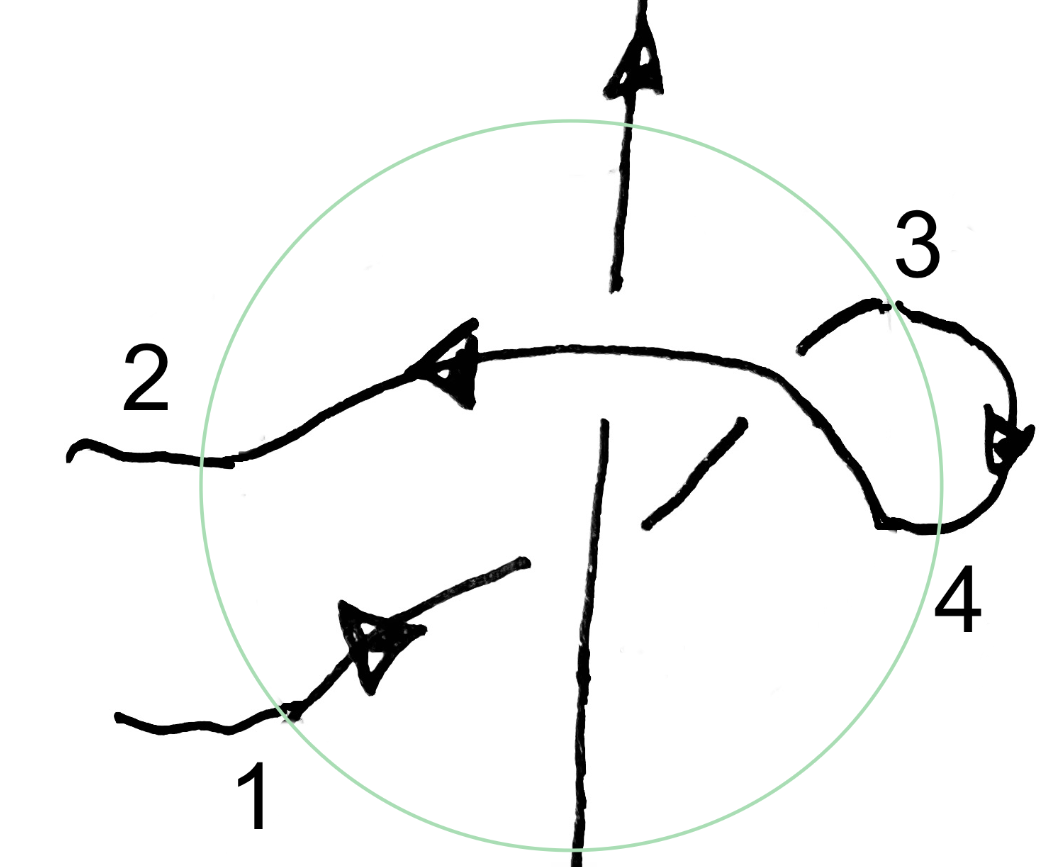}
}
\quad \raisebox{1.7cm}{$\stackrel{R1}{\Longleftrightarrow}$}
\includegraphics[width=0.3\textwidth]{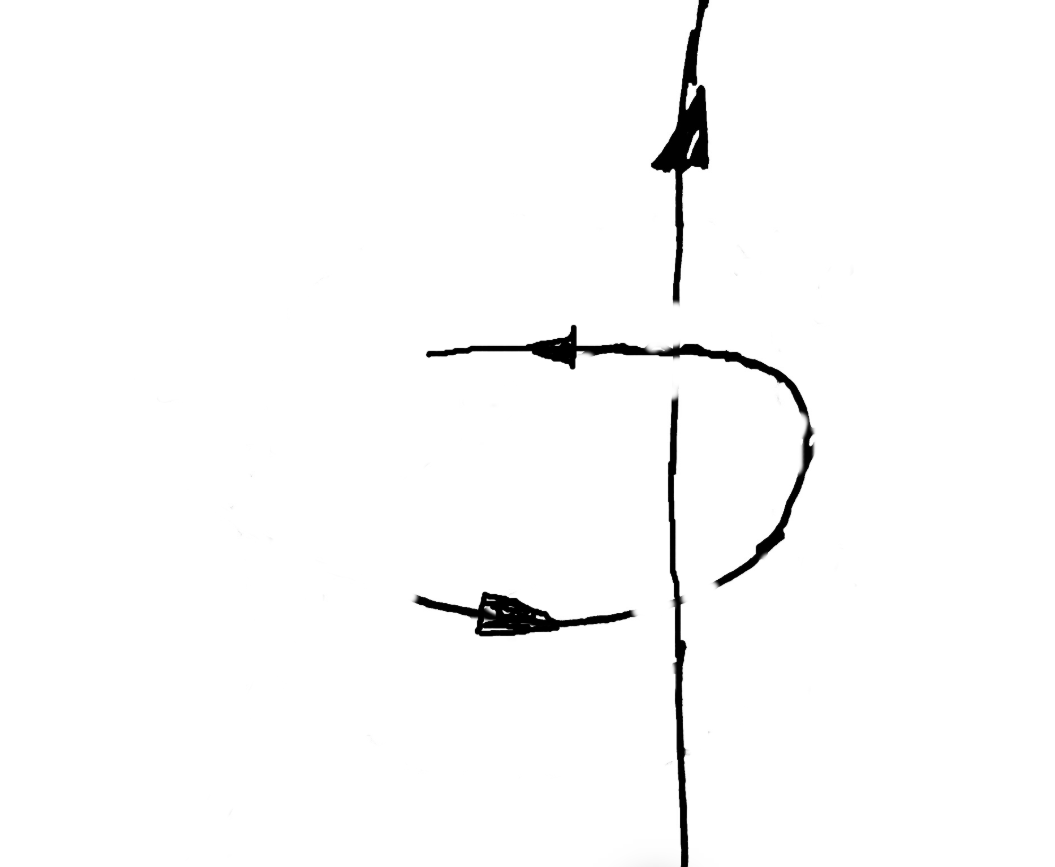}
\]
Now insert a canceling Reidemeister--II pair with signs
$(s,-s)$ inside the residual twist disk. This does not change the
isotopy class and restores a length-$L$ residual state:
\[
D\left([s,s;s,s];Z=\{s\}^{L-2}\right)
 \simeq
D\left([s,s;s,s];Z=Z_s^{\mathrm{ne}}\right),
\]
where, after reindexing the residual coordinates,
\[
Z_s^{\mathrm{ne}}
 =
(s,\ldots,s,-s)
\]
contains $L-1$ crossings of sign $s$ and one crossing of sign $-s$.
In particular,
\[
Z_s^{\mathrm{ne}}
 \notin
\{\alpha_Z,\omega_Z\}.
\]
Thus, each exceptional state is isotopic, relative to the planting
disk boundary, to a nonexceptional state:
\[
([++;++];\alpha_Z)
 \simeq
([++;++];Z_+^{\mathrm{ne}}),
\]
and
\[
([--;--];\omega_Z)
 \simeq
([--;--];Z_-^{\mathrm{ne}}).
\]
Therefore, whenever
\[
r>\widehat B,
\qquad
G_{r+1}(Q_{\mathcal M}^E),
\qquad
G_{r+2}(Q_{\mathcal M}^E),
\]
hypothesis \textnormal{(C1d)} gives
\[
f_r(E;[++;++];\alpha_Z)=0,
\qquad
f_r(E;[--;--];\omega_Z)=0.
\]
Hence, \textnormal{(C2)} holds.

An illustration of the conversion procedure for a twist with $L=3$ is provided below.
\[
\underbrace{
\raisebox{0.3cm}{
\includegraphics[width=0.40\textwidth]{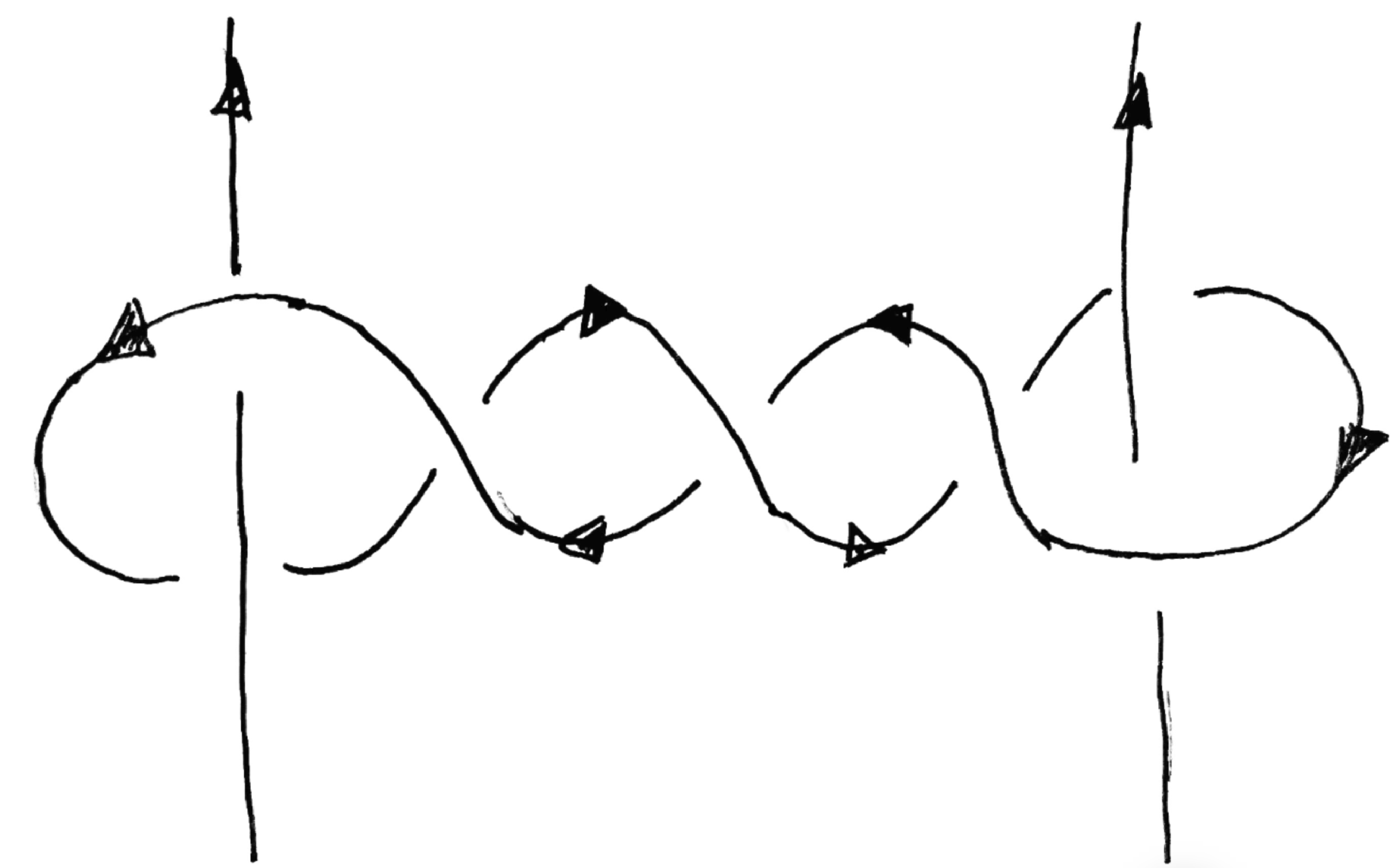}}}_{\text{\large $([s,s;s,s]; Z=(s,s,s))$}}\quad
\raisebox{2cm}{\Large $\stackrel{2\times R3}{\Longleftrightarrow}$} \quad
\includegraphics[width=0.54\textwidth]{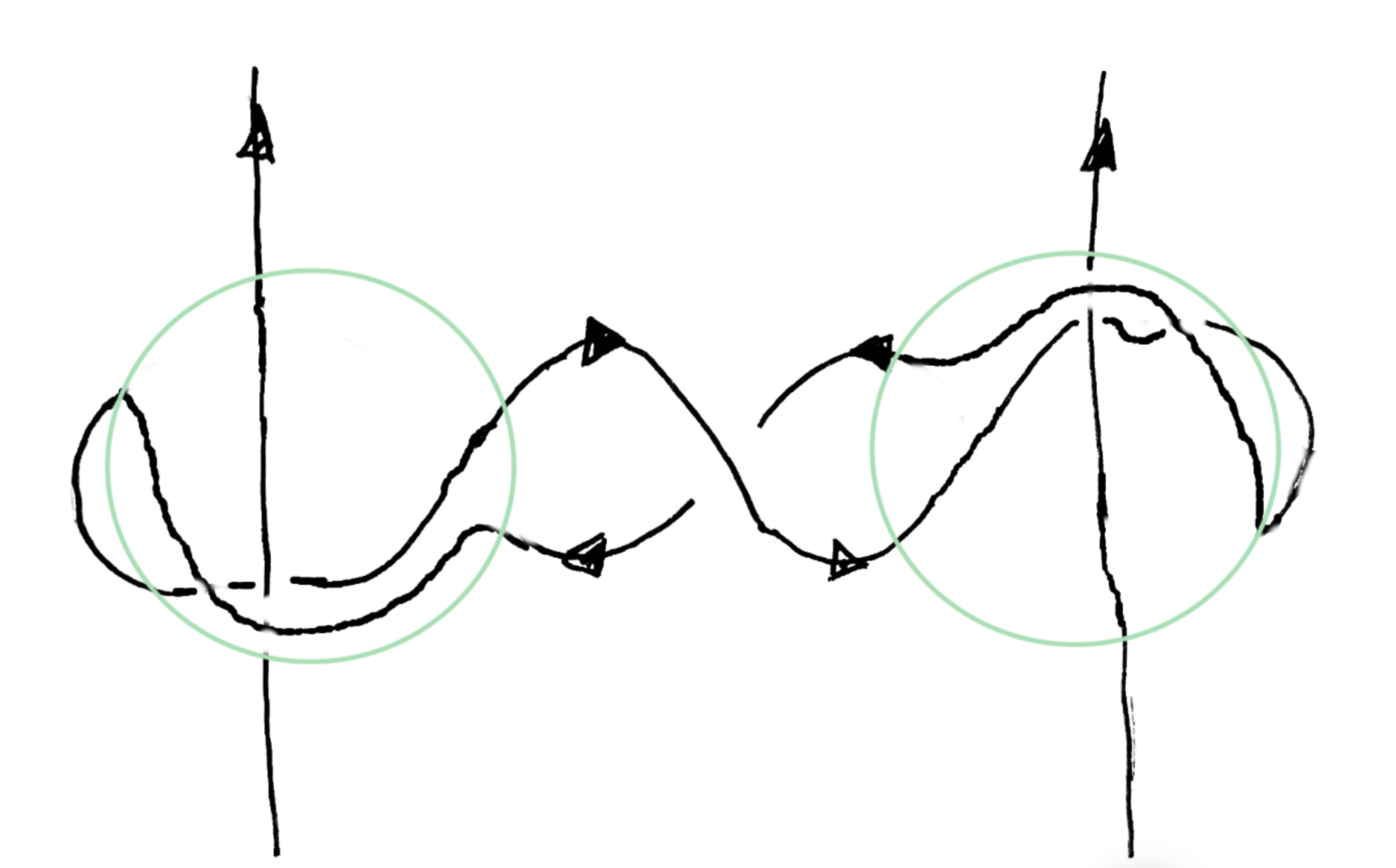}
\]
\[
\raisebox{0.3cm}{
\includegraphics[width=0.40\textwidth]{omega_3.png}}\quad
\raisebox{2cm}{\Large $\stackrel{2\times R1}{\Longleftrightarrow}$} \quad
\includegraphics[width=0.54\textwidth]{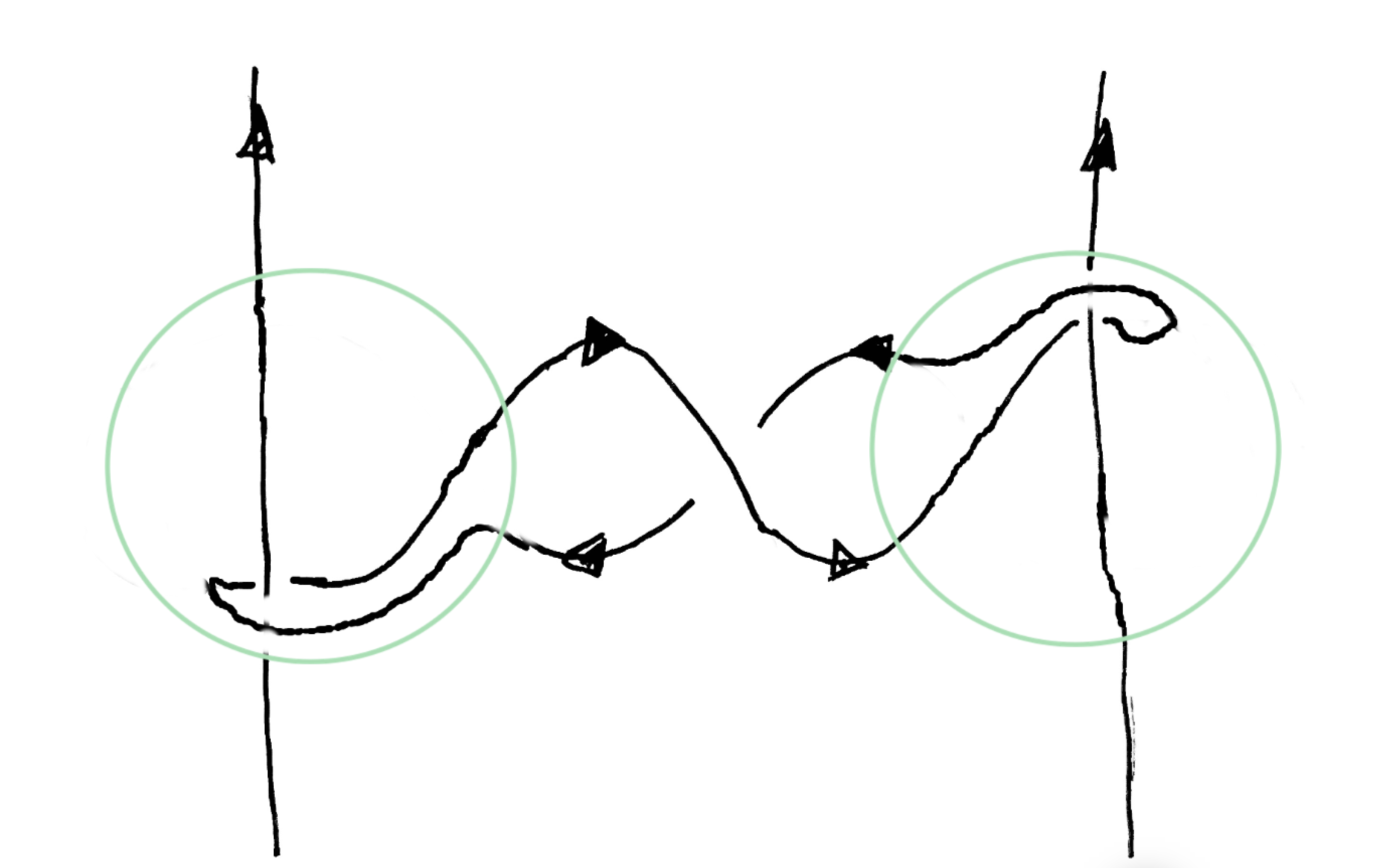}
\]
\[
\raisebox{0.3cm}{
\includegraphics[width=0.40\textwidth]{omega_4.png}}\quad
\raisebox{2cm}{\Large $\stackrel{R2}{\Longleftrightarrow}$} \quad
\underbrace{
\includegraphics[width=0.54\textwidth]{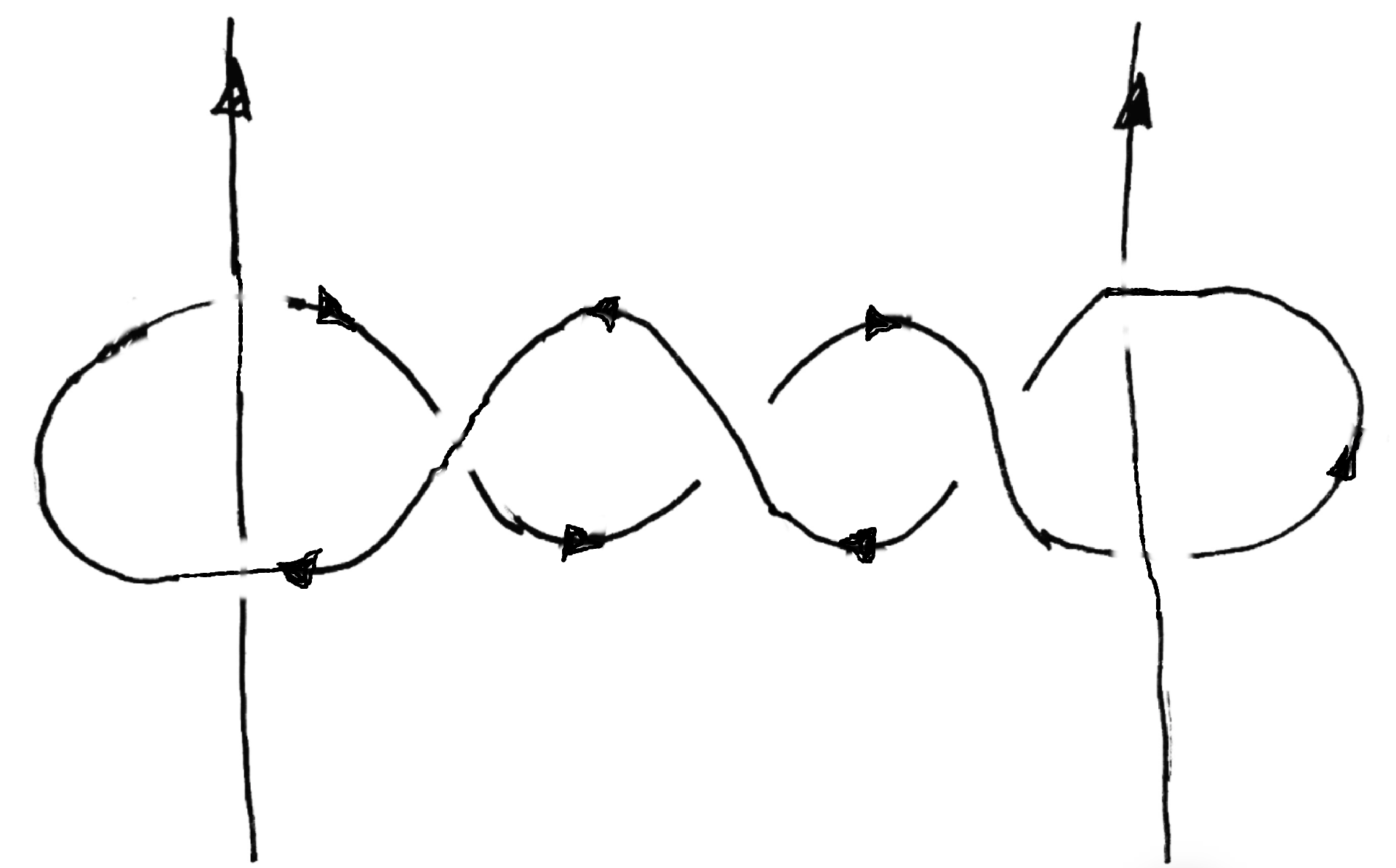}}_{\text{\large $([s,s;s,s]; Z=(-s,s,s))$}}
\]

\paragraph{$\bf (G3)$ Anti-parallel clasps.} Smoothing any crossing of a clasp pair $(y,y')$ renders the remaining crossing nugatory. This is illustrated for $y_0=0$ below.
\[
\raisebox{1.5cm}{$(y_0=0, y_0') \quad \cong$} \;
\includegraphics[width=0.2\textwidth]{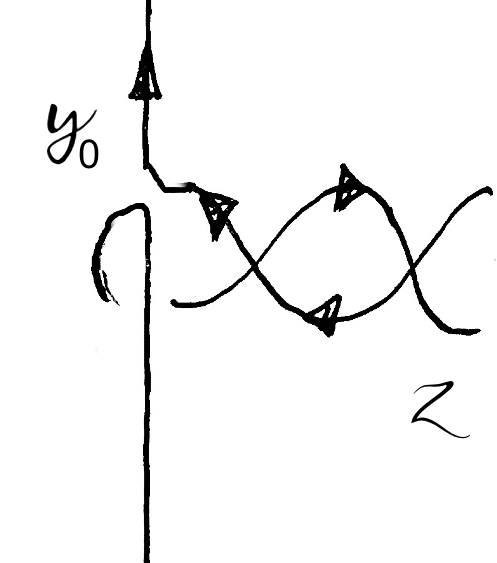}
\]
Here, $y_0'$ becomes nugatory. Hence, the anti-parallel condition,
\[
\partial_{y_0'} f_r(y_0=0)=0,
\]
which holds irrespective of crossing signs.

\paragraph{$\bf (G_4)$ Recovering the ordinary twist.} The residual diagram is obtained by smoothing one crossing per clasp, $y_0=0, y_{L+1}=0$. The planted diagram in this case reduces to
\[
\raisebox{1.5cm}{$\rho = (y_0=0, y_{L+1}=0) \quad \cong$} \;
\includegraphics[width=0.3\textwidth]{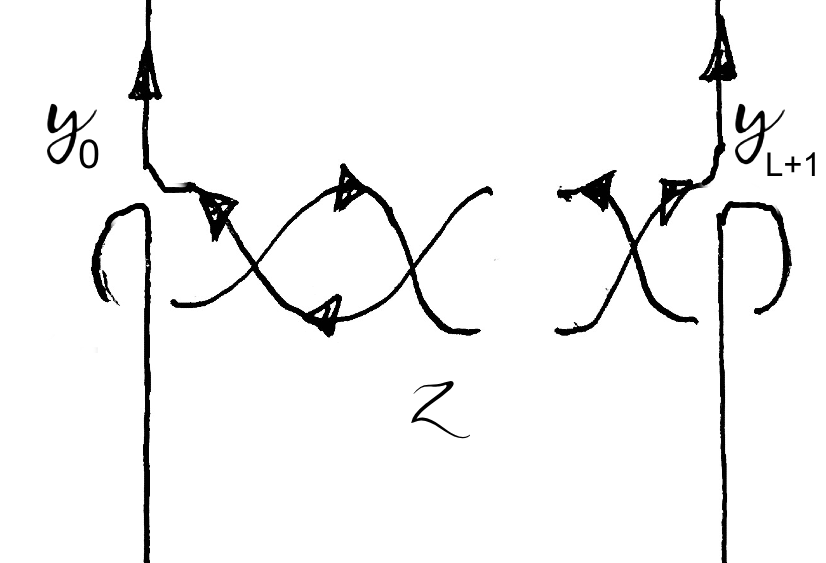}
\raisebox{1.5cm}{$\quad \cong$} \;
\includegraphics[width=0.3\textwidth]{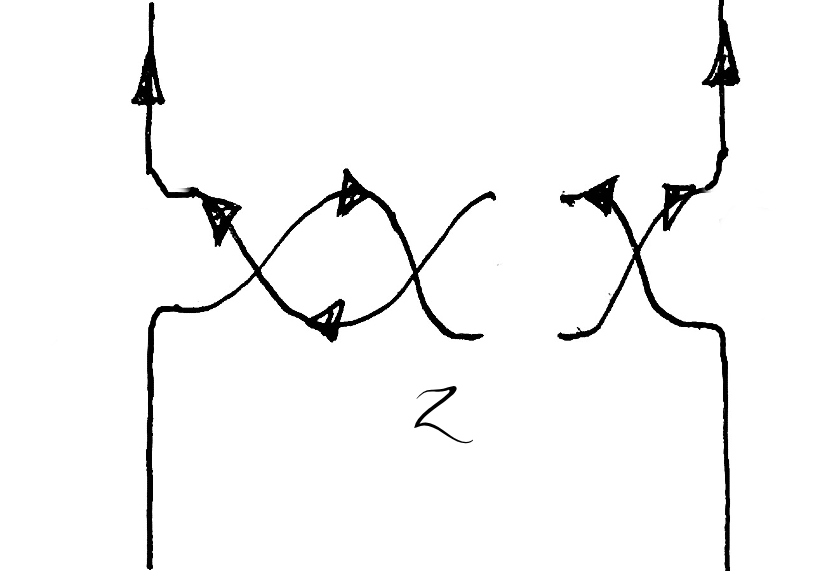}
\]
which is an ordinary anti-parallel twist planted instead of the original base crossing.

\paragraph{$\bf (G5),(G6)$ Residual sign states and faces.} 
Fix the Reidemeister-II pairing pattern shown in the diagram. The two chosen
states \(Z^+\) and \(Z^-\) are the states in which every paired crossing has
the opposite sign from its partner, while the single unpaired crossing survives
with positive, respectively negative, sign.
\[
\raisebox{1.5cm}{$(\rho,Z=Z^\pm) = \left(y_0=0, y_{L+1}=0, Z=Z^\pm\right) \quad \cong$} \;
\includegraphics[width=0.3\textwidth]{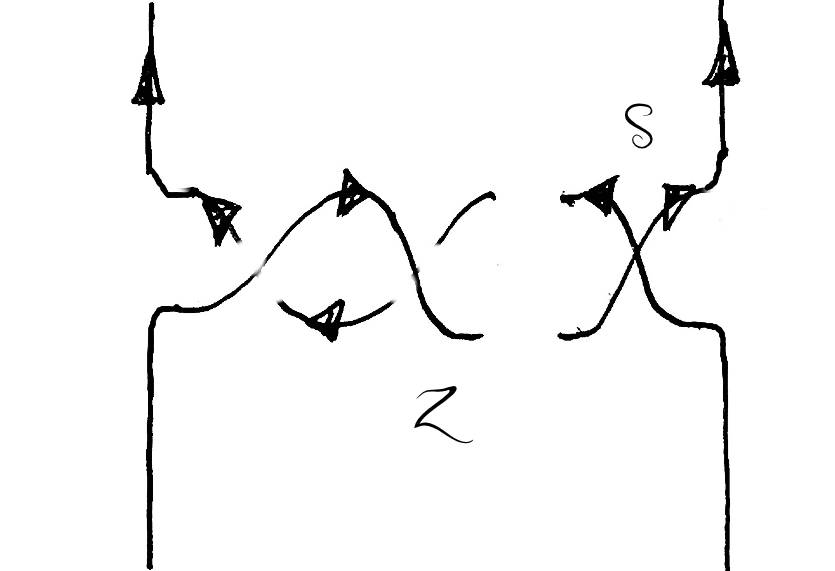}
\]
This leaves one effective crossing, thereby rendering the resulting diagram as either isotopic to the host or to a knot which is one crossing flip away from it. We let $Z^\circ$ denote the state whose effective crossing has the sign of the original seed crossing. Hence, $D(\rho,Z^\circ) \cong D$.

\paragraph{$\bf (G7)$ Smoothing proxy.} By the pairwise anti-parallel property of the twist, smoothing any crossing in $Z$ makes the remaining nugatory. The residual twist has an odd length, and so it preserves the local connectivity of the crossing disk. Therefore, smoothing any of its crossing results in a diagram that may be reduced by a sequence of $R1$ kinks to that of the host with its seed crossing smoothed.   
\[
\raisebox{1.5cm}{$(\rho,z_i=0) = (y_0=0, y_{L+1}=0, z_i=0) \quad \cong$} \;
\includegraphics[width=0.3\textwidth]{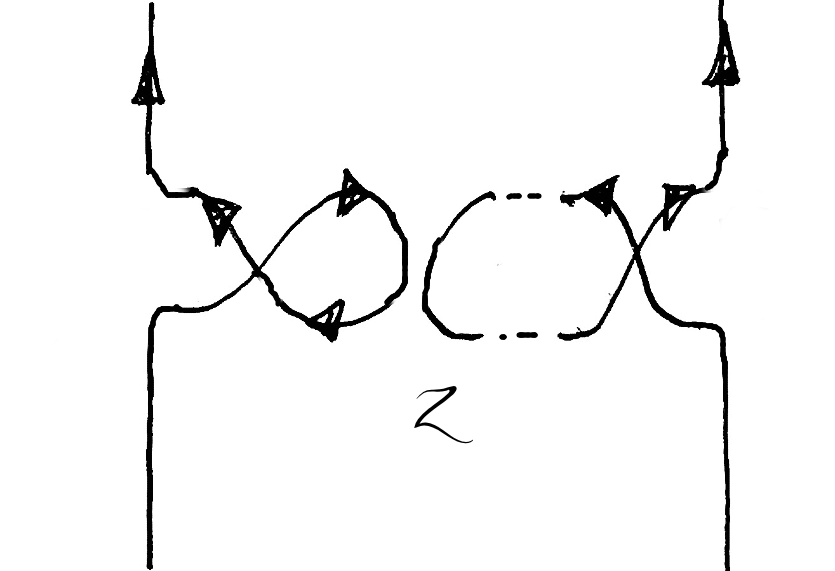}
\raisebox{1.5cm}{$\quad \cong$} \;
\includegraphics[width=0.17\textwidth]{G1_3.png}
\]
For \(Z=Z^\pm\), smoothing the effective crossing \(\tilde s\) after the
fixed Reidemeister-II cancellations is equivalent, after undoing those
cancellations, to smoothing the surviving residual crossing in the original
\(Z\)-block.

\paragraph{$\bf (G8)$ Machine locality.} All modifications remain confined to a local crossing disk. Multiple machines may be planted independently within isolated disks of the same host. Their states therefore underlie genuine crossing variation cubes. In particular, this applies to the residual sign faces $F_{Z_u}$, $F_{Z_v}$ and their product subcube $F_{Z_u} \times F_{Z_v}$.
\end{proof}

\fi

\newif\ifPairmachines
\Pairmachinesfalse

\ifPairmachines

\begin{proposition}[Single-machine residual descent over a host family]
Let \(\mathcal C\) be an admissible clasping twist machine of odd
residual length \(L\) over a host family \(\mathcal F(A)\), with base
cap \(B_{\mathcal F}\). Assume
\[
L>B_{\mathcal F}+4.
\]
Assume also that the host family vanishes through the lower range:
\[
f_q(\mathcal F(A))=0
\qquad
\text{for every }A\text{ and every }2\le q\le B_{\mathcal F}+4.
\]
Then
\[
G_q(Q_\rho(\mathcal F))
\qquad
\text{for every }q\ge2.
\]
\end{proposition}

\begin{proof}
By admissibility of the clasping machine, the Section 4 initialization
gives
\[
G_r(Q_\rho(\mathcal F))
\qquad
(r>B_{\mathcal F}+3).
\]
In particular,
\[
G_{B_{\mathcal F}+6}(Q_\rho(\mathcal F)),
\qquad
G_{B_{\mathcal F}+5}(Q_\rho(\mathcal F))
\]
are available.

We descend on
\[
q=B_{\mathcal F}+4,\ B_{\mathcal F}+3,\ldots,2.
\]
Assume inductively that
\[
G_{q+2}(Q_\rho(\mathcal F)),
\qquad
G_{q+1}(Q_\rho(\mathcal F))
\]
hold.

By the \(31\)-law in the residual anti-parallel twist cube,
\[
a_q
\]
is constant on the branch
\[
Z\neq\omega,
\]
and
\[
b_q
\]
is constant on the branch
\[
Z\neq\alpha.
\]

Let \(Z^\circ\) be the canonical odd residual anchor. Since
\[
Z^\circ\neq\omega,\qquad Z^\circ\neq\alpha,
\]
it lies in both non-exceptional branches. By construction,
\[
\mathcal C(A;Z^\circ,\rho)\cong \mathcal F(A).
\]
Hence, by the host-family vanishing hypothesis,
\[
f_q(\mathcal C(A;Z^\circ,\rho))
=
f_q(\mathcal F(A))
=
0.
\]
Therefore both non-exceptional branch constants vanish:
\[
a_q(A,Z,\rho)=0
\qquad
(Z\neq\omega),
\]
and
\[
b_q(A,Z,\rho)=0
\qquad
(Z\neq\alpha).
\]

Thus \(a_q\), as a function of the \(L\) residual twist coordinates,
may only be supported at \(Z=\omega\), and \(b_q\) may only be
supported at \(Z=\alpha\). Since
\[
q\le B_{\mathcal F}+4<L,
\]
degree-support uncertainty kills both possible point supports. Hence
\[
a_q\equiv0,\qquad b_q\equiv0
\]
on \(Q_\rho(\mathcal F)\).

This proves \(G_q(Q_\rho(\mathcal F))\). Descending to \(q=2\)
proves the proposition.
\end{proof}

\begin{theorem}[Asymmetric two-clasping residual descent]
Let \(K\) be a knot diagram and set
\[
B:=3c(K).
\]
Let \((u,v)\) be an ordered pair of crossings at which clasping twist
machines can be planted. Choose odd integers \(L_1,L_2\) satisfying
\[
L_1>B+4,
\]
and
\[
L_2>B+3L_1+4.
\]

Let \(\mathcal C_1\) be a clasping twist machine of residual length
\(L_1\) planted at \(u\). Denote its residual cube by
\[
Q_{\rho_1}^{(1)}
=
\{Z_1\text{ free},\ \rho_1\text{ fixed}\}.
\]
This is a family of diagrams which we write as
\[
D^{(1)}(Z_1,\rho_1).
\]

Over this singly planted residual family, plant a second clasping
twist machine \(\mathcal C_2\) of residual length \(L_2\) at \(v\).
Let
\[
Q_{\rho_1,\rho_2}^{(12)}
=
\{Z_1,Z_2\text{ free},\ \rho_1,\rho_2\text{ fixed}\}
\]
be the double residual cube, written as
\[
D^{(12)}(Z_1,Z_2,\rho_1,\rho_2).
\]

Assume:
\begin{enumerate}
    \item[(i)] The first machine \(\mathcal C_1\) is admissible over
    \(K\) with base cap
    \[
    B_1=B.
    \]

    \item[(ii)] The second machine \(\mathcal C_2\) is admissible over
    the residual family \(Q_{\rho_1}^{(1)}\) with safe base cap
    \[
    B_2:=B+3L_1.
    \]

    \item[(iii)] The base knot satisfies
    \[
    f_q(K)=0
    \qquad
    2\le q\le B+4.
    \]
\end{enumerate}

Then
\[
G_q(Q_{\rho_1,\rho_2}^{(12)})
\qquad
\text{for every }q\ge2.
\]
In particular,
\[
G_2(Q_{\rho_1,\rho_2}^{(12)}).
\]
\end{theorem}

\begin{proof}
The proof is deliberately asymmetric. We do not apply the one-machine
theorem sequentially inside the already double-planted diagram.
Instead, we first solve an auxiliary singly planted family and then use
that solved family as the anchor face for the second, longer machine.

\medskip

\noindent
\textbf{Stage 1: solve the auxiliary singly planted family.}

Apply the single-machine residual descent proposition to the first
clasping twist machine \(\mathcal C_1\).

The relevant base cap is
\[
B_1=B.
\]
The length hypothesis
\[
L_1>B+4
\]
holds by assumption. Since \(L_1\) is odd, the first machine has a
canonical residual anchor
\[
Z_1^\circ\neq\omega_1,\alpha_1
\]
such that
\[
D^{(1)}(Z_1^\circ,\rho_1)\cong K.
\]
By hypothesis,
\[
f_q(K)=0
\qquad
2\le q\le B+4.
\]
Therefore the proposition gives
\[
\boxed{
G_q(Q_{\rho_1}^{(1)})
\qquad
(q\ge2).
}
\tag{1}
\]

\medskip

\noindent
\textbf{Stage 2: identify the anchor face for the second machine.}

Inside the double residual family, let \(Z_2^\circ\) be the canonical
odd residual anchor of the second clasping twist machine. By the
construction of an odd-length clasping twist machine,
\[
Z_2^\circ\neq\omega_2,\alpha_2,
\]
and for every residual state \(Z_1\) of the first machine we have an
isotopy
\[
D^{(12)}(Z_1,Z_2^\circ,\rho_1,\rho_2)
\cong
D^{(1)}(Z_1,\rho_1).
\tag{2}
\]
Combining (1) and (2), we obtain
\[
\boxed{
f_q(D^{(12)}(Z_1,Z_2^\circ,\rho_1,\rho_2))=0
\qquad
\text{for every }Z_1\text{ and every }q\ge2.
}
\tag{3}
\]

This is the anchor family used for the descent of the second machine.

\medskip

\noindent
\textbf{Stage 3: initialize the second machine.}

Now work in the double residual cube
\[
Q_{\rho_1,\rho_2}^{(12)}.
\]
The first residual twist has length \(L_1\). Hence, uniformly over
\(Z_1\), the second machine sees a safe base cap
\[
B_2:=B+3L_1.
\]
Since
\[
L_2>B+3L_1+4=B_2+4,
\]
and since the second machine is admissible over the first residual
family with base cap \(B_2\), the Section 4 initialization gives
\[
\boxed{
G_r(Q_{\rho_1,\rho_2}^{(12)})
\qquad
(r>B_2+3).
}
\tag{4}
\]
In particular,
\[
G_{B_2+6}(Q_{\rho_1,\rho_2}^{(12)}),
\qquad
G_{B_2+5}(Q_{\rho_1,\rho_2}^{(12)})
\]
are available.

\medskip

\noindent
\textbf{Stage 4: descend the second machine.}

We descend on
\[
q=B_2+4,\ B_2+3,\ldots,2.
\]
Assume inductively that
\[
G_{q+2}(Q_{\rho_1,\rho_2}^{(12)}),
\qquad
G_{q+1}(Q_{\rho_1,\rho_2}^{(12)})
\]
hold.

Fix an arbitrary state \(Z_1\). Consider the residual \(Z_2\)-cube.
By the \(31\)-law for the second residual anti-parallel twist,
\[
a_q(Z_1,Z_2,\rho_1,\rho_2)
\]
is constant on
\[
Z_2\neq\omega_2,
\]
and
\[
b_q(Z_1,Z_2,\rho_1,\rho_2)
\]
is constant on
\[
Z_2\neq\alpha_2.
\]

The canonical anchor \(Z_2^\circ\) lies in both non-exceptional
branches. By (3),
\[
f_q(Z_1,Z_2^\circ,\rho_1,\rho_2)=0.
\]
Therefore both non-exceptional branch constants vanish:
\[
a_q(Z_1,Z_2,\rho_1,\rho_2)=0
\qquad
(Z_2\neq\omega_2),
\]
and
\[
b_q(Z_1,Z_2,\rho_1,\rho_2)=0
\qquad
(Z_2\neq\alpha_2).
\]

Thus, as a function of the \(L_2\) variables \(Z_2\), the coefficient
\(a_q\) may only be supported at \(Z_2=\omega_2\), and \(b_q\) may
only be supported at \(Z_2=\alpha_2\).

Since
\[
q\le B_2+4<L_2,
\]
degree-support uncertainty kills both possible point supports. Hence
\[
f_q(Z_1,Z_2,\rho_1,\rho_2)=0
\]
for every \(Z_2\). Since \(Z_1\) was arbitrary, this proves
\[
G_q(Q_{\rho_1,\rho_2}^{(12)}).
\]

Descending to \(q=2\) gives
\[
G_q(Q_{\rho_1,\rho_2}^{(12)})
\qquad
(q\ge2),
\]
and in particular
\[
G_2(Q_{\rho_1,\rho_2}^{(12)}).
\]
\end{proof}

\begin{remark}[Why the theorem is asymmetric]
The proof does not claim that the first machine can be descended with
base cap \(B\) inside the already double-planted diagram. If both
machines are present, the first machine would see the second machine as
part of the ambient diagram, and the cap \(B\) would no longer be
available.

Instead, the first machine is solved in the auxiliary singly planted
family \(D^{(1)}\). The canonical anchor of the second machine then
identifies a face of the double residual family with this singly
planted family:
\[
D^{(12)}(Z_1,Z_2^\circ,\rho_1,\rho_2)
\cong
D^{(1)}(Z_1,\rho_1).
\]
This solved face is the anchor used to descend the second, longer
machine. The enlarged cap
\[
B_2=B+3L_1
\]
is used only for the second machine.
\end{remark}

\begin{remark}[Use for an admissible pair]
In the application to an admissible ordered pair \((u,v)\), the first
machine is planted at \(u\), and the second is planted at the witness
crossing \(v\). The admissibility condition is not needed for the
formal descent itself; it is used later to identify the two-crossing
face in the double residual cube with the original pair \((u,v)\) and
to obtain the component-count contradiction.
\end{remark}

\begin{theorem}[Asymmetric descent with pair of machines]
Let $K$ be a knot diagram and set
\[
  B := 3\,c(K).
\]
Let $(u, v)$ be an ordered pair of crossings.
Choose odd integers
\[
  L_1 > B + 4,
\]
and
\[
  L_2 > B + 3L_1 + 4.
\]
Plant a clasping twist machine of residual length~$L_1$ at~$u$, and let
\[
  Q_{\rho_1} = \{Z_1 \text{ free};\; \rho_1 \text{ fixed}\}
\]
be its residual cube.
Then, over the residual family $Q_{\rho_1}$, plant a second clasping twist machine of residual length~$L_2$ at the admissible witness crossing~$v$, and let
\[
  Q_{\rho_1, \rho_2} = \{Z_1, Z_2 \text{ free};\; \rho_1, \rho_2 \text{ fixed}\}
\]
be the double residual cube.
Assume both planted machines satisfy the clasping hypotheses, with base caps:
\[
  B_1 := B
\]
for the first machine, and
\[
  B_2 := B + 3L_1
\]
for the second machine uniformly over the first residual cube.
Assume
\[
  f_r(K) = 0 \qquad 1 \le r \le B + 4.
\]
Then
\[
  G_2(Q_{\rho_1, \rho_2}).
\]
In fact,
\[
  G_q(Q_{\rho_1, \rho_2}) \qquad (q \ge 2).
\]
\end{theorem}

 
\begin{proof}
First plant the $u$-machine and pass to its residual cube. Over this residual family, plant the $v$-machine.
\paragraph{Stage 1: solve the first, shorter machine.}
Apply Theorem~\ref{thm:finite-type-descent} to the first machine. The base cap is
\[
  B = 3\,c(K).
\]
Since
\[
  L_1 > B + 4
\]
and $L_1$ is odd, the first residual cube has a canonical base-knot anchor $Z_1^\circ$ with
\[
  D(Z_1^\circ, \rho_1) \cong K.
\]
Because
\[
  f_r(K) = 0 \qquad 1 \le r \le B + 4,
\]
Theorem~\ref{thm:finite-type-descent} gives
\[
  G_q(Q_{\rho_1}) \qquad (q \ge 2). \tag{1}
\]
That is, after smoothing the first clasping pair, the entire residual $Z_1$-cube is already killed at every level needed. This is the prepared anchor for the second machine.
 
\paragraph{Stage 2: initialize the second machine over the first residual cube.}
Now treat the second machine while $Z_1$ is arbitrary in the already solved residual cube.
Because the first residual twist has length $L_1$, the second machine sees a base family with cap
\[
  B_2 := B + 3L_1.
\]
This is the conservative cap: the arbitrary $Z_1$-state may contribute up to $L_1$ extra crossings, hence at most $3L_1$ extra $p$-degree by Corollary~\ref{cor:p-degree-cap}. By Lemma~\ref{lem:vanish-B2} the second machine with base cap $B_2$, we obtain
\[
  G_r(Q_{\rho_1, \rho_2}) \qquad (r > B_2 + 3).
\]
In particular, since
\[
  B_2 + 5 > B_2 + 3, \qquad B_2 + 6 > B_2 + 3,
\]
we have the two layers needed to start the lower descent:
\[
  G_{B_2+6}(Q_{\rho_1, \rho_2}), \qquad G_{B_2+5}(Q_{\rho_1, \rho_2}).
\]
 
\paragraph{Stage 3: descend the second machine using the first machine as anchor.}
We now descend on
\[
  q = B_2 + 4,\; B_2 + 3,\; \ldots,\; 2.
\]
Assume inductively that
\[
  G_{q+2}(Q_{\rho_1, \rho_2}), \qquad G_{q+1}(Q_{\rho_1, \rho_2})
\]
hold. Fix an arbitrary state $Z_1$. Consider the $Z_2$-cube.
Since $Q_{\rho_2}$ is an ordinary anti-parallel residual twist cube, Lemma~\ref{lem:31-law} gives:
\[
  a_q(Z_1, Z_2) \text{ is constant on } Z_2 \neq \omega_2,
\]
and
\[
  b_q(Z_1, Z_2) \text{ is constant on } Z_2 \neq \alpha_2.
\]
The canonical anchor $Z_2^\circ$ of the second residual twist satisfies, by construction,
\[
  D(Z_1, Z_2^\circ, \rho_1, \rho_2) \cong D(Z_1, \rho_1).
\]
By~(1),
\[
  f_q(D(Z_1, \rho_1)) = 0.
\]
Therefore
\[
  f_q(Z_1, Z_2^\circ, \rho_1, \rho_2) = 0.
\]
Since $Z_2^\circ$ is non-extreme, this kills both non-exceptional constants in the $Z_2$-cube
\[
  a_q(Z_1, Z_2) = 0 \qquad (Z_2 \neq \omega_2),
\]
and
\[
  b_q(Z_1, Z_2) = 0 \qquad (Z_2 \neq \alpha_2).
\]
Thus $a_q$ may only be supported at
\[
  Z_2 = \omega_2,
\]
and $b_q$ may only be supported at
\[
  Z_2 = \alpha_2.
\]
As these are one-point supports we have
\[
a_q(Z_1,\omega_2) = (-1)^L \partial_{Z_2} a_q(Z_1,Z_2), \qquad b_q(Z_1,\alpha_2) = \partial_{Z_2} b_q(Z_1,Z_2).
\]
Since
\[
  q \le B_2 + 4 < L_2 = |Z_2|,
\]
both point supports vanish by degree/type considerations. Hence
\[
  f_q(Z_1, Z_2, \rho_1,\rho_2) = 0
\]
for all $Z_2$. Since $Z_1$ was arbitrary,
\[
  G_q(Q_{\rho_1, \rho_2})
\]
holds. Descending to $q = 2$, we get
\[
  G_2(Q_{\rho_1, \rho_2}).
\]
\end{proof}

\fi

\section{Triviality barriers}
\label{sec:triviality-barriers}

\subsection{Non-nugatoriness and admissibility}

\begin{define}[Nugatory crossing~\cite{tait1898}]
\label{def:nugatory-cut-vertex}
Let $D$ be a knot diagram with shadow $S$, and let $u$ be a crossing of $D$. We call $u$ \emph{nugatory} if there exists a Jordan curve in the projection plane meeting $D$ only at~$u$. Otherwise $u$ is \emph{non-nugatory}.
\end{define}

\begin{lemma}[Shadow connectivity criterion]
\label{lem:shadow-connectivity}
Let $D$ be a knot diagram with shadow $S$, and let $|S|$ denote the
underlying embedded topological graph. A crossing $u$ is nugatory if
and only if it is a topological cut point of $|S|$. Equivalently,
\[
u\text{ is non-nugatory}
\quad\Longleftrightarrow\quad
|S|\setminus\{u\}\text{ is connected}.
\]
\end{lemma}

\begin{proof}
Suppose first that $u$ is nugatory. Then there is a Jordan curve in
the projection plane meeting $|S|$ only at $u$, with nonempty portions
of the shadow on both sides. After deleting the point $u$, no path in
$|S|\setminus\{u\}$ can cross this Jordan curve. Hence
$|S|\setminus\{u\}$ is disconnected.

Conversely, suppose that $|S|\setminus\{u\}$ is disconnected.
Choose one of its connected components. By planar separation, a
sufficiently small regular neighborhood of that component has a
boundary component which, after a local adjustment near $u$, is a
Jordan curve meeting the shadow only at $u$. Thus $u$ is nugatory.
\end{proof}

 

\begin{define}[Admissible witness; admissible pair]
\label{def:admissible}
Let $D$ be a knot diagram and let $u$ be a non-nugatory crossing. Since oriented smoothing of a crossing of a knot splits the unique component into two, write
\[
D\setminus\{u\}=L_A\sqcup L_B.
\]
A crossing $v\neq u$ is called an \emph{admissible witness} for $u$ if, in the smoothed diagram $D\setminus\{u\}$, the crossing $v$ lies between the two components $L_A$ and $L_B$. In that case $(u,v)$ is called an \emph{admissible pair}.

\end{define}

 

\begin{lemma}[Existence of admissible witnesses]
\label{lem:admissible}
Every non-nugatory crossing~$u$ of a knot diagram~$D$ admits an admissible
witness~$v$.
\end{lemma}

 \begin{proof}
Let $S$ be the shadow of $D$. Since $u$ is non-nugatory,
Lemma~\ref{lem:shadow-connectivity} gives
\[
|S|\setminus\{u\}
\quad\text{connected}.
\]
In the oriented smoothing
\[
D\setminus\{u\}=L_A\sqcup L_B,
\]
every edge of $|S|\setminus\{u\}$ lies on exactly one of the two
components $L_A$ or $L_B$.
Choose an edge $e_A\subset L_A$ and an edge $e_B\subset L_B$.
Because $|S|\setminus\{u\}$ is connected, there is an edge path
\[
e_0,v_1,e_1,\ldots,v_m,e_m
\]
in $|S|\setminus\{u\}$ with
\[
e_0=e_A,
\qquad
e_m=e_B.
\]

 
Let $t$ be the smallest index such that $e_t\subset L_A$ and
$e_{t+1}\subset L_B$.  Then the intervening vertex $v_{t+1}$ is incident to
one branch lying on~$L_A$ and one branch lying on~$L_B$.  Hence, in the link
diagram $D\setminus\{u\}$, the crossing
\[
  v := v_{t+1}
\]
lies between the two components $L_A$ and~$L_B$.  Therefore $v$ is an
admissible witness for~$u$.
\end{proof}

\begin{lemma}[Intercomponent smoothing merges components]
\label{lem:merge}
Let $L$ be an oriented link diagram, and let $z$ be a crossing between two
distinct components $A$ and~$B$ of~$L$.  Then oriented smoothing at~$z$
merges those two components:
\[
  \#\pi_0(L\setminus\{z\}) = \#\pi_0(L)-1.
\]
\end{lemma}
 
\begin{proof}
At the crossing~$z$, each of $A$ and~$B$ contributes one incoming and one
outgoing half-edge.  In each of the two oriented local models for an
inter-component crossing, the oriented smoothing pairs an incoming half-edge
of~$A$ with an outgoing half-edge of~$B$, and an incoming half-edge of~$B$
with an outgoing half-edge of~$A$.  Thus the smoothing splices $A$ and~$B$
into a single oriented component.  No other component is affected, so the
component count drops by one.
\end{proof}

\begin{lemma}[Signed smoothing slide in an anti-parallel twist]
\label{lem:smooth-slide}
Let \(a,b\) be two adjacent crossings in an oriented anti-parallel two-strand
twist. Then, relative to the boundary of the twist disk, smoothing one crossing
and keeping the other is isotopic to smoothing the other crossing, with the
surviving sign transported across the anti-parallel pair. In the sign
convention of Section 3,
\[
        D(a=0,b=\sigma)\cong D(a=-\sigma,b=0),
        \qquad \sigma\in\{+,-\}.
\]
The analogous identity holds with \(a\) and \(b\) interchanged.
Consequently, a smoothing point can be slid through an anti-parallel twist,
provided one records the induced sign transport on the crossings it passes.
\end{lemma}
\begin{proof}
This is precisely the local anti-parallel nugatory identity from Section 3.
After smoothing \(a\), the crossing \(b\) lies in a local monogon and is
nugatory. Removing that monogon by a Reidemeister-I move gives the same
relative-boundary tangle as smoothing \(b\) and keeping \(a\) with the transported
opposite sign. The argument with \(a\) and \(b\) interchanged is identical.
Repeating the local move slides the smoothing point through any finite
anti-parallel twist.
\end{proof}



\begin{lemma}[Smoothing proxy for the residual sign face]
\label{lem:smooth-proxy}
Let a clasping twist machine be planted at a seed crossing \(s\), and let
\[
        F_Z=\{Z^+,Z^-\}
\]
be its residual sign face. Let \(\tilde s\) denote the effective crossing
realized by \(F_Z\) after the fixed Reidemeister-II cancellations.
Then the oriented smoothing of \(\tilde s\) is isotopic, relative to the exterior
of the planting disk, to the oriented smoothing of the original seed crossing:
\[
        D_{\mathrm{ext}}(\tilde s=0)\cong D(s=0).
\]
If \(t\) is any crossing outside the planting disk, then
\[
        D_{\mathrm{ext}}(\tilde s=0,t=0)\cong D(s=0,t=0).
\]
For two disjoint planted machines at seed crossings \(u\) and \(v\), with
effective residual crossings \(\tilde u,\tilde v\), one has
\[
        D_{\mathrm{ext}}(\tilde u=0,\tilde v=\sigma)
        \cong
        D(u=0,v=\sigma),
        \qquad \sigma\in\{+,-\},
\]
and
\[
        D_{\mathrm{ext}}(\tilde u=0,\tilde v=0)
        \cong
        D(u=0,v=0).
\]
\end{lemma}
\begin{proof}
By Proposition~\ref{prop:clasping-twist-machines}, the two vertices of \(F_Z\) cancel by fixed
Reidemeister-II moves to a single effective crossing \(\tilde s\). Smoothing
that surviving crossing and then undoing the cancellation description gives a
smoothing inside the original residual anti-parallel twist. By the signed
smoothing-slide Lemma~\ref{lem:smooth-slide}, this smoothing may be transported through the twist
to the original seed crossing \(s\). Hence
\[
        D_{\mathrm{ext}}(\tilde s=0)\cong D(s=0)
\]
relative to the exterior.
All isotopies are supported inside the planting disk. Therefore they commute
with smoothing any exterior crossing \(t\). The two-machine statements follow
because the two planting disks are disjoint.
\end{proof}

\subsection{Bound on $J_N$-triviality}

Corollary~\ref{cor:p-degree-cap} sets an algebraic degree cap for the JVP of any link. For knots, the finite-type coefficients $f_q$ vanish beyond $3c(K)$. The next theorem recovers this uniform bound from a geometric perspective for the class of $J_N$-trivial knots, i.e., knots whose JVP coefficients $f_q = 0$ for $1 \leq q \leq N-1$ (Definition~\ref{def:jm-trivial}). A nontrivial knot with $N > 3c(K)$ leads to impossible component counts following the smoothing of admissible pairs.

\begin{theorem}[Triviality barrier]
\label{thm:JN-triviality-barrier}
Let \(K\) be a nontrivial knot, and let \(c(K)\) denote its crossing
number. If \(K\) is \(J_N\)-trivial, then
\[
N\le 3c(K).
\]
\end{theorem}

\begin{proof}
Let
\[
B:=3c(K).
\]
Suppose, for contradiction, that
\[
N>B.
\]
Since \(K\) is \(J_N\)-trivial, we have
\[
f_r(K)=0
\qquad
1\le r<N.
\]
In particular,
\[
f_r(K)=0
\qquad
1\le r\le B.
\]
On the other hand, by the crossing-number degree cap Corollary~\ref{cor:p-degree-cap},
\[
f_r(K)=0
\qquad
r>B.
\]
Hence $f_r(K)=0$ for all $r\ge1$.
In particular,
\begin{equation}
\label{eq:s592}
f_r(K)=0
\qquad
1\le r\le \widehat B+2,
\end{equation}
with $\widehat B = B+\kappa_{\mathrm{mach}}$.
Choose a crossing-minimal diagram \(D\) of \(K\). Since \(K\) is
nontrivial, \(D\) has at least one non-nugatory crossing \(u\). By
Lemma~\ref{lem:admissible}, \(u\) admits an admissible witness \(v\).
%
Thus smoothing \(u\) separates the diagram into two
components, and after smoothing \(u\), the crossing \(v\) is an
intercomponent crossing. Consequently smoothing \(v\) after smoothing
\(u\) merges the two components (Lemma~\ref{lem:merge}).

\paragraph{Step 1. Facewise descent with two clasping twist machines.}
The geometry of the clasping twist machine renders $\kappa_{\mathrm{mach}} = 11$. Hence, $\widehat B = B+11 = 3c(K)+11$. Choose odd integers
\[
        L_u>\widehat B+2,\qquad L_v>\widehat B+2.
\]
Plant clasping twist machines of residual lengths \(L_u\) and \(L_v\) at the
crossings \(u\) and \(v\), respectively, in disjoint disks. Let
\[
        F_u=\{\tilde u=+,\tilde u=-\}
\]
and
\[
        F_v=\{\tilde v=+,\tilde v=-\}
\]
be the residual sign faces, identified by Proposition~\ref{prop:clasping-twist-machines} with genuine
effective crossing coordinates \(\tilde u\) and \(\tilde v\).
By Theorem~\ref{thm:joint-descent2}, applied with the anchor vanishing \eqref{eq:s592}, we
obtain
\[
        G_q(F_u\times F_v) \qquad (q\ge 2).
\]
In particular,
\[
        G_2(F_u\times F_v).
\]

\paragraph{Step 2. Deriving component count contradiction.}
Let \(D_{\mathrm{ext}}\) denote the extended diagram family carried by the
two-dimensional effective crossing face
\[
        F_u\times F_v
        =
        \{\tilde u=\pm,\tilde v=\pm\}.
\]
By Proposition~\ref{prop:clasping-twist-machines}, this is a genuine two-crossing variation cube. Its four signed vertices are knot diagrams, since they differ from the original diagram
only by crossing signs. From
\[
        G_2(F_u\times F_v)
\]
we have
\[
        a_2=b_2=0,\qquad c_2:=a_2+b_2=0
\]
on this face. Therefore,
\begin{equation}
\label{eq:s595}
\partial_{\widetilde u}\partial_{\widetilde v}c_2=0.
\end{equation}
We now apply the symbolic component-count recursion Lemma~\ref{lem:BLsymskein}.
Since every signed vertex of \(F_u\times F_v\) is a knot, by Lemma~\ref{JVPfreec} and Lemma~\ref{lem:f1-vanishing}
\[
c_0\equiv1,
\qquad
c_1\equiv0
\]
on this signed face.
Apply the order-one recursion at \(\widetilde u\):
\begin{equation}
\label{eq:s59-1}
\partial_{\widetilde u}c_1
=
2c_0+c_0(\widetilde u=0).
\end{equation}
Since $c_1\equiv0$ on the signed face, we have
\[
\partial_{\widetilde u}c_1=0.
\]
Also $c_0=1$, so from \eqref{eq:s59-1}
\begin{equation}
\label{eq:s596}
c_0(\widetilde u=0)=-2.
\end{equation}
By the component-count interpretation of \(c_0\), Lemma~\ref{JVPfreec} this means that for each \(\sigma\in\{+,-\}\),
\[
        D_{\mathrm{ext}}(\tilde u=0,\tilde v=\sigma)
\]
has two components.
Next apply the order-two recursion at \(\widetilde u\), and then take
the \(\widetilde v\)-difference:
\begin{equation}
\label{eq:s59-2}
\partial_{\widetilde u}\partial_{\widetilde v}c_2
=
2\partial_{\widetilde v}c_1
+
\partial_{\widetilde v}c_1(\widetilde u=0).
\end{equation}
By \eqref{eq:s595}, $\partial_{\widetilde u}\partial_{\widetilde v}c_2=0$.
Also, because $c_1\equiv0$ on the signed knot face, $\partial_{\widetilde v}c_1=0$.
Hence \eqref{eq:s59-2} reads
\begin{equation}
\label{eq:s597}
\partial_{\widetilde v}c_1(\widetilde u=0)=0.
\end{equation}
Now apply the order-one recursion at \(\widetilde v\) after smoothing
\(\widetilde u\):
\begin{equation}
\label{eq:s59-3}
\partial_{\widetilde v}c_1(\widetilde u=0)
=
2c_0(\widetilde u=0)
+
c_0(\widetilde u=0,\widetilde v=0).
\end{equation}
Using \eqref{eq:s596} and \eqref{eq:s597}, we get from \eqref{eq:s59-3}
\[
0
=
2(-2)
+
c_0(\widetilde u=0,\widetilde v=0).
\]
Thus
\begin{equation}
\label{eq:s598}
c_0(\widetilde u=0,\widetilde v=0)=4.
\end{equation}
Again by the component-count interpretation of \(c_0\), Lemma~\ref{JVPfreec} this means that
\[
D_{\mathrm{ext}}(\widetilde u=0,\widetilde v=0)
\]
has three components.

It remains to identify these effective smoothings with the original crossings
\(u\) and \(v\). By the smoothing-proxy Lemma~\ref{lem:smooth-proxy},
\[
        D_{\mathrm{ext}}(\tilde u=0,\tilde v=\sigma)
        \cong
        D(u=0,v=\sigma),
        \qquad \sigma\in\{+,-\},
\]
and
\[
        D_{\mathrm{ext}}(\tilde u=0,\tilde v=0)
        \cong
        D(u=0,v=0).
\]
Taking \(\sigma\) to be the original sign of \(v\), equation \eqref{eq:s596} says that
\[
        D\setminus\{u\}
\]
has two components, while equation \eqref{eq:s598} says that
\[
        D\setminus\{u,v\}
\]
has three components.
This contradicts admissibility of the pair \((u,v)\). Indeed, by construction,
after smoothing \(u\), the crossing \(v\) lies between the two components of
\(D\setminus\{u\}\). Therefore Lemma~\ref{lem:merge} says that smoothing \(v\) must merge
those two components, so $D\setminus\{u,v\}$ must have one component, not three. The contradiction proves that the assumption \(N>3c(K)\) was impossible. Hence
\[
        N\le 3c(K).
\]

\end{proof}

\subsection{Unknot detection}


\begin{theorem}[Jones unknot conjecture]
\label{thm:detection}
Let $K$ be a knot. If
\[
V_K(x)=1,
\]
then $K$ is the unknot.
\end{theorem}

\begin{proof}
Suppose, toward a contradiction, that $K$ is nontrivial and that
\[
V_K(x)=1.
\]
By uniqueness of the Jones--Vassiliev expansion,
\[
V_K(x,p)
 =
\sum_{q\ge0}\bigl(a_q(K)+b_q(K)x\bigr)p^q,
\]
the equality $V_K=1$ implies
\[
a_q(K)=b_q(K)=0
\qquad(q\ge1).
\]
Hence $K$ is $J_N$-trivial for every positive integer $N$.
Choose
\[
N=3c(K)+1.
\]
Then $K$ is $J_N$-trivial, while Theorem~\ref{thm:JN-triviality-barrier}, applied to the
nontrivial knot $K$, asserts
\[
N\le 3c(K).
\]
This contradicts $N=3c(K)+1$.
Therefore $K$ cannot be nontrivial, and hence $K$ is the unknot.
\end{proof}

\medskip

\begin{corollary}[Vassiliev--Goussarov unknot-detection conjecture]
Let \(K\) be a knot. Suppose that \(K\) is \(m\)-trivial in the
sense of Goussarov~\cite{goussarov1994} for every positive integer \(m\). Then \(K\) is
the unknot.
\end{corollary}

\begin{proof}
Let \(U\) denote the unknot. By Goussarov's \(m\)-equivalence
theorem, if \(K\) is \(m\)-trivial, then every Vassiliev invariant of
order strictly less than \(m\) takes the same value on \(K\) and on
\(U\) \cite{goussarov1994}.

Fix \(q\geq1\) and choose \(m=q+1\). By Lemma~\ref{JVP}, the
Jones--Vassiliev coefficients \(a_q\) and \(b_q\) are finite-type
invariants of order at most \(q\). Hence
\[
a_q(K)=a_q(U),
\qquad
b_q(K)=b_q(U).
\]
Since the Jones--Vassiliev polynomial of the unknot is \(1\),
\[
a_q(U)=b_q(U)=0
\qquad(q\geq1).
\]
Therefore
\[
a_q(K)=b_q(K)=0
\qquad(q\geq1).
\]
Equivalently, \(K\) is \(J_N\)-trivial for every positive integer
\(N\).
Suppose that \(K\) were nontrivial. Taking
\[
N=3c(K)+1,
\]
Theorem~\ref{thm:JN-triviality-barrier} would give
\[
N\leq3c(K),
\]
contradicting the choice of \(N\). Thus \(K\) is the unknot.
\end{proof}

\paragraph{The strength of finite-type invariants.}
This consequence is stronger than the assertion that the totality of
finite-type invariants detects the unknot: the proof requires only the
particular finite-type invariants arising as the Jones--Vassiliev
coefficients.  Indeed, Theorem~\ref{thm:JN-triviality-barrier} shows quantitatively that every
nontrivial knot \(K\) has a nonzero Jones--Vassiliev coefficient of
order at most \(3c(K)\).  It follows in particular that any universal
finite-type invariant over a field of characteristic zero, such as the
normalized Kontsevich integral \cite{kontsevich1993}, detects the
unknot.  If its value on \(K\) were equal to its value on the unknot,
universality would force every finite-type invariant---and hence every
Jones--Vassiliev coefficient---to agree on the two knots, which by the
corollary is possible only when \(K\) is the unknot.  This establishes
unknot detection; it does not by itself assert that finite-type
invariants distinguish every pair of non-isotopic knots.

\newpage

\appendix

\addcontentsline{toc}{section}{Appendices}
{\centering\Large\bfseries Auxiliary lemmas and proofs \par}
\vspace{1em}

\section{Spectral representation of pseudo-Boolean functions}

\paragraph{Fourier basis and Walsh characters.}
Let $f:\{\pm1\}^{n}\to\mathbb{C}$ be a pseudo-Boolean function. For every subset $S\subseteq[n]$, the Walsh character is given by,
\begin{equation}
\label{eq:chi}
\chi_{S}(\varepsilon) = \prod_{i\in S}\varepsilon_{i}\quad (\chi_{\varnothing}\equiv1).
\end{equation}
These characters form an orthonormal basis (with respect to the uniform measure on $\cQ_{n}$), so
\begin{equation}
\label{eq:fexpansion}
f(\varepsilon) = \sum_{S\subseteq[n]}\,\widehat{f}(S)\,\chi_{S}(\varepsilon), \qquad \widehat{f}(S)=2^{-n}\sum_{\varepsilon}f(\varepsilon)\,\chi_{S}(\varepsilon).
\end{equation}
is the Walsh-Fourier expansion of $f$.
The polynomial degree, $\deg(f)$, matches the largest $|S|$ with $\widehat{f}(S)\neq0$.

\paragraph{Discrete normalized derivatives.}
For coordinate $i$ we define the (normalized) derivative
\begin{equation}
\label{eq:normdiff}
(\partial_{i}f)(\varepsilon) \equiv \frac{1}{2}\bigl(f(\varepsilon_i^+)-f(\varepsilon_i^-)\bigr),
\end{equation}
where $\varepsilon_i^\pm$ is $\varepsilon$ with the $i$-th entry set to $\pm 1$. For a subset $T\subseteq[n]$,
\begin{equation}
\partial_{T} = \prod_{i\in T}\partial_{i}, \qquad (\partial_{T}f)(\varepsilon) = 2^{-|T|}\!\!\!\sum_{\sigma\in\{\pm1\}^{T}}\!\! (-1)^{\#\{\sigma_{i}=-1\}}\, f(\varepsilon_{[n]\setminus T},\sigma).
\end{equation}
Equivalently, it is the alternating sum of $f$ over the $T$-subcube through $\varepsilon$.
For any fixed subset $T\subseteq[n]=\{1,\ldots,n\}$ the $T$-fold discrete derivative acts on Walsh characters by ``dropping'' the coordinates in $T$:
\begin{equation}
\partial_{T}\,\chi_{S} = \begin{cases} 
\chi_{S\setminus T}, & T\subseteq S,\\[4pt] 
0, & T\not\subseteq S.
\end{cases}
\end{equation}
Because the characters $\{\chi_{S}\}_{S\subseteq[n]}$ form an orthonormal basis, we obtain the Fourier (Walsh) expansion of the derivative simply by applying this rule term-wise to the expansion of $f$:
\begin{equation}
\label{eq:therule}
(\partial_{T} f)(\varepsilon) = \sum_{[n]\supseteq U\supseteq T} \widehat{f}(U) \chi_{U\setminus T}(\varepsilon).
\end{equation}
Thus, the Fourier coefficient of $\partial_{T}f$ at character $\chi_{S}$ (with $S\cap T=\varnothing$) is precisely the coefficient of $f$ at the larger character $\chi_{S\cup T}$, namely, $\widehat{\partial_{T}f}(S)=\widehat{f}(S\cup T)$.
This explicit form makes transparent how derivatives shift weight ``down'' the spectrum and why $\partial_{T}f$ vanishes whenever $f$ has no Fourier support on sets containing $T$ (the finite-type condition).

\begin{remark}[Normalized vs. unnormalized derivatives]
    The standard Fourier/Walsh framework is formulated using the normalized difference operator \eqref{eq:normdiff}. The Vassiliev skein, on the other hand, adopts the unnormalized version, namely \eqref{eq:unnormdiff},
    \begin{equation}
        (\partial_{i}f)(\varepsilon) \equiv f(\varepsilon_i^+)-f(\varepsilon_i^-).
    \end{equation}
    At times, this convention alleviates the burden of dealing with factors such as $2^{-m}$ when taking derivatives over cube $m$-faces. For that reason and because it being consistent with the standard theory, unless otherwise stated, $\partial_i$ means \eqref{eq:unnormdiff} for the rest of this work. 
\end{remark}


\subsection{Degree-support uncertainty}

\begin{lemma}[Degree-support uncertainty on the Boolean cube]
\label{lem:uncertainty-bound}
Let $Q_n=\{\pm1\}^n$ and let $f:Q_n\to \mathbb C$ be a pseudo-Boolean function. Write its Walsh expansion as
\[
f(\varepsilon)=\sum_{S\subseteq[n]}\widehat f(S)\chi_S(\varepsilon),
\qquad
\chi_S(\varepsilon)=\prod_{i\in S}\varepsilon_i.
\]
Let
\[
\deg(f):=\max\{|S|:\widehat f(S)\neq0\}.
\]
If $f\not\equiv0$, then
\[
|\operatorname{supp} f|\ge 2^{n-\deg(f)}.
\]
Consequently, if
$|\operatorname{supp} f|\le m$,
then either $f\equiv0$, or
\begin{equation}
\label{eq:uncertainty-bound}
\deg(f)\ge n-\lfloor \log_2 m\rfloor .
\end{equation}
For a finite-type coefficient $f_q$ of order $\le q$ on an $n$-dimensional variation cube, this gives
\[
|\operatorname{supp} f_q|\le m,\quad q<n-\lfloor\log_2m\rfloor
\quad\Longrightarrow\quad
f_q\equiv0.
\]
\end{lemma}
 
\begin{proof}
It is enough to prove the equivalent statement:
\[
\deg(f)\le d,\quad f\not\equiv0
\quad\Longrightarrow\quad
|\operatorname{supp} f|\ge 2^{n-d}.
\]
We argue by induction on $n$.
For $n=0$, the cube consists of one vertex, and the claim is immediate.
Assume $n\ge1$. Write a vertex as
\[
(\eta,t)\in\{\pm1\}^{n-1}\times\{\pm1\}.
\]
Define the two face restrictions
\[
f_+(\eta):=f(\eta,+1),
\qquad
f_-(\eta):=f(\eta,-1).
\]
Then
\[
|\operatorname{supp} f|
=
|\operatorname{supp} f_+|
+
|\operatorname{supp} f_-|.
\]
If $d=n$, the claim is trivial since $f\not\equiv0$ implies
$|\operatorname{supp} f|\ge1=2^{n-n}$.
 
So assume $d<n$. First suppose both $f_+$ and $f_-$ are nonzero. Restricting to a face cannot increase Walsh degree, hence
\[
\deg(f_+)\le d,
\qquad
\deg(f_-)\le d.
\]
By induction on the $(n{-}1)$-cube,
\[
|\operatorname{supp} f_+|
\ge
2^{(n-1)-d},
\]
and
\[
|\operatorname{supp} f_-|
\ge
2^{(n-1)-d}.
\]
Therefore
\[
|\operatorname{supp} f|
\ge
2^{(n-1)-d}+2^{(n-1)-d}
=
2^{n-d}.
\]
Now suppose exactly one restriction is nonzero. Without loss of generality,
\[
f_-\equiv0,
\qquad
f_+\not\equiv0.
\]
Then
\[
f(\eta,t)=\frac{1+t}{2}\,f_+(\eta).
\]
Thus
\[
\deg(f)=\deg(f_+)+1.
\]
Indeed, multiplying by $(1+t)/2$ creates a top term $t\chi_S(\eta)$ from every top Walsh character $\chi_S$ appearing in $f_+$, and this top term cannot cancel with any term not involving $t$.
Hence, if $\deg(f)\le d$, then
$\deg(f_+)\le d-1$.
By induction,
\[
|\operatorname{supp} f|
=
|\operatorname{supp} f_+|
\ge
2^{(n-1)-(d-1)}
=
2^{n-d}.
\]
The case $f_+\equiv0$, $f_-\not\equiv0$, is identical, using
$f(\eta,t)=\frac{1-t}{2}\,f_-(\eta)$.
Therefore every nonzero degree-${\le d}$ pseudo-Boolean function has support at least
$2^{n-d}$.
Taking $d=\deg(f)$, we obtain
\[
|\operatorname{supp} f|\ge 2^{n-\deg(f)}.
\]
 
Now assume
$|\operatorname{supp} f|\le m$
and $f\not\equiv0$. Then
\[
m\ge |\operatorname{supp} f|\ge 2^{n-\deg(f)}.
\]
Taking logarithms gives
$\deg(f)\ge n-\log_2 m$.
Since $\deg(f)$ is an integer,
\[
\deg(f)\ge n-\lfloor\log_2 m\rfloor.
\]
This proves the lemma.
\end{proof}

\begin{remark}[Sharpness]
The bound \eqref{eq:uncertainty-bound} is sharp. Let
\[
k=\lfloor\log_2 m\rfloor,
\qquad
2^k\le m<2^{k+1}.
\]
Consider the indicator of a $k$-dimensional face ($\delta$ function for the all-positive vertex):
\[
g(\varepsilon)
=
2^{-(n-k)}
\prod_{i=k+1}^{n}(1+\varepsilon_i).
\]
Then
$|\operatorname{supp} g|=2^k\le m$,
and
$\deg(g)=n-k=n-\lfloor\log_2 m\rfloor$.
So no stronger universal lower bound depending only on $|\operatorname{supp} f|\le m$ is possible.
\end{remark}

\section{Jones--Vassiliev polynomial}

\subsection{Proof of Lemma~\ref{JVP}}\label{Proof_JVP}
\begin{proof}
    The Jones polynomial lives in $\mathbb{Z}[x,x^{-1}]$. Its skein relation is given by
    \begin{equation}
        \label{eq:sk1}
        x^{-2} V^{(+)} - x^2 V^{(-)} = (x-x^{-1}) V^{(0)},
    \end{equation}
    where the superscripts designate the variation of a single crossing as over, under, and 
    annihilation through oriented smoothing. There exists an isomorphism, $\mathbb{Z}[x,x^{-1}] \cong \mathbb{Z}[p][1,x] \cong \mathbb{Z}[p][x]/(x^2-px-1)$. In particular, letting $p = x - x^{-1}$, we obtain
    \begin{equation}
        \label{eq:x2}
        x^2 = px + 1, \; \; x^{-2} = p^2 - px + 1.
    \end{equation}
    \emph{In $R = \mathbb{Z}[p][x]/(x^2 - px - 1)$ we reduce every product via $x^2 = px + 1$; all identities below are taken in $R$.}
    Substituting \eqref{eq:x2} into the skein relation \eqref{eq:sk1} and rearranging yield the Vassiliev skein relation,
    \begin{equation}
        \label{Vas}
        \partial V \equiv V^{(+)} - V^{(-)} = p \left[ x V^{(-)} + V^{(0)} + (x-p) V^{(+)} \right].
    \end{equation}
    
    The above definitions imply that the zero modulo in the underlying ring is $x^2-px-1$, while $p$ may be of any degree. Because $x^2-px-1=0$ is monic, $\mathbb{Z}[p][x]/(x^2-px-1)$ is free of rank $2$ over $\mathbb{Z}[p]$ with basis $\{1,x\}$, and thus may be written as
    \begin{equation}
        \label{eq:theJVP}
        V_K(x,p) = \sum_{q \geq 0} \left[a_q(K) + b_q(K)x\right] p^q. 
    \end{equation}
    where the coefficients $a_q(K),b_q(K)$ are well defined and unique. These coefficients are $q$-type Vassiliev invariants. This follows from the fact that repeated application of the Vassiliev skein relation \eqref{Vas} increases the lowest degree of $p$. In particular,
    \begin{equation}
        \label{eq:modp}
        \partial^n V_K \equiv 0 \mod p^n,
    \end{equation}
    i.e., the $p^q$ coefficients of $\partial^n V_K$ vanish for every $q < n$. 
\end{proof}

\begin{proposition}[Freeness and uniqueness for the JVP expansion]
\label{prop:freeness}
Let $R=\mathbb{Z}[p]$ and consider the $R$-algebra $R[x]/(x^2-px-1)$. Then
$R[x]/(x^2-px-1)$ is a free $R$-module of rank $2$ with basis $\{1,x\}$. In particular,
every Laurent polynomial $V(x)\in \mathbb{Z}[x,x^{-1}]$ admits a unique expansion
\[
V(x)=\sum_{q\ge 0}\big(a_q+b_q\,x\big)\,p^q\qquad(a_q,b_q\in\mathbb{Z}),
\]
when viewed in $R[x]/(x^2-px-1)$ via the homomorphism $p\mapsto x-x^{-1}$.
\end{proposition}

\begin{proof}
Because $x^2-px-1$ is monic of degree $2$, reduction modulo this polynomial shows
$R[x]/(x^2-px-1)$ is generated over $R$ by $\{1,x\}$. Suppose $F(p)+G(p)\,x=0$ in
$R[x]/(x^2-px-1)$. Passing to the fraction field $K=\mathbb{Q}(p)$, the image of $x$ in
$K[x]/(x^2-px-1)$ has minimal polynomial $t^2-pt-1$, hence $\{1,x\}$ is $K$-linearly
independent. Therefore, $F=G=0$ in $K[p]$, and since the module is torsion-free over $R$,
it follows $F,G\in R$ vanish. Uniqueness of the coefficients $\{a_q,b_q\}$ follows.
\end{proof}

\subsection{Proof of Lemma~\ref{JVPfreec}}

\begin{proof}
    The Jones polynomial of a two-component unlink is $V_{\unl} = -x-x^{-1}$, and the corresponding JVP is $V_{\unl} = p - 2x$. It is well known that evaluating the Jones polynomial at $x=1$ yields $V_{\unl}(x=1)^{l_K - 1} = (-2)^{l_K - 1}$, where $l_K$ is the number of link components. As $p = x - x^{-1}$ we have $p=0$ once $x=1$. Therefore, evaluating the Jones polynomial at $1$ amounts to evaluating the JVP at $x=1, \, p=0$, namely, the number of link components are encoded exclusively by the free coefficient. Note, however, that $p$ vanishes also for $x=-1$, in which case $V_{\unl}(x=-1,p=0) = 2$, from which we conclude that $V_{\unl}(x=\pm 1, p=0) = \mp 2$, and,
    \begin{equation}
        V_K(x=\pm1, p=0) = V_{\unl}(x=\pm 1, p=0)^{l_K - 1} = (\mp 2)^{l_K - 1}.
    \end{equation}
    Therefore, $V_K(x=\pm1, p=0) = a_0(K) \pm b_0(K) = (\mp 2)^{l_K - 1}$. Adding and subtracting the two identities for $x=1$ and $x=-1$ yields the stated result. 
\end{proof}

\subsection{Proof of Lemma~\ref{JVPskeins}}
\begin{proof}
For brevity write
\[
V^-:=V(\varepsilon_i^-),\qquad V^0:=V(\varepsilon_i^0),\qquad V^+:=V(\varepsilon_i^+),
\]
and similarly
\[
a_r^\sigma:=a_r(\varepsilon_i^\sigma),\qquad b_r^\sigma:=b_r(\varepsilon_i^\sigma),
\qquad \sigma\in\{-,0,+\}.
\]
By the polynomial-level JVP skein relation \eqref{Vas},
\[
\partial_i V
= p\bigl(xV^-+V^0+(x-p)V^+\bigr)
\]
in the quotient ring \(R[x]/(x^2-px-1)\). Expand each of the three terms in the JVP basis:
\[
V^\sigma=\sum_{r\ge 0}(a_r^\sigma+b_r^\sigma x)p^r,\qquad \sigma\in\{-,0,+\}.
\]
Using the relation \(x^2=px+1\), we compute, for each \(r\ge 0\),
\[
p\,x(a_r^-+b_r^-x)p^r
= p(a_r^-x+b_r^-x^2)p^r
= p(a_r^-x+b_r^-(px+1))p^r
= (b_r^-+a_r^-x)p^{r+1}+b_r^-x\,p^{r+2},
\]
\[
p(a_r^0+b_r^0x)p^r
= (a_r^0+b_r^0x)p^{r+1},
\]
and
\[
\begin{aligned}
p(x-p)(a_r^+ + b_r^+x)p^r
&= p\,x(a_r^+ + b_r^+x)p^r-(a_r^+ + b_r^+x)p^{r+2} \\
&= \bigl((b_r^+ + a_r^+x)p^{r+1}+b_r^+x\,p^{r+2}\bigr)
   -(a_r^+ + b_r^+x)p^{r+2} \\
&= (b_r^+ + a_r^+x)p^{r+1}-a_r^+p^{r+2}.
\end{aligned}
\]
Therefore
\[
p\,xV^-
=\sum_{q\ge 0}\Bigl(b_{q-1}^-+(a_{q-1}^-+b_{q-2}^-)x\Bigr)p^q,
\]
\[
p\,V^0
=\sum_{q\ge 0}\Bigl(a_{q-1}^0+b_{q-1}^0x\Bigr)p^q,
\]
and
\[
p(x-p)V^+
=\sum_{q\ge 0}\Bigl(b_{q-1}^+-a_{q-2}^+ + a_{q-1}^+x\Bigr)p^q,
\]
where, as throughout the paper, we use the convention \(a_q=b_q\equiv 0\) for \(q<0\).

On the left-hand side,
\[
\partial_iV
=V(\varepsilon_i^+)-V(\varepsilon_i^-)
=\sum_{q\ge 0}\bigl((\partial_i a_q)+(\partial_i b_q)x\bigr)p^q.
\]
Comparing the coefficients of \(1\) and \(x\) in degree \(p^q\) on both sides gives
\[
\partial_i a_q
= b_{q-1}^- + a_{q-1}^0 + b_{q-1}^+ - a_{q-2}^+,
\]
\[
\partial_i b_q
= a_{q-1}^- + b_{q-1}^0 + a_{q-1}^+ + b_{q-2}^-.
\]
Reinstating the crossing-state notation,
\[
\partial_i a_q
= b_{q-1}(\varepsilon_i^-)+a_{q-1}(\varepsilon_i^0)+b_{q-1}(\varepsilon_i^+)-a_{q-2}(\varepsilon_i^+),
\]
\[
\partial_i b_q
= a_{q-1}(\varepsilon_i^-)+b_{q-1}(\varepsilon_i^0)+a_{q-1}(\varepsilon_i^+)+b_{q-2}(\varepsilon_i^-),
\]
which is exactly the finite-type recursion claimed in \eqref{eq:skein}.
\end{proof}

\subsection{Proof of Lemma~\ref{lem:JVP-mirror}}
\begin{proof}
Fix a cube state \(\eta\). The mirror of the diagram \(D(\eta)\) is obtained by reversing every
crossing sign, hence it corresponds to the state \(\tau\eta\) on the original shadow. By the standard mirror symmetry of the Jones polynomial,
\[
V_{\overline{D(\eta)}}(x)=V_{D(\eta)}(x^{-1}) = V_{D(\tau \eta)}(x).
\]
Therefore, one can change variables $x \mapsto x^{-1}$ ($x \mapsto x-p$ and $p \mapsto -p$ in the JVP's ring) or apply $\tau$, in which case $(+) \leftrightarrow (-)$. The equivalence between either of these approaches is apparent from the Jones skein \eqref{eq:sk1}
\[
x^{-2} V^{(+)} - x^2 V^{(-)} = (x-x^{-1}) V^{(0)},
\]
and the JVP's skein \eqref{Vas}
\[
\partial V \equiv V^{(+)} - V^{(-)} = p \left[ x V^{(-)} + V^{(0)} + (x-p) V^{(+)} \right].
\]
Clearly, swapping $(+) \leftrightarrow (-)$ alongside the variable transformation does essentially nothing, as we are back with the primal skein. 

In the JVP ring \(R[x]/(x^2-px-1)\) one has
\[
x^{-1}=x-p,
\]
because \(x(x-p)=x^2-px=1\). Since \(p=x-x^{-1}\), the substitution \(x\mapsto x^{-1}\) sends
\(p\) to \(-p\). Therefore, pointwise on the cube,
\[
V_{\overline D}(\eta;x,p)=V_D(\eta;x-p,-p).
\]
Now expand the right-hand side in JVP form:
\[
V_D(\eta;x-p,-p)
=\sum_{q\ge 0}\bigl(a_q(\eta)+b_q(\eta)(x-p)\bigr)(-p)^q.
\]
Distribute and separate the \(x\)-part:
\[
\begin{aligned}
V_D(\eta;x-p,-p)
&=\sum_{q\ge 0}(-1)^q\bigl(a_q(\eta)+b_q(\eta)x\bigr)p^q
 -\sum_{q\ge 0}(-1)^q b_q(\eta)p^{q+1}.
\end{aligned}
\]
Reindex the second sum (\(q\mapsto q-1\)) and use the convention \(b_{-1}\equiv 0\):
\[
V_D(\eta;x-p,-p)
=
\sum_{q\ge 0}
\Bigl(
(-1)^q\bigl(a_q(\eta)+b_{q-1}(\eta)\bigr)
+(-1)^q b_q(\eta)x
\Bigr)p^q.
\]
On the other hand,
\[
V_{\overline D}(\eta;x,p)=\sum_{q\ge 0}\bigl(a_q(-\eta)+ b_q(-\eta)x\bigr)p^q.
\]
By uniqueness of the \(\{1,x\}\)-expansion, the coefficients must agree:
\[
a_q(-\eta)=(-1)^q\bigl(a_q(\eta)+b_{q-1}(\eta)\bigr),
\qquad
b_q(-\eta)=(-1)^q b_q(\eta).
\]
This proves the lemma.
\end{proof}

\section{Symbolic calculus}

\subsection{Proof of Lemma~\ref{lem:BLsymskein}}

\begin{proof}
Let $c_q \equiv a_q + b_q$, where $a_q$ and $b_q$ are the JVP coefficients. From \eqref{eq:skein} it follows
    \begin{equation}
        \label{eq:cskein}
        \partial_i c_q = c_q(\eps_i^+) - c_q(\eps_i^-) = c_{q-1}(\eps_i^-) + c_{q-1}(\eps_i^0) + c_{q-1}(\eps_i^+) - a_{q-2}(\eps_i^+) + b_{q-2}(\eps_i^-).
    \end{equation}
    Let $J \subseteq \cros(\cQ)$ be a $q$-subset of crossing indices, $|J|=q$. The symbol skein relation follows by taking derivative, $\partial_{J \setminus \{i\}}$, on both sides of \eqref{eq:cskein},
    \begin{equation}
        \label{eq:symsk}
        \partial_J c_q = \partial_{J \setminus \{i\}} c_{q-1}(\eps_i^-) + \partial_{J \setminus \{i\}} c_{q-1}(\eps_i^0) + \partial_{J \setminus \{i\}} c_{q-1}(\eps_i^+) - \partial_{J \setminus \{i\}} a_{q-2}(\eps_i^+) + \partial_{J \setminus \{i\}} b_{q-2}(\eps_i^-).
    \end{equation}
    Invoking the type condition, $\partial_T f = \text{const}$, for any order-$|T|$ invariant $f$,
    $$\partial_{J \setminus \{i\}} c_{q-1}(\eps_i^+) = \partial_{J \setminus \{i\}} c_{q-1}(\eps_i^-) = \partial_{J \setminus \{i\}} c_{q-1}.$$ 
    By the same argument, $\partial_{J \setminus \{i\}} a_{q-2}(\eps_i^+) = \partial_{J \setminus \{i\}} b_{q-2}(\eps_i^-) = 0$. Note, however, that $\partial_{J \setminus \{i\}} c_{q-1}(\eps_i^0)$ generally differs from its non-smoothed counterpart, $\partial_{J \setminus \{i\}} c_{q-1}$. To make this difference explicit, we shall record the subset $K \subseteq J$ of smoothed crossings using the explicit index-set notation of the main text, $\Symb{K}{J} := \partial_J c_{|J|}^K$ (first perform oriented smoothings at all crossings in $K$, then take the unnormalized alternating sum over the $J$-face). 
    \begin{equation}
        \partial_J c_q = 2 \partial_{J \setminus \{i\}} c_{q-1} + \partial_{J \setminus \{i\}} c_{q-1}^{0}\quad \Longrightarrow \quad
       \Symb{\varnothing}{J} = 2 \Symb{\varnothing}{J \setminus \{i\}} + \Symb{\{i\}}{J \setminus \{i\}}.
    \end{equation}
    Clearly, the skein holds for any subset $K$ of smoothed crossings, hence
    \begin{equation}
    \label{eq:nsmoothings}
      \Symb{K}{J} = 2 \Symb{K}{J \setminus \{i\}} + \Symb{K \cup \{i\}}{J \setminus \{i\}}.
    \end{equation}
 

\end{proof}



\newif\ifJonesFlatness
\JonesFlatnessfalse   

\ifJonesFlatness

\section{Jones-flat cubes}
\label{sec:Jones-flat-cubes}

\begin{define}
    A variation cube all of whose vertices are unknots is said to be \textbf{Jones-flat}. Equivalently, every variation of the underlying shadow is an unknot. This is a property of the entire cube, not a property of a single diagram.
\end{define}

\begin{convention}[Set/cardinality-indexed symbols]
    \label{conv:setsymb}
    From now on we write $\Symb{K}{J}:=\partial_J c^K_{|J|}$ to record both the smoothed subset $K$ and the differenced subset $J$ explicitly. If only the cardinalities matter, we write $\cS_{r,n}$ for the class generated by all $\Symb{K}{J}$ with $|K|=r, |J|=n$, e.g.
    \[
    \cS_{r,n} = \partial_{[n]} c_n^r
    \]
    is the $n$-th symbol after $r$ oriented smoothings.
\end{convention}

\begin{lemma}[Jones-flatness]
\label{Jflat}
Denote by $\delta_n$ the Kronecker delta $\delta_{n0}$. A necessary and sufficient condition for Jones-flatness is the vanishing of all non-trivial symbols,
    \begin{equation}
        \cS_{0,n} = \partial_{[n]} c_n^0 = \delta_n.
    \end{equation}
\end{lemma}

\begin{proof}
    Let $\cS_{r,n} = \partial_{[n]} c_n^r$ be the $n$-th symbol after $r$ oriented smoothings.
    The proof is by induction on $r$ using the symbol recursion \eqref{eq:BLsymskein1}
    \[
    \cS_{r,n-1} = \cS_{r-1,n} - 2 \cS_{r-1, n-1},
    \]
    where $r$ and $n$ are the cardinalities of the smoothed and differentiated crossing sets (Convention~\ref{conv:setsymb}).
    As $\cS_{0,n} = \delta_n$, it follows that, $\cS_{r,n} = (-2)^r \delta_n$, and in particular, the free coefficient after $r$ smoothings, $c_0^r = (-2)^r$. Fix a vertex $K$ in the variation cube. Note that $r=|\cros(S)|$ is the number of crossings in any diagram obtained from the shadow. Therefore, by Corollary~\ref{cor:cp} (P) holds for all $k \leq |\cros(S)|$, rendering $K$ the unknot. As this holds for every vertex $K$, the cube is Jones-flat.

 The opposite direction, every variation is the unknot $\Rightarrow \, \cS_{0,n} = \delta_n$, follows immediately upon recognizing that all partial derivatives over the cube vanish except for the trivial $c_0 = 1$, from which the symbols follow. 
\end{proof}


\begin{define}[Jones–$m$–flat subcube]
\label{def:mflat}
Let $S$ be a shadow and $\cU\subseteq\cros(S)$. The $\cU$-subcube $\cQ_\cU$ is Jones–$m$–flat if $\partial_J c_{|J|} \equiv 0$ for every nonempty $J\subseteq \cU$ with $|J| \leq m-1$. Equivalently, the order-$< m$ symbols vanish, $\Symb{\varnothing}{J} \equiv 0$ for every nonempty $J \subseteq \cU$ with $|J| \leq m-1$. If this holds for all $m$, call $\cQ_\cU$ Jones-$\infty$-flat or \textit{flat} for short.
\end{define}

\begin{remark}
``Flat'' here means that all finite‑type symbols (alternating sums) below $|\cU|$ vanish on the $\cU$-subcube; only when $\cU$ is the full crossing set does Lemma~\ref{Jflat} implies that all vertices are unknots.
\end{remark}

\begin{corollary}[Symbolic binomial expansion]
\label{cor:pathind}
Let $K \subseteq I$ and $A \subseteq I \setminus K$. Let $J \subseteq A$. Then:
\begin{equation}
\label{eq:c45}
    \Symb{K \cup J}{A \setminus J}
    \;=\;
    \sum_{L \subseteq J} (-2)^{|L|}\, \Symb{K}{A \setminus L}.
\end{equation}
The symbol in the LHS depends only on the set $J$ and not on the order in which its elements are smoothed.
\end{corollary}
\begin{proof}
We prove \eqref{eq:c45} by induction on the finite set $J$ (equivalently, on $|J|$).

\paragraph{Base case: $J = \varnothing$.}
Then $K \cup J = K$ and $A \setminus J = A$, so the left side is $\Symb{K}{A}$.
On the right side, the only subset of $\varnothing$ is $\varnothing$, so
\[
    \sum_{L \subseteq \varnothing} (-2)^{|L|}\, \Symb{K}{A \setminus L}
    = (-2)^0\, \Symb{K}{A \setminus \varnothing}
    = \Symb{K}{A}.
\]
Thus \eqref{eq:c45} holds for $J = \varnothing$.

\paragraph{Inductive step.}
Assume \eqref{eq:c45} holds for all subsets of size $< |J|$. Let $J \neq \varnothing$, and pick an element $i \in J$. Define
\[
    J_0 := J \setminus \{i\}.
\]
Because $J \subseteq A$, we have $i \in A$. Also $i \notin J_0$, hence
\begin{equation}
\label{eq:tag1}
    i \in A \setminus J_0.
\end{equation}
\medskip
\noindent\emph{Step 1: A one-step smoothing rearrangement from Lemma~\ref{lem:BLsymskein}.}
Apply the skein recursion \eqref{eq:BLsymskein1} with smoothed set $K' := K \cup J_0$, differenced set $J' := A \setminus J_0$, and element $i \in J'$ (true by~\eqref{eq:tag1}). Then Lemma~\ref{lem:BLsymskein} gives:
\begin{equation}
\label{eq:tag2}
    \Symb{K \cup J_0}{A \setminus J_0}
    \;=\;
    2\, \Symb{K \cup J_0}{(A \setminus J_0) \setminus \{i\}}
    \;+\;
    \Symb{K \cup J_0 \cup \{i\}}{(A \setminus J_0) \setminus \{i\}}.
\end{equation}
But note the set identities:
\[
    (A \setminus J_0) \setminus \{i\} = A \setminus (J_0 \cup \{i\}) = A \setminus J, \qquad K \cup J_0 \cup \{i\} = K \cup J.
\]
So \eqref{eq:tag2} becomes:
\begin{equation}
\label{eq:tag3}
    \Symb{K \cup J_0}{A \setminus J_0}
    \;=\;
    2\, \Symb{K \cup J_0}{A \setminus J}
    \;+\;
    \Symb{K \cup J}{A \setminus J}.
\end{equation}
Rearrange \eqref{eq:tag3} to solve for the term we want:
\begin{equation}
\label{eq:tag4}
    \Symb{K \cup J}{A \setminus J}
    \;=\;
    \Symb{K \cup J_0}{A \setminus J_0}
    \;-\;
    2\, \Symb{K \cup J_0}{A \setminus J}.
\end{equation}
So far we have used only Lemma~\ref{lem:BLsymskein} and set algebra.
\medskip

\noindent\emph{Step 2: Expand each right-hand symbol by the induction hypothesis.}
We will apply the induction hypothesis \eqref{eq:c45} to each term on the right of \eqref{eq:tag4}, but with appropriate parameters.

\medskip
\noindent\textit{(a) Expand $\Symb{K \cup J_0}{A \setminus J_0}$.}
Here $J_0 \subseteq A$, so the induction hypothesis applies directly (since $|J_0| < |J|$):
\begin{equation}
\label{eq:tag5}
    \Symb{K \cup J_0}{A \setminus J_0}
    \;=\;
    \sum_{L \subseteq J_0} (-2)^{|L|}\, \Symb{K}{A \setminus L}.
\end{equation}

\medskip
\noindent\textit{(b) Expand $\Symb{K \cup J_0}{A \setminus J}$.}
To make the induction hypothesis fit this term, observe that
\[
    A \setminus J = (A \setminus \{i\}) \setminus J_0,
\]
and also $J_0 \subseteq A \setminus \{i\}$ (because $i \notin J_0$). So we can apply the induction hypothesis with $A' := A \setminus \{i\}$ and $J_0 \subseteq A'$:
\begin{equation}
\label{eq:tag6}
    \Symb{K \cup J_0}{(A \setminus \{i\}) \setminus J_0}
    \;=\;
    \sum_{L \subseteq J_0} (-2)^{|L|}\, \Symb{K}{(A \setminus \{i\}) \setminus L}.
\end{equation}
But the left side is exactly $\Symb{K \cup J_0}{A \setminus J}$, so:
\begin{equation}
\label{eq:tag7}
    \Symb{K \cup J_0}{A \setminus J}
    \;=\;
    \sum_{L \subseteq J_0} (-2)^{|L|}\, \Symb{K}{(A \setminus \{i\}) \setminus L}.
\end{equation}
Now simplify the set difference inside the symbol on the right. Because $i \notin L$ for $L \subseteq J_0$, we have the identity
\begin{equation}
\label{eq:tag8}
    (A \setminus \{i\}) \setminus L = A \setminus (L \cup \{i\}).
\end{equation}
So \eqref{eq:tag7} becomes:
\begin{equation}
\label{eq:tag9}
    \Symb{K \cup J_0}{A \setminus J}
    \;=\;
    \sum_{L \subseteq J_0} (-2)^{|L|}\, \Symb{K}{A \setminus (L \cup \{i\})}.
\end{equation}

\medskip
\noindent\emph{Step 3: Substitute back into \eqref{eq:tag4} and reindex as one sum over $L \subseteq J$.} Substitute~\eqref{eq:tag5} and \eqref{eq:tag9} into \eqref{eq:tag4}:
\begin{equation}
\label{eq:tag10}
\begin{aligned}
    \Symb{K \cup J}{A \setminus J}
    &= \left(\sum_{L \subseteq J_0} (-2)^{|L|}\, \Symb{K}{A \setminus L}\right)
    - 2\left(\sum_{L \subseteq J_0} (-2)^{|L|}\, \Symb{K}{A \setminus (L \cup \{i\})}\right) \\[4pt]
    &= \sum_{L \subseteq J_0} (-2)^{|L|}\, \Symb{K}{A \setminus L}
    \;+\; \sum_{L \subseteq J_0} (-2)^{|L|+1}\, \Symb{K}{A \setminus (L \cup \{i\})}.
\end{aligned}
\end{equation}

Now observe the standard subset decomposition: subsets of $J$ not containing $i$ are exactly the subsets of $J_0$; subsets of $J$ containing $i$ are exactly sets of the form $L \cup \{i\}$ with $L \subseteq J_0$. Moreover, the map $L \mapsto L \cup \{i\}$ is a bijection between $\mathcal{P}(J_0)$ and $\{M \subseteq J : i \in M\}$, and $|L \cup \{i\}| = |L| + 1$. So we can reindex the second sum in \eqref{eq:tag10} by $M := L \cup \{i\}$, giving:
\begin{equation}
\label{eq:tag11}
    \sum_{L \subseteq J_0} (-2)^{|L|+1}\, \Symb{K}{A \setminus (L \cup \{i\})}
    \;=\;
    \sum_{\substack{M \subseteq J \\ i \in M}} (-2)^{|M|}\, \Symb{K}{A \setminus M}.
\end{equation}
Similarly, the first sum is:
\begin{equation}
\label{eq:tag12}
    \sum_{L \subseteq J_0} (-2)^{|L|}\, \Symb{K}{A \setminus L}
    \;=\;
    \sum_{\substack{M \subseteq J \\ i \notin M}} (-2)^{|M|}\, \Symb{K}{A \setminus M}.
\end{equation}
Add \eqref{eq:tag11} and \eqref{eq:tag12}:
\[
    \Symb{K \cup J}{A \setminus J}
    \;=\;
    \sum_{M \subseteq J} (-2)^{|M|}\, \Symb{K}{A \setminus M}.
\]
This is exactly \eqref{eq:c45}. 
\end{proof}


\begin{remark}
Take $A=J$ and $K=\varnothing$ in Corollary~\ref{cor:pathind}.
In general, the RHS symbols $\Symb{\varnothing}{J \setminus L}$ in \eqref{eq:c45} depend on the set $L$, not only on its size. On a Jones-$m$-flat subcube the dependence disappears for $|J|=m-1$ because $\Symb{\varnothing}{J \setminus L} = 0$, for $0 \leq |L| < m-1$; this is the only specialization used in the proof of Theorem~\ref{thm:main}.
\end{remark}

\subsection{Last-row fingerprint and degree cap}
\label{sec:mfaceskein}

\begin{theorem}[Last‑row fingerprint of Jones‑$m$‑flat subcubes]
\label{thm:main}
Let $\cQ_\cU$ be Jones‑$m$-flat. For any prescribed set $J\subseteq \cU$ of $m-1$ crossings and any diagram $D$ carried by $S$, the free (order-$0$) coefficient satisfies the \textbf{last-row fingerprint}
$$
c_0(D\setminus J) = (-2)^{m-1}.
$$
Equivalently, for every such $J$, 
$$\Symb{J}{\varnothing}=(-2)^{m-1} \Symb{\varnothing}{\varnothing} =(-2)^{m-1}.$$ As this holds for every $J \subseteq \cU$ with $|J| = m-1$ we write $c_0^{m-1} = (-2)^{m-1}$.
By Corollary~\ref{cor:cp}, this last-row equality implies property (P) for \textit{every} $k \leq m-1$ (uniformly over all diagrams carried by the same shadow).
\end{theorem}

\begin{proof}
Fix $J\subseteq \cU$ with $|J|=m-1$. Apply Corollary~\ref{cor:pathind} with $A=J$ and $K=\varnothing$
\[
\Symb{J}{\varnothing} = \sum_{L\subseteq J}(-2)^{|L|} \, \Symb{\varnothing}{J \setminus L}.
\]
On a Jones-$m$-flat subcube we have exactly the vanishing
\[
\Symb{\varnothing}{J \setminus L}=0\quad\text{for }1\le |J \setminus L| \le m-1,
\]
If $L \neq J$, then $1 \le |J \setminus L| \le m - 1$, hence $\Symb{\varnothing}{J \setminus L} = 0$ by Jones-$m$-flatness. Thus, only $L = J$ contributes, $\Symb{\varnothing}{\varnothing} = c_0$
%
\[
\Symb{J}{\varnothing} = (-2)^{m-1}\,\Symb{\varnothing}{\varnothing} = (-2)^{m-1},
\]
i.e.\ $c_0(D\setminus J)=(-2)^{m-1}$. This is the last-row fingerprint.

\smallskip

\noindent
\emph{From last-row fingerprint to property (P)}. By Lemma~\ref{JVPfreec}, $c_0(L)=(-2)^{\ell(L)-1}$. Thus, $\Symb{J}{\varnothing}=(-2)^{m-1}$ means that after smoothing the prescribed $|J|=m-1$ crossings, the component count is $m$. Smoothing in any order can raise the component count by at most $1$ at each step; hitting the final value $m$ forces every intermediate increment to be exactly $+1$. Hence, for every $K\subseteq J$, $c_0(D\setminus K)=(-2)^{|K|}$, i.e. property (P) holds along $J$ (Corollary~\ref{cor:cp}).
\end{proof}

\section{Shadow Lemma}
\label{sec:shadow}

\begin{define}[Nugatory crossing; cut-vertex]
\label{def:nugatory-cut-vertex}
A crossing is nugatory iff there exists a Jordan curve in the projection plane meeting the diagram only at that crossing; at the shadow level, this is equivalent to the vertex being a cut-vertex (removing it disconnects the shadow).~\cite{tait1898, knotatlas}
\end{define}

\begin{lemma}[Shadow]
\label{shadow}
Let $S$ be a connected shadow with $n>0$ vertices and let
\[
\mathcal{D}(S) = \{\text{all }2^{n}\text{ diagrams carried by }S\}.
\]
For a diagram $D \in \mathcal{D}(S)$ write $c(D)$ for its set of crossings and, for $K \subseteq c(D)$, let $D \setminus K$ be the link obtained from $D$ after \emph{smoothing} every crossing in $K$.
\\[0.5ex]
The following statements are equivalent:
\begin{enumerate}
\item \textbf{Component-increment property (P)}. For every $D \in \mathcal{D}(S)$ and every subset $K \subseteq c(D)$ of size $k$,
\begin{equation}\tag{P}
\#\text{components of }(D \setminus K) = 1 + k.
\end{equation}

\item Every vertex of $S$ is a \emph{cut-vertex} (equivalently, every crossing is \emph{nugatory} in Tait's sense). Hence, $S$ is the planar embedding of a tree in which each edge appears exactly once.

\item All diagrams in $\mathcal{D}(S)$ are diagrams of the unknot (or of a split unlink if $S$ is disconnected).
\end{enumerate}
Consequently, any diagram satisfying (P)---and every diagram obtained from it by crossing flips---represents the unknot.
\end{lemma}

\begin{proof} $ $\\
\textbf{(1) $\Rightarrow$ (2)} 
%
%
%
Work with oriented diagrams and fix a vertex $x$ of the connected shadow $S$. By Corollary~\ref{cor:components}, for a fixed set of smoothings the number of components of $D\setminus J$ is independent of the over/under choices in $D\in\mathcal D(S)$ (the component count after smoothing a fixed set is a shadow invariant); thus we may argue at the level of the shadow. Property $(P)$ with $|K|=1$ says that for every diagram $D$ the oriented smoothing at $x$ yields two components $D\setminus\{x\}=L_A\sqcup L_B$. Suppose $x$ were not a cut-vertex. Then the plane graph $S\setminus\{x\}$ is connected; hence there exists a path in $S\setminus\{x\}$ from an arc of $L_A$ to an arc of $L_B$, and along that path choose the first vertex $y$ where the two sides meet. In the diagram $D\setminus\{x\}$ the crossing $y$ is a crossing between different components $L_A$ and $L_B$. Now smooth $y$ as well. Oriented smoothing at a crossing that connects distinct components does not increase the number of components (it either keeps it the same or reduces it by one). Therefore
\[
\#\pi_0\big(D\setminus\{x,y\}\big)\ \le\ 2,
\]
contradicting Property $(P)$ for $|K|=2$, which requires $1+2=3$ components. Hence, our assumption was false: $S\setminus\{x\}$ must be disconnected, i.e.\ $x$ is a cut-vertex (nugatory)~\cite{tait1898}.

\noindent
\textbf{(2) $\Rightarrow$ (3)} \quad Theorem~3 in \cite{medina2017}: \emph{A shadow $S$ has only unknot diagrams iff every vertex of $S$ is a cut-vertex.}
Assumption (2) therefore forces every diagram in $\mathcal{D}(S)$ to be an unknot diagram.\\[1ex]

\noindent
 \textbf{(3) $\Rightarrow$ (1)} By Theorem~3 of \cite{medina2017}, if every diagram carried by the connected shadow $S$ is an unknot, then every vertex of $S$ is a cut-vertex. Hence, $(3)$ implies $(2)$. By the direct implication $(2) \Rightarrow (1)$ proved below, for every diagram $D \in D(S)$ and every subset $K \subseteq c(D)$, smoothing the crossings in $K$ increases the number of components by exactly $|K|$. Therefore
\[
    \#\pi_0(D \setminus K) = 1 + |K|,
\]
i.e.\ property~(P) holds.\\[1ex]
 



\noindent
\textbf{(2) $\Leftrightarrow$ (1)} (direct argument)
\begin{itemize}
\item \textbf{(2) $\Rightarrow$ (1)} If every crossing is nugatory, smoothing one crossing splits one component into two. Repeating the argument shows the component count always increases by one.
\item \textbf{(1) $\Rightarrow$ (2)} The single-crossing case $k=1$ recovers the nugatory condition used in the first implication.
\end{itemize}

\noindent
Thus all three statements are equivalent.
\end{proof}

\begin{corollary}[``Tree diagrams'']
A connected knot diagram satisfies property (P) \emph{iff} its shadow is a tree. Such a diagram is sometimes called a \emph{descending} or \emph{standard} diagram; it can be reduced to a round circle by $n$ Reidemeister I moves (one per nugatory crossing) followed by $n$ paired Reidemeister II moves along the tree edges.
\end{corollary}

\fi

\ifJonesFlatness

\subsection{Locality of Jones-$m$-flatness}

\begin{lemma}[Global vs.\ local Jones-$m$-flatness]
\label{lem:local-m-flatness}
Let $S$ be a connected knot shadow with crossing set $I = \cros(S)$, and let $\cQ(S)$ be the full variation cube on $I$.

\begin{enumerate}
\item 
The full cube $\cQ(S)$ is always Jones-2-flat.

\item Moreover, if $\cQ(S)$ is Jones-$m$-flat for some $m \ge 3$, then $S$ is a tree and every diagram carried by $S$ is an unknot; in particular $\cQ(S)$ is isotopically constant. Equivalently, for a non-tree shadow the maximal index of global Jones-flatness is $m=2$.
\end{enumerate}



\end{lemma}

\begin{proof}


First note that for knots we always have $c_1(K) = 0$ for every knot $K$.
Thus on any knot shadow $S$ the first Jones layer $c_1$ is identically zero on the full cube $\cQ(S)$, so for every singleton $J = \{i\}\subset I$,
\[
\Symb{\varnothing}{J} = \partial_J c_1 \equiv 0.
\]
By Definition~\ref{def:mflat} this is exactly Jones-2-flatness.

Now suppose that the full cube $\cQ(S)$ is Jones--$m$-flat for some $m \ge 3$. By definition this means
\[
\Symb{\varnothing}{J} = \partial_J c_{|J|} \equiv 0
\quad\text{for every nonempty } J\subseteq I \text{ with } |J|\le m-1.
\]
In particular, all symbols of order 1 and 2 vanish:
\[
\Symb{\varnothing}{\{i\}} = 0,\quad
\Symb{\varnothing}{\{i,j\}} = 0
\quad\text{for all } i,j\in I.
\]

Apply the last-row fingerprint (Theorem~\ref{thm:main}) with $m=3$ to any 2-element set $J=\{i,j\}\subseteq I$. Since $\cQ(S)$ is Jones-3-flat, Theorem~\ref{thm:main} gives
\[
c_0(D\setminus\{i,j\}) = (-2)^2 = 4
\quad\text{for all diagrams }D\in \mathcal{D}(S).
\]
By Corollary~\ref{cor:cp} this implies property $P$ along each pair $\{i,j\}$: for every $K\subseteq \{i,j\}$,
\[
c_0(D\setminus K) = (-2)^{|K|}
\quad\Longleftrightarrow\quad
\#\pi_0(D\setminus K) = 1 + |K|.
\]
In particular:
\begin{itemize}
\item for every singleton $\{i\}$ we can choose any $j\neq i$ and use the pair $\{i,j\}$ to see that
\[
\#\pi_0(D\setminus\{i\}) = 2,
\]
\item for every pair $\{i,j\}$ we have
\[
\#\pi_0(D\setminus\{i,j\}) = 3.
\]
\end{itemize}
Thus property $P$ holds for all subsets $K\subseteq I$ of size $|K|=1$ and $|K|=2$.
Note that in the proof of the Shadow Lemma~\ref{shadow}, $P \Rightarrow$ every vertex is a cut-vertex, only the cases $|K|=1$ and $|K|=2$ are used:
\begin{itemize}
\item For $|K|=1$: smoothing $x$ yields exactly 2 components.
\item If $x$ were not a cut-vertex, there is a path in the shadow from one component to the other in $S\setminus\{x\}$; taking the first vertex $y$ on such a path and smoothing both $x$ and $y$ gives at most 2 components, contradicting the $|K|=2$ case of $P$, which demands 3 components.
\end{itemize}
Since we have property $P$ for all singletons and pairs, this argument applies and shows that every crossing of $S$ is a cut-vertex, so $S$ is a tree. By the Medina-Ramírez-Salazar criterion \cite{medina2017}, a shadow all of whose vertices are cut-vertices carries only unknot diagrams; hence every vertex of $\cQ(S)$ is an unknot and the full cube is isotopically constant.

The contrapositive says: if $S$ is not a tree (so the cube is not isotopically constant), then $\cQ(S)$ cannot be Jones-$m$-flat for any $m\ge 3$. Combined with the universal Jones-2-flatness, this means that on non-tree shadows the maximal global Jones-flatness index is exactly $m=2$.

\end{proof}

\fi

\newif\ifshowconverselastrow
\showconverselastrowfalse   

\ifshowconverselastrow

\section{Converse last-row fingerprint lemma}

\begin{lemma}[Converse last-row fingerprint]
\label{thm:main-converse}
Let \(S\) be a shadow and \(U\subseteq \cros(S)\). Fix \(m\ge 2\) with \(|U|\ge m-1\). Assume that for every \((m-1)\)-subset \(J\subseteq U\) one has the last-row component fingerprint
\[
c_0(D\setminus J)=(-2)^{m-1}
\]
for some (hence every) diagram \(D\in \mathcal{D}(S)\). (Equivalently, \(\Symb{J}{\varnothing}=(-2)^{m-1}\Symb{\varnothing}{\varnothing}\).)
Then the \(U\)-subcube \(\cQ_U\) is Jones-\(m\)-flat, i.e.
\[
\Symb{\varnothing}{K}=\partial_K c_{|K|}\equiv 0\qquad\text{for every nonempty }K\subseteq U,\ |K|\le m-1.
\]
So, combined with Theorem~\ref{thm:main}, Jones-\(m\)-flatness is equivalent to the \((m-1)\)-row fingerprint.
\end{lemma}

\begin{proof}
~\\
\begin{enumerate}
\item \emph{Extend the last-row data to all smaller smoothings.}
Fix \(K\subseteq U\) with \(1\le |K|\le m-1\). Choose \(J\subseteq U\) with \(|J|=m-1\) and \(K\subseteq J\).
By the hypothesis \(c_0(D\setminus J)=(-2)^{m-1}\), Corollary~\ref{cor:cp} (\(c_0\) fingerprint \(\Rightarrow\) property \(P\) along subsets) gives
\[
c_0(D\setminus K)=(-2)^{|K|}.
\]
Thus, for every such \(K\),
\[
\Symb{K}{\varnothing}=c_0(D\setminus K)=(-2)^{|K|}\Symb{\varnothing}{\varnothing}.
\]

\item \emph{Use the smoothing-operator identity and induct up the cube.}
Corollary~\ref{cor:pathind} gives the operator formula
\[
h_K=\prod_{i\in K}h_i=\sum_{L\subseteq K}(-2)^{|L|}v_L,
\]
and applying this to the top-row symbol \(\Symb{\varnothing}{K}\) yields
\begin{equation}
\label{eq:symb-rec}
\Symb{K}{\varnothing}
= h_K\big(\Symb{\varnothing}{K}\big)
=\sum_{L\subseteq K}(-2)^{|L|}\,v_L(\Symb{\varnothing}{K})
=\sum_{L\subseteq K}(-2)^{|L|}\,\Symb{\varnothing}{K\setminus L}.
\end{equation}
\end{enumerate}

Now induct on \(r=|K|\).
\begin{itemize}
\item Base \(r=1\): \(K=\{i\}\). Then \eqref{eq:symb-rec} reads
\[
\Symb{\{i\}}{\varnothing}=\Symb{\varnothing}{\{i\}}+(-2)\Symb{\varnothing}{\varnothing}.
\]
But Step 1 says \(\Symb{\{i\}}{\varnothing}=(-2)\Symb{\varnothing}{\varnothing}\). Hence, \(\Symb{\varnothing}{\{i\}}=0\).

\item Step \(r\to r+1\): assume \(\Symb{\varnothing}{K'}=0\) for all nonempty \(K'\subset U\) with \(|K'|<r\).
For \(|K|=r\), every term \(\Symb{\varnothing}{K\setminus L}\) with \(L\neq \varnothing, K\) has \(1\le |K\setminus L|\le r-1\), hence vanishes by induction. So \eqref{eq:symb-rec} collapses to
\[
\Symb{K}{\varnothing}=\Symb{\varnothing}{K}+(-2)^{r}\Symb{\varnothing}{\varnothing}.
\]
Using Step 1 again, \(\Symb{K}{\varnothing}=(-2)^r\Symb{\varnothing}{\varnothing}\), so \(\Symb{\varnothing}{K}=0\).
\end{itemize}

Therefore, \(\Symb{\varnothing}{K}\equiv 0\) for every nonempty \(K\subseteq U\) with \(|K|\le m-1\), i.e. \(\cQ_U\) is Jones-\(m\)-flat (Definition~\ref{def:mflat}).

\end{proof}

\fi

\section*{Acknowledgements}
This work was supported by the European Union’s Horizon Europe research and innovation programme under grant agreement No. 101178170.

\end{document}